\documentclass[reqno]{amsart}

\usepackage{amsmath}
\usepackage{amsfonts}
\usepackage{pdfsync}
\usepackage{color}
\usepackage{graphicx}
\usepackage{amssymb}
\usepackage{epsf}
\usepackage{eucal}
\usepackage{amscd}
\usepackage{graphics}
\usepackage{amsthm}
\usepackage{fancyhdr}
\usepackage{delarray}
\usepackage{fullpage}
\usepackage{paralist}
\usepackage[all]{xy}
\usepackage{tikz}
\usepackage{verbatim}
\usepackage{cite}
\usepackage{mathrsfs}
\usepackage{bm}

\usepackage{hyperref}

\usepackage[shortlabels]{enumitem}
\usepackage[capitalise]{cleveref}

\numberwithin{equation}{section}

\definecolor{darkgreen}{cmyk}{1,0,1,.2}
\definecolor{m}{rgb}{1,0.1,1}

\newdimen\theight
\def\TeXref#1{%
             \leavevmode\vadjust{\setbox0=\hbox{{\tt
                     \quad\quad  {\small \textrm #1}}}%
             \theight=\ht0
             \advance\theight by \lineskip
             \kern -\theight \vbox to
             \theight{\rightline{\rlap{\box0}}%
             \vss}%
           }}%

\makeatletter
\def\moverlay{\mathpalette\mov@rlay}
\def\mov@rlay#1#2{\leavevmode\vtop{%
   \baselineskip\z@skip \lineskiplimit-\maxdimen
   \ialign{\hfil$\m@th#1##$\hfil\cr#2\crcr}}}
\newcommand{\charfusion}[3][\mathord]{
    #1{\ifx#1\mathop\vphantom{#2}\fi
        \mathpalette\mov@rlay{#2\cr#3}
    }
    \ifx#1\mathop\expandafter\displaylimits\fi}
\makeatother

\newcommand{\cupdot}{\charfusion[\mathbin]{\cup}{\cdot}}
\newcommand{\bigcupdot}{\charfusion[\mathop]{\bigcup}{\cdot}}

\newcounter{nameOfYourChoice}
\newcounter{nameOfYourChoice2}

\renewcommand{\epsilon}{\varepsilon}

\DeclareMathOperator{\dom}{dom}
\DeclareMathOperator{\im}{im}

\DeclareMathOperator{\id}{id}

\DeclareMathOperator{\Pen}{Pen}
\DeclareMathOperator{\CPen}{CPen}

\DeclareMathOperator{\Int}{Int}
\DeclareMathOperator{\ev}{ev}

\DeclareMathOperator{\Aut}{Aut}

\DeclareMathOperator{\rep}{rep}
\DeclareMathOperator{\Ch}{\operatorname{Ch}}
\DeclareMathOperator{\Pa}{\operatorname{Pa}}

\newcommand{\FF}{\mathcal{F}}

\newcommand{\GG}{\mathcal{G}}

\newcommand{\MM}{\mathcal{M}}

\newcommand{\OO}{\mathcal{O}}
\newcommand{\TT}{\mathcal{T}}

\newcommand{\Z}{\mathbb{Z}}

\newcommand{\R}{\mathbb{R}}
\newcommand{\N}{\mathbb{N}}

\newcommand{\bfH}{\mathbf{H}}

\newcommand{\ffmnz}{\mathfrak{h}^m_{n,z}}

\newcommand{\ffmlz}{\mathfrak{h}^m_{l,z}}
\newcommand{\zznn}{\mathfrak{X}^n_n}
\newcommand{\zzmn}{\mathfrak{X}^m_n}

\newcommand{\zzlk}{\mathfrak{X}^l_k}
\newcommand{\zzml}{\mathfrak{X}^m_l}

\newcommand{\zzln}{\mathfrak{X}^l_n}

\newcommand{\zzlpin}{\mathfrak{X}_n^{l'}}

\newcommand{\zzsnn}{\mathfrak{X}^{n+1}_n}

\newcommand{\ppmn}{\mathfrak{P}^m_n}

\newcommand{\ppmyn}{\mathfrak{P}^m_{n-1}}
\newcommand{\ppmone}{\mathfrak{P}^{m}_{-1}}

\newcommand{\olppmn}{\overline{\mathfrak{P}}^m_n}
\newcommand{\olppmyn}{\overline{\mathfrak{P}}^m_{n-1}}

\newcommand{\xn}{X_n}
\newcommand{\xyn}{X_{n-1}}
\newcommand{\xm}{X_m}

\newcommand{\xset}{X}

\newcommand{\fset}{F}

\newcommand{\Zmiyn}{Z_{n-1}^-}

\newcommand{\Zpmyn}{Z_{n-1}^\pm}

\newcommand{\cnyn}{C_{n,n-1}}

\newcommand{\csmm}{C_{m+1,m}}
\newcommand{\cnm}{C_{n,m}}
\newcommand{\cnsm}{C_{n,m+1}}

\newcommand{\czeroone}{C_{0,-1}}
\newcommand{\cnone}{C_{n,-1}}

\newcommand{\olcnyn}{\overline{C}_{n,n-1}}

\newcommand{\olcmym}{\overline{C}_{m,m-1}}

\newcommand{\sphere}{S}

\newcommand{\corona}{C}

\newcommand{\aset}{A}

\newcommand{\yset}{Y}

\newcommand{\xxn}{\mathfrak{X}_n}
\newcommand{\xxyn}{\mathfrak{X}_{n-1}}

\newcommand{\xxm}{\mathfrak{X}_m}
\newcommand{\xxl}{\mathfrak{X}_l}

\newcommand{\Rn}{\mathfrak{P}_n}
\newcommand{\Ryn}{\mathfrak{P}_{n-1}}

\newcommand{\Rone}{\mathfrak{P}_{-1}}

\newcommand{\rrn}{\mathfrak{r}_n}
\newcommand{\rrm}{\mathfrak{r}_m}
\newcommand{\rrl}{\mathfrak{r}_l}
\newcommand{\ssn}{\mathfrak{s}_n}

\newcommand{\ttn}{\mathfrak{t}_n}

\newcommand{\Yn}{Y_n}

\newcommand{\Ynn}{Y^n_n}
\newcommand{\Ymn}{Y^m_n}
\newcommand{\Yln}{Y^l_n}
\newcommand{\Yymn}{Y^{m-1}_n}

\newcommand{\wtYmn}{\widetilde{Y}^m_n}

\newcommand{\hhmx}{\mathfrak{h}_{m,x}}

\theoremstyle{plain}

\newtheorem{thm}{Theorem}[section]
\newtheorem{lem}[thm]{Lemma}
\newtheorem{cor}[thm]{Corollary}
\newtheorem{prop}[thm]{Proposition}

\theoremstyle{definition}

\newtheorem{defn}[thm]{Definition}

\theoremstyle{remark}

\newtheorem{rem}{Remark}

\crefname{thm}{theorem}{theorems}
\crefname{lem}{lemma}{lemmas}
\crefname{cor}{corollary}{corollaries}
\crefname{prop}{proposition}{propositions}

\crefname{defn}{definition}{definitions}
\crefname{conj}{conjecture}{conjectures}
\crefname{ex}{example}{examples}
\crefname{exs}{examples}{examples}
\crefname{prob}{problem}{problems}
\crefname{quest}{question}{questions}

\crefname{rem}{remark}{remarks}
\crefname{rems}{remarks}{remarks}
\crefname{claim}{claim}{claims}
\crefname{case}{case}{cases}
\crefname{hyp}{hypothesis}{hypotheses}
\crefname{notation}{notation}{notations}

\title{Limit aperiodic and repetitive colorings of graphs}

\author[J.A. \'Alvarez L\'opez]{Jes\'us A. \'Alvarez L\'opez}
\address{Department/Institute of Mathematics\\
         University of Santiago de Compostela\\
         15782 Santiago de Compostela\\
         Spain}
\email{jesus.alvarez@usc.es}
  
\author[R. Barral Lij\'o]{Ram\'on Barral Lij\'o}
\address{Research Organization of Science and Technology\\
         Ritsumeikan University\\
         Nojihigashi 1-1-1\\
         Kusatsu\\
         Shiga\\
         525-8577\\
         Japan}
\email{ramonbarrallijo@gmail.com}


\thanks{The authors were partially supported by the grants FEDER/Ministerio de Ciencia, Innovaci\'on y Universidades/AEI/MTM2017-89686-P, and Xunta de Galicia/ED431C 2019/10. The second author was also supported by a Canon Foundation in Europe Research Grant.}


\date{}

\begin{document}

\maketitle

\begin{abstract}
Let $X$ be a (repetitive) infinite connected simple graph with a finite upper bound $\Delta$ on the vertex degrees. The main theorem states that $X$ admits a (repetitive) limit aperiodic vertex coloring by $\Delta$ colors. This refines a theorem for finite graphs proved by Collins and Trenk, and by Klav\v{z}ar, Wong and Zhu, independently. It is also related to a theorem of Gao, Jackson and Seward stating that any countable group has a strongly aperiodic coloring by two colors, and to recent research on distinguishing number of graphs by Lehner, Pil\'{s}niak and Stawiski, and by H\"{u}ning et al. In our theorem, the number of colors is optimal for general graphs of bounded degree. We derive similar results for edge colorings, and for more general graphs, as well as a construction of limit aperiodic and repetitive tilings by finitely many prototiles. In a subsequent paper, this result is also used to improve the construction of compact foliated spaces with a prescribed leaf.
\end{abstract} 

\tableofcontents

\section{Introduction}\label{s: intro}

\subsection{Estimates of the distinguishing number}\label{ss: distinguishing number}

Let $\xset\equiv(\xset,E)$ be a simple (undirected countable) graph (with finite vertex degrees). Assume that $X$ is connected and consider its natural distance. The \emph{degree} of $X$, denoted by $\deg X$, is the supremum of its vertex degrees. 

Consider a (vertex) coloring $\phi\colon \xset \to \fset$ (the set of colors $\fset$ is usually assumed to be a subset of $\N$). It is said that $\phi$ (or $(X,\phi)$) is \emph{aperiodic} or \emph{distinguishing} if there is no nontrivial automorphism of $(\xset,\phi)$. The \emph{distinguishing number} of $X$ is
\[
D(\xset)=\min\{\,n\in\Z^+\mid\text{$\xset$ has some aperiodic coloring by $n$ colors}\,\}\;. 
\]
This concept was introduced by Albertson and Collins \cite{AlbertsonCollins1996}, and the calculation of $D(\xset)$ (or bounds thereof) for many families of graphs has been the subject of much research in recent years (see e.g.\ \cite{ImrichJerebicKlavzar2008,ImmelWenger2017}). This resulted in the following sharp estimate for finite graphs, where $K_n$, $K_{n,n}$ and $C_n$ denote the complete graph on $n$ vertices, the $(n,n)$-bipartite graph and the cyclic graph with $n$ vertices, respectively.

\begin{thm}[Collins-Trenk \cite{CollinsTrenk2006}, Klav\v zar-Wong-Zhu \cite{KlavzarWongZhu2006}]\label{t: finite case}
If $X$ is a finite connected simple graph different from $K_n$, $K_{n,n}$, and $C_5$ {\rm(}$n\ge2$\/{\rm)}, then $D(\xset)\le\deg X$. If $X$ is $K_n$, $K_{n,n}$ or $C_5$ {\rm(}$n\ge2$\/{\rm)}, then $D(\xset)=\deg X+1$.
\end{thm}

For infinite graphs, the following result has been recently proved. It is easy to check that the bound it provides is sharp.

\begin{thm}[Lehner-Pil\'{s}niak-Stawiski \cite{LehnerPilsniakStawiski}]\label{t: infinite case}
Let $X$ be an infinite connected simple graph with $\deg X\geq 3$. Then $D(X)\leq \deg X -1$.
\end{thm}

It is clear that $D(X)=2$ if $X$ is an infinite graph of degree two. Also very recently, H\"{u}ning et al.\ have provided a complete classification of all connected graphs $X$ with $\deg X =D(X)=3$ \cite{Huningetal}.

\subsection{Space of colored graphs}\label{ss: widehat GG_*}

Consider pointed connected colored simple graphs, $(\xset,x,\phi)$, with colors in $\N$. Their isomorphism classes, $[\xset,x,\phi]$, form a Polish space $\widehat\GG_*$ with a canonical topology (\Cref{ss: colorings}). For any such graph $(\xset,\phi)$, there is a canonical map $\hat\iota_{\xset,\phi}:\xset\to\widehat\GG_*$ defined by $\hat\iota_{\xset,\phi}(x)=[\xset,x,\phi]$. The images of the maps $\hat\iota_{\xset,\phi}$, denoted by $[X,\phi]$, form a canonical partition of $\widehat\GG_*$. We have $[X,\phi]\equiv\Aut(X,\phi)\backslash X$, where $\Aut(X,\phi)$ is the group of color-preserving automorphisms of $(X,\phi)$. Every closure $\overline{[X,\phi]}$ is saturated. It is said that $(X,\phi)$ is:
\begin{description}

\item[aperiodic] when $\Aut(X,\phi)=\{\id_X\}$ ($\hat\iota_{X,\phi}$ is injective);

\item[limit aperiodic] when $(Y,\psi)$ is aperiodic  for all $[Y,\psi,y]\in\overline{[X,\phi]}$; and,

\item[repetitive] if, roughly speaking, every colored disk of $(X,\phi)$ is repeated uniformly in $X$  (\Cref{ss: colorings}).

\end{description}
The closure $\overline{[X,\phi]}$ is compact if and only if $\deg X,|\im\phi|<\infty$ (\Cref{p: compact}). Moreover $\overline{[X,\phi]}$ is minimal if $(X,\phi)$ is repetitive, and the reciprocal holds when $\overline{[X,\phi]}$ is compact.

By forgetting the colorings $\phi$, we get a Polish space $\GG_*$, with a partition defined by the images of maps $\iota_X:X\to\GG_*$, obtaining obvious versions without colorings of the above properties.

\subsection{Strongly aperiodic colorings of groups}\label{ss: colorings of groups}

Let $G$ be a countable group, and $F$ a finite set equipped with the discrete topology. Then the $F$-valued colors on $G$ form the compact second countable space $F^G$, which has a canonical left action of $G$ defined by $(g\cdot\phi)(h)=\phi(g^{-1}h)$. This $G$-space is called a \emph{shift}, and any non-empty $G$-invariant closed subset of $F^G$ is called a \emph{subshift}. In particular, the orbit closure $\overline{G\cdot\phi}$ of any $\phi\in F^G$ is a subshift. If the action of $G$ on $\overline{G\cdot\phi}$ is free (respectively, minimal), then $\phi$ is said to be \emph{strongly aperiodic} (respectively, \emph{strongly repetitive}). The existence of such colorings is guaranteed by the following sharp result.

\begin{thm}[Gao-Jackson-Seward \cite{GaoJacksonSeward2009}; see also \cite{AubrunBarbieriThomasse2019}]\label{t: group colorings}
Every countable group admits a strongly aperiodic and strongly repetitive coloring by $2$ colors.
\end{thm}

Indeed, the original statement in \cite{GaoJacksonSeward2009} only gives strong aperiodicity, but then strong repetitivity follows immediately with the following short argument. The existence of a strongly aperiodic coloring on $G$ means that $G$ acts freely on some subshift $X\subset\{0,1\}^G$. Then there is a minimal subset $Y\subset X$, and any coloring in $Y$ is strongly aperiodic and strongly repetitive. 

Suppose from now on that $G$ is finitely generated, and let $S$ be a minimal set of generators such that all elements of $S\cap S^{-1}$ are of order two. Consider the (left-invariant) Cayley graph defined by $S$, also denoted by $G$, where the degree of every vertex is $|S|$. Up to isomorphisms, the only possible limit of the graph $G$ is $G$. Thus $F^G$ is closed by taking limits of colors in the sense of \Cref{ss: widehat GG_*}.  But, in this setting, it is natural to modify the definition of a limit of a coloring $\phi\in F^G$ by using only graph isomorphisms between disks given by left translations of $G$. The ``limits by left translations'' obtained in this way are just the elements of $\overline{G\cdot\phi}$, and the corresponding notion of ``limit aperiodicity by left translations" means strong aperiodicity. Similarly, we can also define ``repetitivity by left translations,'' which turns out to be strong repetitivity. By definition, limit aperiodicity is stronger than ``limit aperiodicity by left translations'' (strong aperiodicity), whereas repetitivity is weaker than ``repetitivity by left translations'' (strong repetitivity).

The Cayley graph of $G$ induced by $S$ is also equipped with a $G$-invariant edge coloring $\psi_0$ by colors in $S$, assigning to an edge between vertices $a,b\in G$ the unique element $s\in S$ satisfying $as^{\pm1}=b$. Moreover, if the order of $s$ is not $2$, then the choice of $\pm1$ in the above exponent defines an orientation of the edge. This defines a canonical partial $G$-invariant direction $\OO_0$ of $G$. The left translations are just the graph isomorphisms of $G$ that preserve $\psi_0$ and $\OO_0$. Consider the obvious extensions of the concepts of limit aperiodicity and repetitivity to triples $(\phi,\psi,\OO)$, where $\phi$ is a vertex coloring, $\psi$ an edge coloring and $\OO$ a partial direction. Then  a coloring $\phi\in\{0,1\}^G$ is strongly aperiodic (respectively, strongly repetitive) if and only if $(\phi,\psi_0,\OO_0)$ is limit aperiodic (respectively, repetitive). Thus, in this case, \Cref{t: group colorings} can be restated by saying that $G$ admits a coloring $\phi\in\{0,1\}^G$ such that $(\phi,\psi_0,\OO_0)$ is limit aperiodic and repetitive.

\subsection{Main theorem}\label{ss: main thm}

The distinguishing number can be refined as follows. The \emph{limit distinguishing number} of $X$ is
\[
D_L(\xset)=\inf\{\,n\in \Z^+\mid\text{$X$ has a limit aperiodic coloring by $n$ colors}\,\}\;.
\]
When $X$ is repetitive, its \emph{repetitive limit distinguishing number} is
\[
D_{RL}(\xset)=\inf\{\,n\in \Z^+\mid\text{$X$ has a repetitive limit aperiodic coloring by $n$ colors}\,\}\;.
\]
It only makes sense to consider these concepts when $X$ is infinite because, if $X$ is finite, then limit aperiodicity means aperiodicity, and repetitivity always holds, obtaining $D_{RL}(\xset)=D_L(\xset)=D(\xset)$. Our main result is the following estimate of $D_L(\xset)$ and $D_{RL}(\xset)$, which can be considered as a refined version of \Cref{t: finite case}.

\begin{thm}\label{t: main}
If $X$ is an infinite connected simple graph, then $D_L(\xset)\le\deg X$. If moreover $\xset$ is repetitive, then $D_{RL}(\xset)\le\deg X$. 
\end{thm}

With this generality, the estimates of \Cref{t: main} are sharp, as shown by the Cayley graph of $\Z$ (defined with the generating set $\{1\}$). For $\deg X\ge3$, the estimates of \Cref{t: main} might not be optimal, according to \Cref{t: infinite case}. In this case, an obvious approach to get the  optimal estimate would be to try to somehow incorporate the idea of the proof of \Cref{t: infinite case} in \cite{LehnerPilsniakStawiski} into our techniques. However this may be difficult because we divide $X$ into finite pieces and work ``locally", whereas they make essential use of an infinite geodesic ray. In any case, like in \Cref{t: infinite case}, it is obvious that the optimal estimates in \Cref{t: main} are at least $\deg X-1$ if $\deg X\ge3$.

\Cref{t: main} will be derived from \Cref{t: finitary}, which is actually stronger in the following sense. The conditions of being limit aperiodic and repetitive can be restated quantitatively, involving some choice of constants. We prove that these constants can be taken to depend only  on $\deg X$ and not on the particular choice of $X$, which does not follow from \Cref{t: main}. The precise statement of this dependence can be found in \Cref{t: finitary}. The same can be said for finite graphs, where the analogue of \Cref{t: finitary} would give a quantitative result stronger than \Cref{t: finite case}.

In \Cref{t: main}, the minimality does not follow directly from the limit aperiodicity, like in \Cref{t: group colorings}, because $\overline{[X]}$ may contain elements $[Y,y]$ with $Y\not\cong X$.

In the case of a group $G$ finitely generated by $S$ (\Cref{ss: colorings of groups}), \Cref{t: main} states that $G$ has a repetitive limit aperiodic vertex coloring by $|S|$ colors. Since the total number of colors of $(\phi,\psi_0,\OO_0)$ is $2+|S|$, without taking into account the additional values of $\OO_0$, it can be said that \Cref{t: main} somehow improves \Cref{t: group colorings} in this case.

\subsection{An idea of the proof}\label{s. idea}

We have to prove that, if $\deg X<\infty$, then $X$ has a limit aperiodic coloring $\phi$ by $\deg X$ colors, which is repetitive if $X$ is repetitive. 

First, we divide the graph $X=X_{-1}$ into finite connected clusters of uniformly bounded size, such that their centers form a Delone set $X_0\subset X_{-1}$. 
Moreover $X_0$ can be endowed with a connected graph structure with $\deg X_0<\infty$.
On every cluster with center $x\in X_0$, the method of the proof of \Cref{t: finite case} is used to construct a large enough amount of different colorings $\psi_{0,x}^i$ by $\deg X$ colors breaking its symmetry.
Any assignment of such colorings, $x\mapsto\psi_{0,x}^i$, is considered as a coloring, $x\mapsto i$, of $X_0$.
For these colorings of $X_0$, we have enough avaliable colors to be able to proceed in the same way. 
Thus $X_0$ is divided into clusters, defining a graph $X_1\subset X_0$. The above type of colorings of $X_0$ are considered in the new clusters. Again, for every $x\in X_1$, we can break the symmetry of the corresponding cluster with a large enough amount of different colors $\psi_{1,x}^i$ of the above kind. 
Any assignment of such colorings, $x\mapsto\psi_{1,x}^i$, is considered a coloring, $x\mapsto i$, of $X_1$.
This process is continued indefinitely, producing a sequence of graphs $X_n$, divided into clusters whose centers form $X_{n+1}$, and colorings $\psi_{n+1,x}^i$ breaking the symmetry in the cluster of $X_n$ with center $x\in X_{n+1}$. We use these data for $0\leq n \leq N$ to define a coloring $\phi^N$ preventing isomorphisms between disks centered at points within a certain distance; namely, given any $\epsilon\in\Z^+$, there is some $N,\delta\in\Z^+$ such that
\begin{equation}\label{epsilon, delta}
0<d(x,y)<\epsilon\Longrightarrow[D (x,\delta),x,\phi^N]\neq[D (y,\delta),y,\phi^N]
\end{equation}
for all $x,y\in X$. By taking a subsequence if necessary, we can assume that the sequence $\phi^N$ is eventually constant on finite sets, converging in this sense to a coloring $\phi$. This coloring $\phi$ is limit aperiodic because it satisfies~\eqref{epsilon, delta}. Indeed $\delta$ depends only on $\deg X$ and $\epsilon$ in~\eqref{epsilon, delta}, as stated in \Cref{t: finitary}, the indicated refinement of \Cref{t: main}. 

The definition of every $X_n$ resembles very much the notion of a \emph{shallow minor} of $X_{n-1}$ at certain \emph{depth} (see \cite{NesetrilOssona2008} and other references therein).

In the above process, there is a sequence of integers $r_n$ that provides a lower bound for the ``radii" of the clusters in $X_{n-1}$. Two crucial quantities that one needs to control are the number of suitable aperiodic colorings on each cluster, which depends exponentially on the cardinality of the cluster, and the number of clusters that are close to each other (depending on $\epsilon_n$), which is always lower than the maximum cardinality of a disk of radius $O(r_n)$. If our graph has a uniform growth function, then we can choose $r_n$ large enough so that there are enough different colorings on each cluster compared to the number of neighbouring clusters. At first glance, a similar argument could not work if the growth of the graph is not uniform, since for any choice of $r_n$ there could be points $x \in X_n$ such that there are not enough colorings compared to the number of nearby clusters. However, the crucial observation is that, if there are many neighbouring clusters, then the disk of radius $O(r_n)$ has large enough cardinality, and we can construct sufficiently many aperiodic colorings on a cluster containing the disk. This observation makes the argument more involved, since we need to divide every $X_n$ into two subsets, $X_n^\pm$, and different definitions and estimates are used in each of them. Besides this difficulty, the proof becomes quite complex with the arguments about repetitivity. It may be interesting to focus in the limit aperiodicity at first reading, omitting the arguments about repetitivity (\Cref{s. repetitivity} and its further use). 

For the sake of brevity, a preliminary part of the construction of $X_n$, concerning repetitivity, is shown in the companion paper \cite{AlvarezBarral-realization}; actually, a version for Riemannian manifolds is proved there, and the case of graphs involves simpler arguments. 

Despite its complexity, the proof only uses elementary tools, and it would be much simpler without achieving the optimal number of colors.

\subsection{Applications}

As first straightforward applications, we derive some versions of \Cref{t: main} for edge colorings and for more general graphs, and the existence of limit aperiodic and repetitive tilings. In the subsequent paper \cite{AlvarezBarral-realization}, we will give a more involved application of \Cref{t: main} concerning the realization of manifolds as leaves of compact foliated spaces.

\subsubsection{Limit aperiodic and repetitive edge colorings}\label{sss: edge colorings}

The notions of \emph{aperiodicity}, \emph{limit aperiodicity} and \emph{repetitivity} have obvious analogues for edge colorings of a connected simple graph $X$. The analogue of $D(X)$ for edge colorings is called the \emph{distinguishing index} \cite{BroerePilsniak2015}, and denoted by $DI(X)$. When $X$ is infinite, it makes sense to consider the obvious versions of $D_L(X)$ and $D_{RL}(X)$ for edge colorings, denoted by $DI_L(X)$ and $DI_{RL}(X)$, and called (\emph{repetitive}) \emph{limit distinguishing index}. 

Recall that the \emph{line graph} $X'$ of $X$ is defined as follows: the vertices of $X'$ are the edges of $X$, and two vertices of $X'$ are joined by an edge if they are edges of $X$ meeting at some vertex; thus the edges of $X'$ can be also identified to the vertices of $X$. Note that $X'$ is connected and simple, $\deg X'\le2(\deg X-1)$, and
\[
DI(X)=D(X')\;,\quad DI_L(X)=D_L(X')\;,\quad DI_{RL}(X)=D_{RL}(X')\;.
\]
Then the following is a direct consequence of \Cref{t: main}.

\begin{cor}\label{c: edge colorings}
If $X$ is an infinite connected simple graph, then $DI_L(X),DI_{RL}(X)\le2(\deg X-1)$.
\end{cor}

However, \Cref{c: edge colorings} is not very satisfactory. Its estimate can be surely improved by adapting the proof of \Cref{t: main}, probably obtaining $DI_L(X),DI_{RL}(X)\le\deg X$. We hope to prove this in another publication.

\subsubsection{Extension to general graphs}\label{sss: general graphs}

Now let $Y$ be a (countable) \emph{general graph} (with finite vertex degrees); namely, $Y$ may have a partial direction, multiple edges, and loops. Assuming that $Y$ is connected, there are obvious extensions of the concepts of \Cref{ss: main thm,sss: edge colorings} to this general setting. There is an induced undirected simple graph $\overline Y$ with the same vertex set, where the partial orientation and loops are forgotten, and with a single edge between every pair of adjacent vertices in $Y$.  Clearly, $D(Y)\le D(\overline Y)$ and $D_L(Y)\le D_L(\overline Y)$.

\begin{cor}\label{c: general graph}
If $Y$ is an infinite connected general graph, then $D_L(Y),D_{RL}(Y)\le\deg\overline Y$.
\end{cor}

The  inequality $D_L(Y)\le\deg\overline Y$ is a direct consequence of \Cref{t: main} since $D_L(Y)\le D_L(\overline Y)$. 

The inequality $D_{RL}(Y)\le\deg\overline Y$ follows with a small modification of the proof of \Cref{t: finitary}.
Namely, in \Cref{s. repetitivity}, the sets $\Omega_n$ must be defined using isometries between disks of $\overline{Y}$ induced by isomorphisms between subgraphs of $Y$. Then the isometries $\mathfrak h_{n,x}$ between disks of $\overline Y$, constructed in \Cref{s. repetitivity}, can be assumed to be induced by isomorphisms between subgraphs of $Y$. The rest of the proof can be obviously adapted. 

For example, with the notation of \Cref{ss: colorings of groups}, we can consider the \emph{Schreier graph} $Y$ defined by $G$, $S$ and any subgroup $H<G$. It is a general graph whose vertex set is $H\backslash G$, where the edges between vertices $Ha$ and $Hb$ are given by the elements $s\in S$ with $Has^{\pm1}=Hb$. By \Cref{c: general graph}, $Y$ has some limit aperiodic vertex coloring by $\deg\overline Y$ colors. Note that $\deg\overline Y\le|S|$.

\subsubsection{Limit aperiodic and repetitive tilings} 

Let us recall the general definition of tiling given in \cite{BlockWeinberger1992} (see also \cite{DranishnikovSchroeder2007}). We use the term \emph{$n$-complex} for a connected topological space with a simplicial complex structure of dimension $n$. A \emph{set of prototiles} $\TT\equiv(\TT,\FF)$ consists of a finite collection $\TT$ of compact metric $n$-complexes, called \emph{prototiles}, and a collection $\FF$ of subcomplexes of dimension $<n$, called \emph{faces}, together with an \emph{opposition} involution $o:\FF\to\FF$. A \emph{tiling} or \emph{tessellation} $\alpha$ of a metric space $X$ \emph{by $\TT$} is a collection of isometries $a_\lambda:t_\lambda\subset X\to t'_\lambda\in\TT$, where every $t_\lambda$ is called a \emph{tile} with \emph{faces} defined via $a_\lambda$, such that:
\begin{itemize}

\item $X=\bigcup_\lambda t_\lambda$;

\item the complement in $t_\lambda$ of its faces is $\Int(t_\lambda)$ in $X$;

\item if $\Int(t_\lambda\cup t_{\lambda'})\ne\Int(t_\lambda)\cup\Int(t_{\lambda'})$, then $t_\lambda$ and $t_{\lambda'}$ intersect along a face, $f$ in $t_\lambda$ and $o(f)$ in $t_{\lambda'}$; and

\item there are no free faces of $t_\lambda$. 

\end{itemize}
Similarly, we can define a \emph{set of colored prototiles} by endowing $\TT$ with a coloring $\phi$, and a \emph{set of prototiles with colored faces} by endowing $\FF$ with a coloring $\psi$ preserved by the opposition map. Then we get the corresponding definitions of (\emph{tile-}) \emph{colored tiling by $(\TT,\phi)\equiv(\TT,\FF,\phi)$} and \emph{face-colored tiling by $(\TT,\psi)\equiv(\TT,\FF,\psi)$}. These concepts can be also described by colorings of $\{t_\lambda\}$, and colorings of the set of intersections $t_\lambda\cap t_{\lambda'}$ along faces. Like $\mathcal G_*$ and $\widehat{\mathcal G}_*$ (\Cref{ss: widehat GG_*}), the sets of tilings of $X$ by $\TT$, colored tilings of $X$ by $(\TT,\phi)$ and face-colored tilings of $X$ by $(\TT,\psi)$ can be endowed with topologies after choosing a distinguished point of $X$, and there are obvious versions of \emph{aperiodicity}, \emph{limit aperiodicity} and \emph{repetitivity} for tilings, colored tilings and face-colored tilings, using isometries of the ambient metric spaces \cite{Sadun2003,BellissardBenedettiGambaudo2006,DranishnikovSchroeder2007}. Like in the case of groups (\Cref{ss: colorings of groups}), refined versions of these concepts can be given by using some subgroup of isometries, obtaining a weaker version of (limit) aperiodicity and a stronger version of repetitivity; for instance, if $X$ is a Lie group, it is natural to use its left translations.

Every tiling $\alpha$ of $X$ by $\TT$ defines a connected undirected simple graph $G$ whose vertices are the tiles of $\alpha$, with an edge between two tiles if they meet along a face. Thus $G$ is infinite just when $X$ is not compact, and $\deg G$ is bounded by the maximum number of faces of the prototiles in $\TT$, which is bounded by $|\FF|$. Therefore the following is a direct consequence of \Cref{t: main} and \Cref{c: edge colorings}.

\begin{cor}\label{c: colorings of tilings}
Suppose that $X$ is not compact, and let $\Delta$ denote the maximum number of faces of the prototiles in $\TT$. Then any {\rm(}repetitive\/{\rm)} tiling of $X$ by $\TT$ has a {\rm(}repetitive\/{\rm)} limit aperiodic tile-coloring by $\Delta$ colors, and a {\rm(}repetitive\/{\rm)} limit aperiodic face-coloring by $2(\Delta-1)$ colors.
\end{cor}

Since the face-colorings can be geometrically realized by dovetailing the faces, we get the following.

\begin{cor}\label{c: (repetitive) limit aperiodic tilings}
With the notation and conditions of \Cref{c: colorings of tilings}, if $X$ has a {\rm(}repetitive\/{\rm)} tiling by $\TT$, then it has a {\rm(}repetitive\/{\rm)} limit aperiodic tiling by at most $2|\TT|\Delta(\Delta-1)$ prototiles.
\end{cor}

For example, let $\widetilde M$ be any regular covering of a compact Riemannian $n$-manifold $M$, let $\Gamma$ denote its group of deck transformations, and let $t$ be a fundamental domain. Then the $\Gamma$-translates of $t$ form a repetitive periodic tiling of $\widetilde M$ by the prototile $t$. Here, every face $f$ of $t$ corresponds to an element $\gamma_f\in\Gamma$ such that $t\cap\gamma_ft=f$. These elements $\gamma_f$ form a generating set $S$ of $\Gamma$. By \Cref{c: (repetitive) limit aperiodic tilings}, it follows that $\widetilde M$ has a repetitive limit aperiodic tiling by at most $2|S|(|S|-1)$ prototiles; in particular, every hyperbolic space $\bfH^n$ has a repetitive limit aperiodic tiling by finitely many prototiles (cf.\ \cite{BlockWeinberger1992,DranishnikovSchroeder2007}).

With more generality, let $\Gamma$ be a discrete group acting by isometries properly and cocompactly on a metric space $X$. For any fixed $x\in X$, the orbit $\Gamma x$ is a Delone set in $X$, and the corresponding \emph{Voronoi cells},
\[
V_{\gamma x}=\{\,y\in X\mid d(y,\gamma x)\le d(y,\Gamma x)\,\}\quad(\gamma\in\Gamma)\;,
\]
form a repetitive periodic tiling of $X$ by one prototile (all tiles are isometric). Let $\Delta$ denote the number of faces of these tiles. Then, by \Cref{c: colorings of tilings}, $X$ has a repetitive limit aperiodic tiling by at most $2\Delta(\Delta-1)$ prototiles (cf.\ \cite{DranishnikovSchroeder2007}).

In \Cref{c: colorings of tilings,c: (repetitive) limit aperiodic tilings}, and in the previous examples, the number of colors or prototiles would be improved by the expected improvement of \Cref{c: edge colorings}.

\section{Preliminaries on graphs and colorings}\label{s: prelims on graphs & colorings}

Let us recall some basic definitions and elementary results about graphs and its metric properties. Short proofs are indicated for completeness.

\subsection{Graphs}\label{ss: graphs} 

 Let $\xset\equiv(\xset,E)$ be an (undirected) {simple graph} ($\xset$ and $E$ are the sets of vertices and edges, respectively). The term ``simple'' refers to the existence of no loops and of at most one edge joining any pair of vertices. Thus we may also consider $E$ as a symmetric relation on $X$ where no point is related to itself. Recall that the {\em degree\/} (or {\em valency\/}) $\deg x$ of a vertex $x$ is the number of edges connecting to $x$. The \emph{degree} of $X$ is $\deg X=\sup_{x\in X}\deg x$. For\footnote{We assume that $0\in\N$.} $n\in\N$, a {\em path\/} of {\em length\/} $n$ from $x$ to $y$ in $\xset$ is a sequence of $n$ consecutive\footnote{With a common vertex.} edges joining $x$ to $y$; in terms of their vertices, it can be considered as a sequence $(z_0,\dots,z_n)$, where $z_0=x$, $z_n=y$, and $z_{i-1}Ez_i$ for all $i=1,\dots,n$. If any two vertices of $\xset$ can be joined by a path, then $\xset$ is called {\em connected\/}. The topological and geometric properties of $X$ indeed refer to its geometric realization.

On any $\yset \subset \xset$, we get the subgraph $\yset \equiv(\yset,E|_{\yset})$. By Zorn's lemma, there are maximal connected subgraphs of $\xset$, called {\em connected components\/}, which form a partition of $\xset$. Any connected subgraph of $\xset$ is contained in some connected component of $\xset$.

Let $\xset'\equiv(\xset',E')$ be another graph. Recall that any relation preserving bijection $\xset\to\xset'$ is called an {\em isomorphism\/} ({\em of graphs\/}). Given distinguished points, $x_0\in \xset$ and $x'_0\in \xset'$, a ({\em pointed\/}) {\em isomorphism\/} $f:(\xset,x_0)\to(\xset',x'_0)$ is an isomorphism $f:\xset\to \xset'$ satisfying $f(x_0)=x'_0$. The notation $\xset\cong \xset'$ or $(\xset,x_0)\cong(\xset',x'_0)$) may be used in these cases. The term ({\em pointed\/}) {\em automorphism\/} is used for a (pointed) isomorphism of a (pointed) graphs to itself. The group of automorphisms of $\xset$ (respectively, $(\xset,x_0)$) is denoted by $\Aut(\xset)$ (respectively, $\Aut(\xset,x_0)$).

Assume from now on that $\xset$ is connected. Then the natural $\N$-valued metric $d$ on $X$ is defined by declaring $d(x,y)$ to be the minimum length of paths in $\xset$ from $x$ to $y$. The following property is easily verified:
\begin{equation}\label{exists z}
\forall x,y\in \xset,\ \forall m,n\in\N,\ d(x,y)=m+n\Longrightarrow\exists z\in \xset\mid d(x,z)=m,\ d(y,z)=n\;.
\end{equation}
$E$ and $d$ are equivalent structures on $X$. Thus we may consider the connected graph $X$ as the metric space $(\xset,d)$, and an isomorphism between connected graphs as an isometry. A path $(u_0,\dots,u_n)$ in $X$ is called a {\em minimizing geodesic segment\/} if $d(u_0,u_n)=n$. By~\eqref{exists z}, there exists a minimizing geodesic segment joining any pair of vertices.

Let us recall some basic metric concepts and properties for the particular case of the connected graph $X$. For $x\in \xset$ and $r\in\N$, let $\sphere(x,r)=\{\,y\in X\mid d(x,y)=r\,\}$ and $D (x,r)=\{\,y\in X\mid d(x,y)\le r\,\}$ (the sphere and disk of center $x$ and radius $r$). For another integer $s\geq r\geq 0$, the set $\corona(x,r,s)=D (x,s)\setminus D (x,r)$ is called a \emph{corona}. For $Q\subset X$, its \emph{closed penumbra}\footnote{The penumbra $\Pen(Q,r)$ usually has a similar definition with a strict inequality. On graphs it is more practical to use non-strict inequalities.}  of \emph{radius} $r$ is $\CPen(Q,r)=\{\,y\in X\mid d(Q,y)\le r\,\}$; in particular, $\CPen(D(x,r),t)=D(x,r+t)$ for $r,t\in\N$ by~\eqref{exists z}. We may add $\xset$ as a subindex to all of this notation if necessary. Observe that $D (x,r)$ is connected. More generally, $\CPen(Q,r)$ is connected if $Q$ is connected. Note also that $|\sphere(x,0)|=1$ and $|\sphere(x,1)|=\deg x$. It is said that $Q$ is ($K$-) {\em separated\/} if there is some $K\in\Z^+$ such that $d(x,y)\ge K$ for all $x\ne y$ in $Q$. On the other hand, $Q$ is said to be ($C$-) {\em relatively dense\/}\footnote{A {\em $C$-net\/} is similarly defined with the penumbra. If reference to $C$ is omitted, both concepts are equivalent.} in $X$ if there is some $C>0$ such that $\CPen(Q,C)=X$. A separated relatively dense subset is called a \emph{Delone} subset.

\begin{lem}[\'Alvarez-Candel {\cite[Proof of Lemma~2.1]{AlvarezCandel2011}}]\label{l: maximal}
A maximal $K$-separated subset $Q$ is $(K-1)$-relatively dense in $X$.
\end{lem}

\Cref{l: maximal} has the following easy consequence using Zorn's lemma.

\begin{cor}[{Cf.\ \cite[Lemma~2.3 and Remark~2.4]{AlvarezCandel2018}}]\label{c: existence of a separated net}
Any $K$-separated subset of $\xset$ is contained in some maximal $K$-separated $(K-1)$-relatively dense subset.
\end{cor}

On any connected $\yset\subset\xset$, two canonical metrics can be considered, $d_{\yset}$ (defined by $E|_Y$) and the restriction of $d_{\xset}$. Clearly, $d_X\le d_Y$ on $Y$.

\begin{lem}\label{l: d_Y = d_X}
Let $Y=\CPen(Y_0,r)$ for a connected $Y_0\subset X$ and $r\in\N$. Then $d_Y(x,y)=d_X(x,y)$ for all $x,y\in Y_0$ with $d_X(x,y)\le2r$.
\end{lem}

\begin{proof}
Let $(u_0,\dots,u_n)$ be a minimizing geodesic segment of $X$ between $x,y\in Y_0$ of length $n\le2r$. Then $d_X(x,u_i),d_X(y,u_j)\le r$ if $i,n-j\le r$, yielding $u_0,\dots,u_n\in Y$. So $(u_0,\dots,u_n)$ is a path in $Y$, and therefore $d_Y(x,y)\le n=d_X(x,y)$. 
\end{proof}

\begin{cor}\label{l: K-sep w.r.t. d_Y <=> K-sep w.r.t. d_X}
With the notation of \Cref{l: d_Y = d_X}, let $A\subset Y_0$ and $2r\ge K\in\Z^+$. Then $A\subset Y_0$ is $K$-separated with respect to $d_Y$ if and only if it is $K$-separated in $d_X$.
\end{cor}

\begin{defn}
For connected $Y,Z\subset X$ and $m\in\N$, a map $f:X\to Y$ is called an \emph{$m$-short scale isometry} if $d_Z(f(x),f(y))=d_Y(x,y)$ for all $x,y\in Y$ with $d_Y(x,y)\le m$.
\end{defn}

The above definition is also valid for maps between arbitrary metric spaces.

\begin{cor}\label{c: graph partial}
Let $Y=\CPen(Y_0,r)$ and $Z=\CPen(Z_0,r)$ for connected $Y_0,Z_0\subset X$ and $r\in\N$, and let $2r\ge m\in\N$. If $f:Y\to Z$ is a graph isomorphism with $f(Y_0)=Z_0$, then $f:Y_0\to Z_0$ is an $m$-short scale isometry with respect to the restrictions of $d_X$.
\end{cor}

\begin{proof}
For $x,y\in Y_0$ with $d_X(x,y)\le m\le2r$, we have $d_Y(x,y)=d_X(x,y)$ by \Cref{l: d_Y = d_X}. So $d_Y(x,y)=d_Z(f(x),f(y))$ since $f:Y\to Z$ is an isomorphism. Thus $d_Z(f(x),f(y))\le2r$, and therefore $d_Z(f(x),f(y))=d_X(f(x),f(y))$ by \Cref{l: d_Y = d_X} because $f(x),f(y)\in Z_0$. Finally, we get $d_X(x,y)=d_X(f(x),f(y))$.
\end{proof}

\begin{cor}\label{c: preserves metric}
For $x,y\in X$ and $r\in\N$, if $h\colon (D(x,2r),x)\to (D(y,2r),y)$ is a pointed isomorphism, then $h\colon D(x,r)\to D(y,r)$ is an isometry with respect to the restrictions of $d_X$.
\end{cor}

\begin{lem}\label{l: X is countable}
If every vertex of $\xset$ is adjacent to a countable set of vertices, then $\xset$ is countable.
\end{lem}

\begin{proof}
Given any $x\in \xset$, since $\xset=\bigcup_{r=0}^\infty \sphere(x,r)$, it is enough to prove that $\sphere(x,r)$ is countable for all $r\in\N$. This is done by induction on $r$. We have $\sphere(x,0)=\{x\}$, and $\sphere(x,1)$ is countable by hypothesis. If $\sphere(x,r)$ is countable for some $r\in\N$, then $\sphere(s,r+1)$ is also countable because it is contained in $\bigcup_{y\in \sphere(x,r)}\sphere(y,1)$.
\end{proof}

\begin{lem}\label{l: proper}
The vertices of $\xset$ have finite degree if and only if its disks are finite.
\end{lem}

\begin{proof}
The ``if'' part is true because $|D (x,1)|=1+\deg x$ for all $x\in X$. Now assume that the vertices have finite degree, and let us show that $|D (x,r)|<\infty$ for all $x\in X$ and $r\in\Z^+$. This follows by induction on $r$ using that $D (x,r+1)=\CPen(D (x,r),1)$ by~\eqref{exists z}.
\end{proof}

The disks of $X$ are finite just when $X$ is a \emph{proper} metric space, in the sense that its  disks are compact.

\begin{lem}\label{l: |S(x r)| ge 1}
  If $X$ is unbounded, then $|\sphere(x,r)|\ge1$ for all $x\in X$ and $r\in\N$.
\end{lem}

\begin{proof}
By~\eqref{exists z} and since $X$ is unbounded, we have $\sphere(x,r)\ne\emptyset$ for all $r\in\N$.
\end{proof}

\begin{cor}\label{c: |D(x,r)| ge r+1}
  If $X$ is unbounded, then $|D (x,r)|\ge r+1$ and $|C(x,r,s)|\ge s-r$ for all $x\in X$ and $r<s$ in $\N$.
\end{cor}

\begin{proof}
Apply \Cref{l: |S(x r)| ge 1} to the expressions\footnote{A dotted union symbol is used for unions of disjoint sets.} $D (x,r)=\bigcupdot_{i=0}^r\sphere(x,i)$ and $\corona(x,r,s)=\bigcupdot_{i=r+1}^s\sphere(x,i)$.
\end{proof}

Now suppose also that $\Delta:=\deg X<\infty$. Since $\xset$ is connected, it is a singleton if $\Delta=0$, and it has two vertices if $\Delta=1$. Thus assume $\Delta\ge2$.

\begin{lem}\label{l: |S(x r)| le k(k-1)^r-1}
  $|\sphere(x,r)|\le \Delta(\Delta-1)^{r-1}$ for all $x\in \xset$ and $r\in\Z^+$.
\end{lem}

\begin{proof}
The vertex $x$ is adjacent with  at most $\Delta$ vertices, which form $\sphere(x,1)$. For all $r\in\Z^+$, any $y\in \sphere(x,r)$ is adjacent with at least one vertex in $\sphere(x,r-1)$ by~\eqref{exists z}, and therefore $y$ is adjacent to at most $\Delta-1$ vertices in $\sphere(x,r+1)$. Then the inequality $|\sphere(x,r)|\le \Delta(\Delta-1)^{r-1}$ follows easily by induction on $r$.
\end{proof}

\begin{cor}\label{c: |D(x r)| le ...}
  Let $x\in \xset$ and $r \in \Z^+$. Then
\[
|D (x,r)|\le 
\begin{cases}
1+2r & \text{if $\Delta=2$}\\
3\cdot2^r & \text{if $\Delta=3$}\\
4(\Delta-1)^r & \text{if $\Delta>3$}\;.
\end{cases}
\]
\end{cor}

\begin{proof}
Applying \Cref{l: |S(x r)| le k(k-1)^r-1} to the disjoint union $D (x,r)=\bigcupdot_{i=0}^{r}\sphere(x,i)$, we get $|D (x,r)|\le 1+2r$ if $\Delta=2$, and
\[
|D (x,r)|\le  1 +\frac{\Delta((\Delta-1)^{r}-1)}{\Delta-2}=\frac{\Delta(\Delta-1)^r-2}{\Delta-2}
\]
if $\Delta\ge3$. But
\[
\frac{\Delta(\Delta-1)^r-2}{\Delta-2}=3\cdot2^r-2<3\cdot2^r
\]
if $\Delta=3$, and
\[
\frac{\Delta(\Delta-1)^r-2}{\Delta-2}<2\frac{\Delta(\Delta-1)^r-1}{\Delta-1}<4\frac{\Delta(\Delta-1)^r}{\Delta}=4(\Delta-1)^r
\]
if $\Delta>3$ because
\[
u\ge v\ge1\Longrightarrow2\frac{u+1}{v+1}>\frac{u}{v}\;.\qedhere
\]
\end{proof}

\begin{lem}\label{l: cardinality separated}
If $A$ is a $K$-separated $(K-1)$-relatively dense subset of $X$ for some $K\in\Z^+$, then $|A|> |X|/\Delta^K$.
\end{lem}

\begin{proof}
We have $X\subset \bigcup_{a \in A} D(a,K-1)$, yielding $|X| \leq \sum_{a\in A} |D(a,K-1)|$. By \Cref{c: |D(x r)| le ...}, for $a\in A$,
\begin{alignat*}{2}
|D(a,K-1)|&\leq1+2(K-1)<2^K&\qquad\text{if}\ \Delta&=2\;,\\
|D(a,K-1)|&\leq 3\cdot2^{K-1}<3^K&\qquad\text{if}\ \Delta&=3\;,\\
|D(a,K-1)|&\leq4(\Delta-1)^{K-1}\le\Delta(\Delta-1)^{K-1}<\Delta^K&\qquad\text{if}\ \Delta&>3\;.\qedhere
\end{alignat*}
\end{proof}

\subsection{Colorings}\label{ss: colorings}

A {\em coloring\/} of a set $\xset$ (by a set $F$ ``of {\em colors\/}'') is a map $\phi:\xset\to F$. The pair $(\xset,\phi)$ is called a {\em colored set\/}.  The sets of colors $F$ will usually be a finite initial segmen\footnote{Recall that a subset $S$ of an ordered set $(Z,\leq)$ is called an \emph{initial segment} if, for all $s\in S$ and $z\in Z$, $z\leq s$ implies $z\in S$.} of $\N$, denoted by $[M ]=\{0,\dots, M-1\}$ for some $M\in \N$.

Let $\xset$ be a simple graph. A coloring of its vertex set, $\phi:\xset\to F$, is called a ({\em vertex\/}) {\em coloring\/} of $\xset$, and $(\xset,\phi)$ is called a {\em colored graph\/}. If $x_0\in \yset\subset \xset$, then the simplified notation $(\yset,\phi)=(\yset,\phi|_{\yset})$ will be used. The following concepts for colored graphs are the obvious extensions of their graph versions: ({\em pointed\/}) {\em isomorphisms\/}, denoted by $f:(\xset,\phi)\to(\xset',\phi')$ and $f:(\xset,x_0,\phi)\to(\xset',x'_0,\phi')$, {\em isomorphic\/} (pointed) colored graphs, denoted by $(\xset,\phi)\cong(\xset',\phi')$ and $(\xset,x_0,\phi)\cong(\xset',x'_0,\phi')$, and {\em automorphism\/} groups of (pointed) colored graphs, denoted by $\Aut(\xset,\phi)$ and $\Aut(\xset,x_0,\phi)$.

Consider only colorings by $F$. Let $\widehat\GG_*$ be the set\footnote{The graphs $X$ are countable (\Cref{l: X is countable}), and therefore we can assume that their underlying sets are contained in $\N$. In this way, $\widehat\GG_*$ becomes a well defined set.} of isomorphism classes, $[\xset,x,\phi]$, of pointed connected colored graphs, $(\xset,x,\phi)$, whose vertices have finite degree. For each $R\in\Z^+$, let 
\[
\widehat U_R=\{\,([\xset,x,\phi],[\yset,y,\psi])\in\widehat\GG_*^2\mid(D_{\yset}(y,R),y,\psi)\cong(D_{\xset}(x,R),x,\phi)\,\}\;.
\]
These sets form a base of entourages of a uniformity on $\widehat\GG_*$, which is easily seen to be complete. Moreover this uniformity is metrizable because this base is countable. 

Note that the {\em degree map\/} $\deg:\widehat\GG_*\to\Z^+$, $[\xset,x,\phi]\mapsto\deg x$, and the {\em evaluation map\/} $\ev:\widehat\GG_*\to F$, $[\xset,x,\phi]\mapsto\phi(x)$, are continuous (the target spaces being discrete). Suppose that $F$ is countable. Then $\widehat\GG_*$ is separable because the elements $[\xset,x,\phi]$, where $\xset$ is finite, form a countable dense subset. Thus $\widehat\GG_*$ becomes a Polish space. 

For any connected simple colored graph $(\xset,\phi)$, there is a canonical map $\hat\iota_{\xset,\phi}:\xset\to\widehat\GG_*$ defined by $\hat\iota_{\xset,\phi}(x)=[\xset,x,\phi]$. Its image, denoted by $[X,\phi]$, has an induced connected colored graph structure, and all of these images form a canonical partition of $\widehat\GG_*$. The saturation of any open subset of $\widehat{\GG}_*$ is open, and therefore the closure operation preserves saturated subsets of $\widehat{\GG}_*$ \cite[Section~2.6]{AlvarezBarral-realization}; in particular, $\overline{[X,\phi]}$ is saturated. The following result indicates the role played by graphs with finite degrees, colored by finitely many colors.

\begin{prop}\label{p: compact}
The closure $\overline{[\xset,\phi]}$ is compact if and only if $\deg X,|\im\phi|<\infty$.
\end{prop}

\begin{proof}
The ``if'' part follows using that, if $\deg X,|\im\phi|<\infty$, then, for each $R\in\Z^+$, the pointed colored disks $(D_{\xset}(x,R),x,\phi)$ ($x\in \xset$) represent finitely many pointed isomorphism classes $[D_{\xset}(x,R),x,\phi]$. The ``only if'' part follows using the continuity of $\deg:\widehat\GG_*\to\Z^+$ and $\ev:\widehat\GG_*\to F$.
\end{proof}

It is said that $(\xset,\phi)$ (or $\phi$) is {\em aperiodic\/} (or \emph{non-periodic}) if $\Aut(\xset,\phi)=\{\id_{\xset}\}$, which means that $\hat\iota_{\xset,\phi}$ is injective; otherwise, it is said that $(\xset,\phi)$ (or $\phi$) is {\em periodic\/}. More strongly, $(\xset,\phi)$ (or $\phi$) is called {\em limit aperiodic\/} if $(\yset,\psi)$ is aperiodic for all $[\yset,y,\psi]\in\overline{[\xset,\phi]}$. If $\xset$ is finite, aperiodicity is equivalent to its limit aperiodicity, and an aperiodic coloring of $\xset$ by finitely many colors can be easily given. If $\xset$ is infinite, limit aperiodic colorings by finite finitely many colors are much more difficult to construct. The following lemma will be useful for that purpose.

\begin{lem}\label{l: limit aperiodic}
$(\xset,\phi)$ is limit aperiodic if and only if, for all sequences, $x_i,y_i$ in $\xset$ and $R_i,S_i\uparrow\infty$ in $\Z^+$, and pointed isomorphisms,
\[
f_i:(D(x_i,R_i),x_i,\phi)\to(D(x_{i+1},R_i),x_{i+1},\phi)\;,\quad h_i:(D(x_i,S_i),x_i,\phi)\to(D(y_i,S_i),y_i,\phi)\;,
\]
such that $d(x_i,y_i)+S_i\le R_i$, $f_i(y_i)=y_{i+1}$, and the diagram
\begin{equation}
\begin{CD}
D(x_{i+1},S_{i+1}) @>{h_{i+1}}>> D(y_{i+1},S_{i+1}) \\
@A{f_i}AA @AA{f_i}A \\
D(x_i,S_i) @>{h_i}>> D(y_i,S_i)
\end{CD}
\end{equation}
is commutative, we have that, either $x_i=y_i$ for $i$ large enough, or $\limsup_id(x_i,y_i)=\infty$. 
\end{lem}

\begin{proof}
This follows easily from the definition of the topology of $\widehat\GG_*$.
\end{proof}

\begin{rem}\label{r: limit aperiodic}
In \Cref{l: limit aperiodic}, the case of bounded sequences $x_i,y_i$ characterizes the aperiodicity of $\xset$. Thus the case of unbounded sequences $x_i,y_i$ describes when $(\yset,\psi)$ is aperiodic for all $[\yset,y,\psi]\in\overline{[\xset,\phi]}\setminus[\xset,\phi]$.
\end{rem}

On the other hand, $(\xset,\phi)$ (or $\phi$) is called {\em repetitive\/} if there is some point $p\in \xset$ and a sequence $R_i\uparrow \infty$ in $\Z^+$ such that the sets 
\[
\{\, x\in \xset  \mid [D(p,R_i), p,\phi]=[D(x,R_i),x,\phi] \, \}
\] 
are relatively dense in $\xset$. This property is clearly independent of the choice of $p$. If $(X,\phi)$ is repetitive, then $\overline{[X,\phi]}$ is minimal, and the reciprocal also holds when $\overline{[X,\phi]}$ is compact \cite[Section~2.6]{AlvarezBarral-realization}.

Removing the colorings from the notation, we get the Polish space $\GG_*$ of isomorphism classes of pointed connected graphs. In this way, we get canonical maps $\iota_{\xset}:\xset\to\GG_*$ for connected graphs $\xset$, defining a canonical partition of $\GG_*$. Then it is said that $\xset$ is {\em aperiodic\/} if $\iota_{\xset}$ is injective, $\xset$ is {\em limit aperiodic\/} if $Y$ is aperiodic for all $[Y,y]\in\overline{[\xset]}$, and $\xset$ is {\em repetitive\/} if $\overline{[\xset]}$ is a minimal set of the canonical partition. Observe also that the forgetful map $\widehat\GG_*\to\GG_*$ is continuous. By \Cref{l: proper}, the space $\GG_*$ is a subspace of the Gromov space $\MM_*$ of isometry  classes of pointed proper metric spaces \cite{Gromov1981}, \cite[Chapter~3]{Gromov1999}. The obvious versions of \Cref{l: limit aperiodic} and \Cref{p: compact} in this setting follow by considering a constant coloring.

\subsection{A refinement of the main theorem}\label{ss: refinement}

We will derive \Cref{t: main} from the following finitary version.

\begin{thm}\label{t: finitary}
Let $\xset$ be a connected infinite simple graph with  $\Delta:=\deg X<\infty$. Then the following properties hold for any sequence $\epsilon_n\uparrow\infty$ in $\Z^+$:
\begin{enumerate}[(i)]

\item\label{i: finitary limit aperiodic} There are:
\begin{itemize}

\item a sequence $\delta_n$ in $\Z^+$, with every $\delta_n$ depending only on $\Delta$, $\epsilon_m$ for $m\leq n$, and $\delta_m$ for $m<n$; and

\item a sequence of colorings $\phi^N$ of $\xset$ by $\Delta$ colors, with every $\phi^N$ depending on $\epsilon_m$ for $m\leq N$;

\end{itemize}
such that, for all $x,y\in\xset$, $N\in\N$ and $0\le n\le N$,
\[
0<d(x,y)<\epsilon_n\Longrightarrow[D (x,\delta_n),x,\phi^N]\neq[D (y,\delta_n),y,\phi^N]\;.
\]

\item\label{i: finitary repetitive} Suppose that, for some $p\in \xset$ and some sequence $\mathfrak{r}_n\uparrow\infty$ and $\omega_i$ in $\Z^+$, with every $\mathfrak{r}_n$ large enough depending on $\Delta$ and $\epsilon_m$ for $m\leq n$, the sets
\[
\Omega_n = \{\,x\in \xset \mid [D (x, \mathfrak{r}_n),x,d_X]= [D (p,\mathfrak{r}_n),p,d_X] \, \}
\]
are $\omega_n$-relatively dense in $\xset$. Then there are:
\begin{itemize}

\item a sequence $r_n\uparrow\infty$ in $\Z^+$, with every $r_n$ depending on $\Delta$, $\epsilon_m$ and $\omega_m$ for $m\leq n$, and $r_m$ for $m<n$;

\item a sequence $\alpha_n$ in $\Z^+$, with every $\alpha_n$ depending on $\Delta$, $\epsilon_m$ and $\omega_m$ for $m\leq n$, and $r_m$ and $\alpha_m$ for $m<n$; and

\item a sequence of colorings $\phi^N$ by $\Delta$ colors, with every $\phi^N$ depending on $\epsilon_m$ and $\mathfrak{r}_m$ for $m\leq N$;

\end{itemize}
such that $\phi^N$ satisfies~\ref{i: finitary limit aperiodic} with some sequence $\delta_n$, and the sets  
\[
\textstyle{\widehat{\Omega}_n = \{\, x\in \xset \mid [D (x,\sum_{i=0}^n r_i),x,\phi^N]= [D (p,\sum_{i=0}^n r_i),p,\phi^N] \,\}}
\]
are $\alpha_n$-relatively dense in $\xset$ for $0\leq n \leq N$.

\end{enumerate}
\end{thm}

As indicated in \Cref{ss: main thm}, \Cref{t: finitary} is stronger than \Cref{t: main} because $\delta_n$, $r_n$ and $\alpha_n$ are independent of the choice of $X$ satisfying the hypothesis. 

In \Cref{t: finitary}, the assumption that $X$ is infinite can be disposed of. The same ideas work with minor tweaks when $X$ is a finite graph large enough depending on $\deg X$, refining also \Cref{t: finite case}. Since the proof is already quite involved, we leave the details to the interested reader.

Let us derive \Cref{t: main}  from \Cref{t: finitary}.
Let $X$ be a graph and $\epsilon_n$ be an increasing sequence of positive integers satisfying the conditions of \Cref{t: finitary}.
Then this result gives a sequence of colorings $\phi^N$.
The set of colorings of $X$ by $\Delta$ colors is endowed with the topology of convergence over finite subsets of $X$.
Since the set $[\Delta]$ of colors is finite,  we can suppose that the sequence of colorings $\phi^N$ converges to some coloring $\phi$, possibly after passing to a subsequence. 
This means that, on any finite $A\subset X$, the colorings $\phi$ and $\phi^N$ coincide for $N$ large enough.
Let us prove that $\phi$ is a limit aperiodic coloring.

Assume by absurdity that there are some $n\in \N$ and $x,y\in X$ so that $0<d(x,y)<\epsilon_n$ and $[D (x,\delta_n),x,\phi]\neq[D (y,\delta_n),y,\phi]$.
By the convergence of $\phi^N$, there is some $N\ge n$ such that $ [D (x,\delta_n),x,\phi] = [D (x,\delta_n),x,\phi^N]$ and $[D (y,\delta_n),y,\phi]=[D (y,\delta_n),y,\phi^N]$, contradicting \Cref{t: finitary}~\ref{i: finitary limit aperiodic}.
Therefore $\phi$ satisfies \Cref{t: finitary}~\ref{i: finitary limit aperiodic}, with the same choice of sequence $\delta_n$. The fact that $\phi$ is limit aperiodic follows now from \Cref{l: limit aperiodic}.	

Suppose that, additionally, the family $\phi^N$ satisfies the conditions of \Cref{t: finitary}~\ref{i: finitary repetitive}, with a distinguished point $p$.
Let us show that $\phi$ is repetitive.
For any $n\leq N$ and $x\in X$, there is some $y\in X$ such that $d(x,y)\leq \alpha_n$ and $[D(y,\sum_{i=0}^n r_i),y , \phi^N]=[D(p,\sum_{i=0}^n r_i),p,\phi^N ]$.
Assume by absurdity that there are some $n\in \N$ and $x\in X$ such that $[D(y,\sum_{i=0}^n r_i),y , \phi]\neq[D(p,\sum_{i=0}^n r_i),p,\phi]$ for all $y\in D(x,\alpha_n)$.
By the convergence of $\phi^N$, we have that $\phi$ and $\phi^N$ coincide over $D(p,\sum_{i=0}^n r_i)$ and $D(y,\sum_{i=0}^n r_i)$ for every $y\in D(x,\alpha_n)$ and $N$ large enough, contradicting \Cref{t: finitary}~\ref{i: finitary repetitive}.
Therefore the sets
\[
\textstyle{\{\, x\in \xset \mid [D (x, \sum_{i=0}^n r_j),x,\phi]= [D (p,\sum_{i=0}^n r_j),p,\phi] \,\}}
\]
are $\alpha_n$-relatively dense in $X$.
So $\phi$ satisfies \Cref{t: finitary}~\ref{i: finitary repetitive}, with the same choice of sequence $\alpha_n$, and the result follows by the definition of repetitiveness.

The rest of the paper is devoted to prove \Cref{t: finitary}.

\section{Constants}\label{s. constants}

In order to prove our result, we need to define quantities depending on the sequences appearing in the statement of \Cref{t: finitary} that will function as a priori upper bounds for parameters that arise in the definition of $\phi$.
They depend on each other in non-trivial ways, so their definitions are quite involved, which makes this section rather technical.

Let $X$ be a graph satisfying the conditions of \Cref{t: finitary}, and let $\epsilon_n$ be an increasing sequence of positive integers.
By induction on $n\in \N$, we are going to define sequences of positive integers, $s_n$, $\hat{r}_n$, $\hat{r}_n^\pm$, $\bar{r}_n$ and $\bar{r}_n^\pm$, and sequences of functions, $\bar{\eta}_n,\bm{R}_n^\pm, \bm{l}_n,\bm{K}_n, \overline{\bm{K}}_n \colon \mathbb{N}\to \N$ and $\bm{L}_n, \bm{\Gamma}_n^\pm, \bm{\Delta}_n\colon\N^{n+1} \to \mathbb{N}$. 
First, set 
\begin{equation}\label{defn s0}
s_0=27+\epsilon_0\;, \quad \Delta_{-1}= \deg X=\Delta\;.
\end{equation}
The notation $\deg X$, $\Delta$ and $\Delta_{-1}$ will be used indistinctly, depending on  convenience. Define $\bar{\eta}_0\colon \N\to \mathbb{Q}$ by
\begin{equation}\label{defn bareta0}
\bar{\eta}_0(a)= \exp_2\big(\big\lfloor(a-\Delta^{11}-1)/\Delta^3\big\rfloor\big)\;,
\end{equation}
where we use the notation $\exp_2(r)=2^r$ for $r\in\R$. The number $\bar{\eta}_0(a)$ will represent an \emph{a priori} lower bound on the number of different (up to pointed graph isomorphism) rigid colorings that we can use on every cluster of ``radius" $a$ in $X$  (see \Cref{s. idea}). Let $\hat{r}_0$ be the smallest positive integer such that \begin{equation}\label{tilder0 defn}
\bar{\eta}_0\left(  \sqrt{\bar{\eta}_0(\hat{r}_0)}-6 \right)> \left(4(\Delta-1)^{\hat{r}_0s^2_0(3s_0+1)}+6\right)^2\;.
\end{equation}
Note that this is well-defined since there is a double exponential in the left-hand side of the inequality, whereas there is a single exponential on the right-hand side.
Observe also that~\eqref{defn bareta0} and~\eqref{tilder0 defn} yield
\begin{equation}\label{hat r_0 > 2^11}
\hat{r}_0>2^{11}
\end{equation}
because $\Delta\ge2$ since $X$ is infinite.
Let 
\begin{equation}\label{defn barr0}
\bar{r}_0=\hat{r}_0(3s_0+1)\;.
\end{equation}
From \eqref{tilder0 defn} and the fact that $\bar{\eta}_0$ is an increasing function we get 
\begin{align}
\bar{\eta}_0\left(  \sqrt{\bar{\eta}_0(\bar{r}_0)}-6\right)&> 
\bar{\eta}_0\left(  \sqrt{\bar{\eta}_0(\hat{r}_0)}-6 \right)\notag\\
&>\left(4(\Delta-1)^{\hat{r}_0s^2_0(3s_0+1)}+6\right)^2=
\left(4(\Delta-1)^{\bar{r}_0s_0^2}+6\right)^2\;.\label{barr0 defn}
\end{align}
Define the remaining functions for $n=0$ as follows: 
\begin{equation}\label{defn functions 0}
\left.\begin{alignedat}{3}
\bm{R}_0^{-}(a)&=4a-1\;,&\quad \bm{R}_0^{+}(a) &= a(2s_0+3)\;, &\quad \bm{l}_0(a)&=2\bm{R}_0^{+}(a)+1\;, \\
\bm{\Delta}_0(a)&= 4(\Delta-1)^{2\bm{R}_0^{+}(a)}\;, &\quad \bm{L}_0(a)&=\bm{l}_0(a)\;, &\quad \bm{\Gamma}_0^\pm(a) &= \bm{R}_0^\pm(a)\;.  
\end{alignedat}\;\right\}
\end{equation}

Now, given $n>0$, suppose that we have defined the desired constants and functions for integers $0\leq m<n$.
Using the notation $\bar{\bm{r}}_{n-1}=(\bar{r}_0,\dots, \bar{r}_{n-1})$, define 
\begin{equation}\label{defn sn}
s_n= 27+10 \bm{L}_{n-1}(\bar{\bm{r}}_{n-1})+2\bm{\Gamma}_{n-1}^+(\bar{\bm{r}}_{n-1}) + \epsilon_n\;.
\end{equation}
Let $\bar{\eta}_n\colon \N \to \mathbb{Q}$ be defined by
\begin{equation}\label{dfn overlineeta}
\bar{\eta}_n(a)= \exp_2\Big(\Big\lfloor\big(a-\bm{\Delta}_{n-1}^{11}(\bar{\bm{r}}_{n-1})-1\big)/\bm{\Delta}_{n-2}^{\bar{r}_{n-1}^2s_{n-1}}(\bar{\bm{r}}_{n-2})\Big\rfloor\Big)\;.
\end{equation}
Then let $\hat{r}_n$  be the smallest positive integer so that
\begin{equation}\label{hatrn}
\bar{\eta}_n\left(  \sqrt{\bar{\eta}_n(\hat{r}_n)}-6 \right)>
\left(4(\bm{\Delta}_{n-1}(\bar{\bm{r}}_{n-1})-1)^{\hat{r}_ns^2_n(3s_n+1)}+6\right)^2\;.
\end{equation}
This is well-defined like in the case of $\hat{r}_0$.
Let 
\begin{equation}\label{defn barrn}
\bar{r}_n=\hat{r}_n(3s_n+1)\;.
\end{equation}
From~\eqref{tilder0 defn},~\eqref{hatrn} and the fact that $\bar{\eta}_n$ is an increasing function, we get 
\begin{align}
\bar{\eta}_n\left(  \sqrt{\bar{\eta}_n(\bar{r}_n)}-6 \right)&>\bar{\eta}_n\left(  \sqrt{\bar{\eta}_n(\hat{r}_n)}-6 \right)\notag\\
&>\left(4(\bm{\Delta}_{n-1}(\bar{\bm{r}}_{n-1})-1)^{\hat{r}_ns_n^2(3s_n+1)}+6\right)^2\notag\\
&=\left(4(\bm{\Delta}_{n-1}(\bar{\bm{r}}_{n-1})-1)^{\bar{r}_ns_n^2}+6\right)^2\;.\label{barrn defn}
\end{align}
For $n\in \N$, using the notation $\bm{a}_n=(a_0,\dots, a_n)$ and $\bm{a}_{n-1}=(a_0,\dots, a_{n-1})$, let 
\begin{equation}\label{defn functions n}
\left.\begin{gathered}
\bm{R}_n^{-}(a)=4a-1\;,\quad   
\bm{R}_n^{+}(a) = a(2s_n+3)\;, \quad 
\bm{l}_n(a)=2\bm{R}_n^{+}(a)+1\;,
\\
\bm{\Delta}_n(\bm{a}_n)=
4\left(\bm{\Delta}_{n-1}(\bm{a}_{n-1})-1\right)^{2\bm{R}_n^{+}(a_n)}\;, 
\\
\bm{L}_n(\bm{a}_N)=
\prod_{i=0}^{n} \bm{l}_i(a_i)\;, \quad 
\bm{\Gamma}_n^\pm(\bm{a}_N)= 
\bm{R}_n^\pm(a_n) \cdot \bm{L}_{n-1}(\bm{a}_{N-1}) + \bm{\Gamma}_{n-1}^+(\bm{a}_{N-1})\;. 
\end{gathered} \;\right\}
\end{equation}
Note that $\bm{R}_n^{-}$ is independent of $n$, it is only included for the sake of notational consistency.
Also, by a simple induction argument, we get, for $l=0,\dots ,N$,
\begin{equation}\label{gamma r bm}
\bm{\Gamma}_n^\pm(\bm{a}_N)\geq \bm{R}_l^\pm(a_l)\;.
\end{equation}

\begin{lem}\label{l: rnpm gamma}
Let $n\in \N$, and let $\bm{a}=(a_0,\dots, a_n)$ be an $(n+1)$-tuple such that, for $0\leq m\leq n$, we have $a_m\leq \bar{r}_m$. Then  
\[
a_n s_n \geq 2\bm{\Gamma}^-_n(\bm{a}_n)+\epsilon_n\;, \quad
a_n s^2_n \geq 2\bm{\Gamma}^+_n(\bm{a}_n)+\epsilon_n\;.
\]
\end{lem}

\begin{proof}
By definition of $s_n$, we have 
\begin{align*}
a_n s_n &= a_n (10\bm{L}_{n-1}(\bar{\bm{r}}_{n-1})+2\bm{\Gamma}^+_{n-1}(\bar{\bm{r}}_{n-1})+\epsilon_n) \\
&> 10 a_n \bm{L}_{n-1}(\bar{\bm{r}}_{n-1})+2\bm{\Gamma}^+_{n-1}(\bar{\bm{r}}_{n-1})+\epsilon_n\;.
\end{align*}
On the other hand, using~\eqref{defn functions n} and the fact that $\bm{L}_{n-1}$ and $\bm{\Gamma}_{n-1}^\pm$ are monotone increasing functions on every coordinate, we have 
\[
\bm{\Gamma}_n^\pm(\bm{a}_n)\leq  
\bm{R}_n^\pm(a_n) \cdot \bm{L}_{n-1}(\bar{\bm{r}}_{n-1}) + \bm{\Gamma}_{n-1}^+(\bar{\bm{r}}_{n-1})\;. 
\]
Then the proof follows by showing that $10a_n >2\bm{R}_n^-(a_n)$ and $10a_ns_n >2\bm{R}_n^+(a_n)$, which is an easy consequence of the definitions.
\end{proof}

Let $\overline{\mathbf K}_{-1}={\mathbf K}_{-1}\equiv\overline{K}_{-1}=K_{-1}= 0$, and continue defining $\overline{\mathbf K}_n$ and ${\mathbf K}_n$ by induction on $n\in\N$ as follows:
\begin{align}
\label{defn bmkn} \overline{\bm{K}}_n(\bm{a}_n)&=\bm{K}_{n-1}(\bm{a}_{n-1}) + \bm{L}_n(\bm{a}_{n}) ( a_ns_n^2 + a_n(2s_n+1))\;,\\  \label{defn kn} \bm{K}_n(\bm{a}_n)&= \overline{\bm{K}}_{n}(\bm{a}_n) + 
\bm{L}_n (\bm{a}_n)(s_{n+1}\bm{R}_{n+1}^+(\bar{r}_{n+1}) + \bm{\Gamma}_n^+(\bar{\bm{r}}_n) + 2\bm{R}_n^+(\bar{r}_{n}))\;.
\end{align}
Finally, for all $n\in \N$, let 
\[
\bar{r}_n^-=\bar{r}_n\;, \quad \bar{r}_n^+=s_n\bar{r}_n\;, \quad \hat{r}_n^-=\hat{r}_n\;, \quad \hat{r}_n^+=s_n\hat{r}_n\;.
\]

\section{Construction of $\xxn$}\label{s. repetitivity}

This section is devoted to the construction of subsets $\xxn\subset X$, which will be used later to achieve the repetitiveness of $\phi$ under the assumptions of \Cref{t: finitary}~\ref{i: finitary repetitive}. Most of the steps will be direct applications of the results from~\cite[Sections~3 and~4]{AlvarezBarral-realization}. 

Suppose that $X$ satisfies the hypothesis of \Cref{t: finitary}~\ref{i: finitary repetitive} throughout this section. 
Then we fix a distinguished point $p\in \xset$, and there are sequences $\mathfrak{r}_n\uparrow\infty$ and $\omega_n$ in $\Z^+$, with every $\mathfrak{r}_n$ large enough depending on $\Delta$ and $\epsilon_m$ for $m\leq n$, so that every set
\[
\Omega_n = \{\, x\in \xset\mid [D (p,\mathfrak{r}_n),p,d_{\xset}]=[ D (x, \mathfrak{r}_n),x,d_{\xset}] \,\}
\] 
is $\omega_n$-relatively dense in $\xset$. Thus there is a pointed isometry $\mathfrak{f}_{n,x}:(D(p,\mathfrak{r}_n),p)\to(D(x,\mathfrak{r}_n),x)$ for every $x\in\Omega_n$. 

By taking a subsequence of $\rrn$ if needed, we can assume that there are other sequences $\ssn,\ttn\uparrow \infty$ in $Z^+$ such that, taking $\mathfrak{r}_{-1}=\mathfrak{s}_{-1} =\mathfrak{t}_{-1}= \omega_{-1}=0$,  
\begin{align}
\label{rn}\mathfrak{r}_n &> \bm{K}_n(\bar{\bm{r}}_{n})+  
s_n^2 4\bm{L}_n(\bar{\bm{r}}_n) (\bm{\Gamma}_n^+(\bar{\bm{r}}_{n})+n)\;,\\
\label{2rl + sn < sl}  \mathfrak{s}_n &>   \bm{L}_{n-1}(\bar{\bm{r}}_{n-1})(2\mathfrak{r}_n
+\bm{K}_{n-1}(\bar{\bm{r}}_{n-1})), 3\bm{L}_n(\bar{\bm{r}}_n)\bm{\Gamma}_{n+1}^+(\bar{\bm{r}}_{n+1})\;,\\
\label{tn} \ttn &> \bm{K}_n(\bar{\bm{r}}_{n})\;,
\end{align}
in addition to the following conditions for some sequence $\lambda_n\downarrow1$ in $\R$ with $\prod_{n=0}^\infty\lambda_n<2$ \cite[Eqs.~(3.1)--(3.6)]{AlvarezBarral-realization}: 
\begin{align*}
\rrn &> \frac{\lambda_0^5}{\lambda_0-1}(\mathfrak{r}_{n-1} +\mathfrak{s}_{n-1}+t_{i-1}+2\omega_{n-1}+1)\;,\\
\ssn &> 2\lambda_0^5(\rrn+ \mathfrak{s}_{n-1}+\omega_n)\;,\\
\ttn &>  \lambda_0^3(5\mathfrak{t}_{n-1} + \rrn +\mathfrak{s}_{n-1} + 2\omega_{n-1}+1)\;, \\
\ttn &> 4\frac{\lambda_n^4 + \lambda_n^2  -1}{\lambda_n^2} \rrn + \mathfrak{t}_{n-1} + \prod_{i=0}^n\lambda_i(\mathfrak{s}_{n-1} + 2\omega_{n-1}+\omega_n)\;,\\
\lambda_n^2 &< \lambda_{n-1} \;,\\
2^{2^{-n}} &> \frac{\rrn (\lambda_n^5-1)\lambda_{n-1}^2}{\mathfrak{r}_{n-1}(\lambda_{n-1}^5-1)\lambda_n^2},
\frac{\rrn (\lambda_n^6-1)\lambda_{n-1}^2}{\mathfrak{r}_{n-1}(\lambda_{n-1}^6-1)\lambda_n^2} \;.
\end{align*}

For $n\in \N$, let $\mathfrak X^n_n=\{p\}$ and $\mathfrak{h}_{n,p}^n= \id_{D(x,\mathfrak{r}_n)}$. In \Cref{p: zmn} we will continue defining  subsets $\mathfrak X^m_n \subset X$ for $0\leq n <m $, and pointed isometries $\ffmnz \colon (D(p,\rrn),p)  \to (D(z,\rrn),z)$ for $z\in \mathfrak X^m_n$. We will use the following notation: 
\[
\ppmn = \{ \, (l,z)\in \mathbb{N}\times \xset \mid n<l<m, \ z\in \mathfrak X^m_l \, \}\;. 
\]

Let $<$ denote the binary relation on $\ppmn$ defined by declaring $(l,z)< (l',z')$ if $l<l'$ and $z\in \mathfrak{h}^m_{l',z'}(\mathfrak{X}^{l'}_l)$, and let $\leq$ denote its reflexive closure. This is actually a partial order relation \cite[Section~4]{AlvarezBarral-realization}. Let $\overline{\mathfrak{P}}^m_n$ denote the subset of maximal elements of $\mathfrak{P}^m_n$. For every $(l,z)\in \mathfrak{P}^m_n$, there is a unique $(l',z')\in \overline{\mathfrak{P}}^m_n$ such that $(l,z)\leq (l',z')$ \cite[Section~4]{AlvarezBarral-realization}.

\begin{prop}[{See \cite[Proposition~4.1 and Remark~6]{AlvarezBarral-realization}}]\label{p: zmn}
 For integers  $0\leq n <m$, there is a set  $\mathfrak X^m_n\subset X$ and a pointed isometry $\mathfrak{h}^m_{n,z}\colon (D(p,\mathfrak{r}_n),p)\to (D(z,\mathfrak{r}_n),z)$ for every  $z\in \mathfrak X^m_n$ satisfying the following properties:
\begin{enumerate}[(i)]
\item \label{i: zmn zvq}The set $\mathfrak X^m_n$ is an $\mathfrak{s}_n$-separated subset of $\Omega_{n}\cap D(p,\rrm -\ttn)$.
\item \label{i: zmn f=ff} For every $ (l,z)\in \ppmn$ and $x\in \mathfrak{X}^m_n\cap D(z,\rrl)$, we have $\mathfrak{h}^m_{n,x}=\mathfrak{h}^m_{l,z} \mathfrak{h}^l_{n,x'}$ on $D (p, \rrn)$, where $x'=(\mathfrak{h}^m_{l,z})^{-1}(x)$.
\item \label{i: zmn cap bln}For any $ (l,z)\in \ppmn$, we have  $\zzmn \cap D(z,\mathfrak{r}_{l'}+\ssn)=\mathfrak{h}^m_{l,z}(\zzln)$.
\item \label{i: zmn either} For any $x\in \mathfrak X^m_n$ and $(l,z)\in \ppmn$, either $d(x,z)\geq \rrl +\ssn$, or $x\in\mathfrak{h}^m_{l,z}(\mathfrak{X}^l_n)$.
\item \label{i: zmn zlk} Consider integers $0\leq k\leq l$ such that either $l< m$ and $k\geq n$, or $l=m$ and $k>n$. Then  $\zzlk\subset \zzmn$, and $\mathfrak{h}^m_{n,z}=\mathfrak{h}^{l}_{k,z}|_{D(p,\rrn)}$ for all $z\in\zzlk$.
\item \label{i: zmn p} We have $p\in \zzmn$ and $\mathfrak{h}^m_{n,p}=\id_{D(p,\rrn)}$.
\end{enumerate}
\end{prop}

For $n<m$, let $\mathfrak{c}^m_n \colon D(p,\mathfrak{r}_m) \to \{n+1, \ldots, m \}$ be defined by
\[
\mathfrak{c}^m_n(x)=\min\{\,n\in\Z \mid n< l \leq m,\ \exists z\in \zzml,\ x\in D(z, \rrl)\,\}\;.
\]

 Since the set $\zzml$ is $2\rrl$-separated by \Cref{p: zmn}~\ref{i: zmn zvq} and~\eqref{2rl + sn < sl}, if $x\in D(z,\rrl)$ for some $z\in \zzml$, then $z$ is the unique point in $\zzml$ that satisfies this condition. Let $\mathfrak{p}^m_n\colon \zzmn \to \xset$ be defined by assigning to every $x\in \zzmn$ the unique point $\mathfrak{p}^m_n(x)$ in $\mathfrak{X}^m_{\mathfrak{c}^m_n(x)}$  satisfying $x\in D(\mathfrak{p}^m_n(x), \mathfrak{r}_{\mathfrak{c}^m_n(x)})$.

For $n\in \N$, let $\preceq^n_n$ be the trivial order relation on $\zznn = \{p\}$.

\begin{prop}\label{p: preceqmn}
For integers  $0\leq n < m $, there is an order\footnote{In order relations, it is assumed that any pair of elements is comparable. Otherwise  we use the term \emph{partial order relation}.} relation $\preceq^m_n$ on $\zzmn$ such that:
\begin{enumerate}[(i)]

\item \label{i: preceqmn p} $p$ is the least element of $(\zzmn,\preceq^m_n)$;

\item \label{i: preceqmn cmnx} for $x,y\in\zzmn$, if $\mathfrak{c}^m_n(x)< \mathfrak{c}^m_n(y)$, then\footnote{This means that $x\preceq^m_n y$ and $x\neq y$. Similar notation is used with other order relations.}  $x \prec^m_n y$; and,

\item \label{i: preceqmn pmn}for any $(l,z)\in \ppmn$, the map $\ffmlz\colon (\zzln , \preceq^l_n ) \to (\zzmn \cap D(z, \rrl) , \preceq^m_n )$ is order preserving.

\end{enumerate}
\end{prop}

\begin{proof}
We proceed by induction on $m$. Let $\preceq^{n+1}_n$ be an arbitrary ordering of $\zzsnn$ whose least element is $p$. For $m=n+1$, we have $\mathfrak{c}^m_n(x)= m$ for every $x\in \zzmn$ if $\ppmn=\emptyset$.  Thus~\ref{i: preceqmn cmnx} and~\ref{i: preceqmn pmn} are trivially satisfied  in this case. 

Suppose now that we have defined $\preceq^l_k$ when either $l>n$, or $l=n$ and $k<m$. Let $\trianglelefteq^m_n$ be an arbitrary ordering of $D(p,\rrm)\setminus \bigcup_{(l,z)\in \ppmn} D(z,\rrl)$. Then we define $\preceq^m_n$ using several cases as follows: 
\begin{enumerate}[(a)]
\item if $\mathfrak{c}^m_n(x)< \mathfrak{c}^m_n(y)$, then $x\prec^m_n y$;
\item if $\mathfrak{c}^m_n(x)= \mathfrak{c}^m_n(y) <m$ and $\mathfrak{p}^m_n(x)=\mathfrak{p}^m_n(y)$, then $x\preceq^m_n y$ if and only if
 \[
 \big(\mathfrak{h}^m_{\mathfrak{c}^m_n(x), \mathfrak{p}^m_n(x)}\big)^{-1}(x)
 \preceq^{\mathfrak{c}^m_n(x)}_{n}\big(\mathfrak{h}^m_{\mathfrak{c}^m_n(x), \mathfrak{p}^m_n(x)}\big)^{-1}(y)\;;
 \]
\item if $\mathfrak{c}^m_n(x)= \mathfrak{c}^m_n(y) <m$ and $\mathfrak{p}^m_n(x)\neq \mathfrak{p}^m_n(y)$, then $x\prec^m_n y$ if and only if $\mathfrak{p}^m_n(x)\prec^m_{\mathfrak{c}^m_n(x)}\mathfrak{p}^m_n(y)$; and,
\item if $\mathfrak{c}^m_n(x)= \mathfrak{c}^m_n(y) =m$, then $x\preceq^m_n y $ if and only if $x\trianglelefteq^m_n y$.
\end{enumerate}

It can be easily checked that this is indeed an order relation, and it is obvious that it satisfies~\ref{i: preceqmn p} and~\ref{i: preceqmn cmnx}. Let us prove that it also satisfies~\ref{i: preceqmn pmn}. Suppose first that $(l,z)\in \olppmn$. For any $x,y\in D(z,\rrl)$, we have $\mathfrak{c}^m_n(x)=\mathfrak{c}^m_n(y)= l$ and $\mathfrak{p}^m_n(x)=\mathfrak{p}^m_n(y)= z$, and therefore $\mathfrak{h}^m_{l,z}$ is order preserving by (b). 

Suppose now that $(l,z)\in \ppmn \setminus \olppmn$. Let $(l',z')\in \olppmn$ be the unique maximal element such that $(l,z)<(l',z')$ and let $z''=(\mathfrak{h}^m_{l',z'})^{-1}(z)$. By the induction hypothesis, the map
\[
\mathfrak{h}^m_{l,z''}\colon \big(\zzln , \preceq^l_n\big) \to \big( \zzlpin\cap D(z'', \rrl) , \preceq^{l'}_n\big)
\]
is order preserving, and 
\[
\mathfrak{h}^m_{l',z'}\colon \big(\mathfrak{X}^{l'}_n, \preceq^{l'}_n\big)\to \big(\zzmn \cap  D(z',\mathfrak{r}_{l'}),\preceq^m_n\big)
 \]
 is order preserving because $(l',z')\in \olppmn$. Therefore
 \[
\ffmlz = \mathfrak{h}^m_{l',z'}  \mathfrak{h}^m_{l,z''}\colon \big( \zzln , \preceq^l_n \big)\to \big(\zzmn \cap D(z,\rrl), \preceq^m_n\big)
\]
is also order preserving. 
\end{proof}

Define 
\begin{align}
\notag \mathfrak{X}_n&= \bigcup_{m\geq n} \mathfrak{X}^m_n\;,\\
\label{rrn} \Rn &= \bigcup_{m\geq n} \ppmn= \{\,(m,x)\in \mathbb{N}\times X \mid n<m,\, x\in \xxm \,\}\;.
\end{align}
For $n\in \N$ and $ x\in \xxn$, there is some $m\geq n$ such that $x\in \zzmn$. 
Let $\mathfrak{h}_{n,x}= \mathfrak{h}^m_{n,x}\colon (D(p,\mathfrak{r}_n),p)\to (D(x,\mathfrak{r}_n),x)$, which is independent of $m$ by \Cref{p: zmn}~\ref{i: zmn zlk}.

Let $<$ be the binary relation on $\Rn$ defined by declaring  $(m,x)<(m',x')$ if  $m<m'$ and $D(x,\rrm )\subset D(x', \mathfrak{r}_{m'})$, and let $\leq$ be the reflexive closure of $<$.

Consider the choice of $\rrn$, $\ssn$ and $\ttn$ given at the beginning of the present section.

\begin{prop}[{See \cite[Proposition~4.2 and Remark~6]{AlvarezBarral-realization}}]\label{p: xxn}
For $n\in \N$, the following properties hold: 
\begin{enumerate}[(i)]

\item \label{i: xxn subset}The set $\xxn$ is an $\mathfrak{s}_n$-separated subset of $\Omega_n$ containing $p$.

\item \label{i: xxn cap bmn}For any $(l,z)\in \Rn$,  we have $\xxn\cap D(z,\rrl)=\mathfrak{h}_{l,z}( \mathfrak{X}^l_n )$.

\item \label{i: xxn h=hh}For any $x\in \xxn$ and $(l,z)\in \Rn$ so that $x\in \mathfrak{X}_n\cap D(z,\mathfrak{r}_l)$, we have $\mathfrak{h}_{n,x}= \mathfrak{h}_{l,z}\mathfrak{h}_{n,x'}$ for $x'=\mathfrak{h}_{l,z}^{-1}(x)$.

\item \label{i: xxn either}  For any $ x\in \xxn$ and $(l,z)\in \Rn$, either $d(x,z)\geq \rrl+\ssn$, or $x\in h_{l,x}(\mathfrak{X}^l_n)$.

\item \label{i: xxn xxm} For $n\leq m$,  we have  $\xxm\subset \xxn$, and $\mathfrak{h}_{n,x}= \mathfrak{h}_{m,x}|_{D(p,\mathfrak{r}_n)}$ for $x\in \xxm$.

\item \label{i: xxn p} We have $p\in \xxn$ and $\mathfrak{h}_{n,p}=\id_{D(p,\mathfrak r_n)}$.

\end{enumerate}
\end{prop}

\begin{rem}
Suppose that we replace the family $\Omega_n$ with  arbitrary, relatively dense, proper subsets 
\[
\Omega_n \subset \{\, x\in \xset\mid [D (p,\mathfrak{r}_n),p,d_{\xset}]=[ D (x, \mathfrak{r}_n),x,d_{\xset}] \,\}
\]
and, for every $x\in \Omega_n$, we choose a pointed isometry $f_{n,x}\colon (D(x,\mathfrak{r}_n),p )\to (D(x,\mathfrak{r}_n),p )$. Then, according to~\cite[Remark~4]{AlvarezBarral-realization}, we may assume that $\xxn\subset \Omega_n$ and every map $\mathfrak{h}_{n,x}$ is a composition of the form $f_{n_m,x_m}\cdots f_{n_1,x_1}$. Note that the constants $\omega_n$ may change with this assumption.
\end{rem}

\begin{prop}[{See \cite[Proposition~4.3 and Remark~6]{AlvarezBarral-realization}}]\label{p: mathfrak X_n is a net}
The subset $\mathfrak{X}_n$ is relatively dense in $X$, and the implied constant depends only on $\Delta$, $\epsilon_m$, and $\omega_m$ for $m\leq n$, and $\rrm$ for $m<n$.
\end{prop}

By Propositions~\ref{p: zmn}~\ref{i: zmn p} and~\ref{p: preceqmn}~\ref{i: preceqmn pmn}, the order relations $\preceq^m_n$, $m\geq n$, define an order relation $\preceq_n$ on $\mathfrak{X}_n$. The following is a consequence of \Cref{p: preceqmn}.

\begin{prop}
For $n\in \N$, the following properties hold:
\begin{enumerate}[{\rm(}i\/{\rm)}]
\item The point $p$ is the least element of $(\mathfrak{X}_n,\preceq_n)$.
\item For $x,y\in\mathfrak X_n$, if $\mathfrak{c}_n(x)< \mathfrak{c}_n(y)$, then $x \prec_n y$.
\item For any $(l,z)\in \Rn$, the map $\mathfrak{h}_{l,z} \colon (\zzln , \preceq^l_n ) \to (\xxn \cap D(z, \rrl) , \preceq_n  )$ is order preserving.
\end{enumerate}
\end{prop}

Let 
\begin{align}
\nonumber \ppmone&=\{\, (l,z)\in \mathbb{N}\times X  \mid 0\leq l<m, \ z\in \zzml \,\}\quad(m\in\N)\;,\\
\label{defn rr-1} \Rone&= \{ \, (m,x)\in \mathbb{N}\times X \mid  x\in \xxm  \,  \}\;.
\end{align}
We can define on both of these sets the relation $<$ by declaring $(l,z)< (l',z')$ if $l<l'$ and $D(z,\rrl)\subset D(z',\mathfrak{r}_{l'})$. The induced reflexive closures $\leq$ are partial order relations. Let $\overline{\mathfrak{P}}^m_{-1}$ denote the subset of maximal elements of $\mathfrak{P}^m_{-1}$. For every $(l,z)\in \mathfrak{P}^m_{-1}$, there is a unique $(l',z')\in \overline{\mathfrak{P}}^m_{-1}$ such that $(l,z)\leq (l',z')$.

\section{Construction of $X_n$}\label{s. xn}

In this section we define a sequence of nested subsets $X_n\subset X$ that will constitute the centers of the clusters used in the construction of the colorings $\phi^N$, as explained in \Cref{s. idea}. This will be used to prove \Cref{t: finitary} in full generality. So we assume that $X$ satisfies the hypothesis of \Cref{t: finitary}~\ref{i: finitary repetitive}. If $X$ only satisfies the hypothesis of \Cref{t: finitary}~\ref{i: finitary limit aperiodic}, the same proof applies to \Cref{t: finitary}~\ref{i: finitary limit aperiodic} by taking  $\mathfrak{X}_n=\emptyset$, and therefore omitting the use of the sets $\mathfrak{P}_n$, numbers $\rrn$, and maps $\mathfrak{h}^m_{n,z}$ and $\mathfrak{h}_{n,x}$; a choice of basepoint $p$ is still needed, however.

For notational convenience, let  
\begin{equation}\label{constants -1}
(X_{-1},E_{-1} )= (X,E)\;, \quad d_{-1}=d\;, \quad r_{-1}=s_{-1}=R_{-1}^\pm=0\;, \quad l_{-1}=L_{-1}=1\;.
\end{equation}
By induction on $n\in\N$, we will continue defining constants $r_n$, subsets $X_n\subset X$ containing $p$ and $\mathfrak X_n$, and a connected graph structure $E_n$ on every $X_n$ with induced metric $d_n$. 
The subindex ``$n$'' will be added to the notation of disks, spheres, closed penumbras and coronas in $(X_n,d_n)$. 
With this notation, let $\eta_n\colon \N\to \mathbb{Q}$ be given by 
\begin{align}\label{etan defn}
\eta_n(a) = 
\begin{cases}
\exp_2\big(\big\lfloor\big(a - (\deg X_{-1})^{11} -1\big)/(\deg X_{-1})^3\big\rfloor\big) &\text{if $n=0$}\\[4pt] 
\exp_2\big(\big\lfloor\big(a - (\deg X_{n-1})^{11} -1\big)/(\deg X_{n-2})^{r_{n-1}^2s_{n-1}}\big\rfloor\big) &\text{if $n>0$}\;.
\end{cases}
\end{align}

Given $n\in \N$, suppose  that the graphs $(X_m,E_m)$ and constants $r_m$ have been defined for integers $-1\leq m <n$. 
Then let $r_n$ be defined  as follows: 
\begin{enumerate}[(A)]
\item \label{i: rn bar} If there is some $x\in D_{n-1}(p,\hat{r}_n(2s_n+1))$ such that 
\[
(|D_{n-1}(x,\hat{r}_ns_n)|+6)^2\geq \eta_n(|D_{n-1}(x,\hat{r}_n)|)\;,
\]
then let $r_n=\bar{r}_n$ (see~\eqref{defn barrn}). Note that $p\in X_{n-1}$.

\item \label{i: rn hat} Otherwise, let $r_n=\hat{r}_n$ (see~\eqref{tilder0 defn} and~\eqref{hatrn}).
\end{enumerate}
Observe that
\begin{equation}\label{r_0 > 2^11}
r_0>2^{11}
\end{equation}
by~\eqref{hat r_0 > 2^11},~\eqref{defn barr0},~\ref{i: rn bar} and~\ref{i: rn hat}.
Moreover, let
\begin{equation}\label{no bm multi}
\left.\begin{gathered}
\begin{alignedat}{2}
\Delta_n&=\bm\Delta_n(r_0,\dots,r_n)\;,&\quad 
L_n&=\bm{L}_n(r_0,\dots,r_n)\;,\\ 
\Gamma^\pm_n&=\bm\Gamma^\pm_n(r_0,\dots,r_n)\;,&\quad 
\overline{K}_n&=\bm{\overline{K}}_n(r_0,\dots,r_n)\;,
\end{alignedat}\\ 
K_n=\bm{K}_n(r_0,\dots,r_n)\;, \quad R_n^\pm = \bm R^\pm (r_n)\; , \quad l_n=\bm l_n(r_n)\;.
\end{gathered}\;\right\}
\end{equation} 
All functions in~\eqref{no bm multi} are monotone increasing on every coordinate (see \Cref{s. constants}). 
So, using the notation $\bm{\hat{r}_n}=(\hat{r}_0,\dots,\hat{r}_n)$, we get 
\begin{equation}\label{functions increasing}
\bm\Delta_n(\hat{\bm{r}}_n)\leq \Delta_n\leq \bm\Delta_n(\bar{\bm{r}}_n)\;, \quad 
\bm R_n^\pm(\hat{r}_n)\leq R_n^\pm\leq \bm R_n^\pm(\bar{r}_n)\;, \quad \bm\Gamma^\pm_n(\hat{\bm{r}}_n)\le \Gamma_n^\pm\le \bm\Gamma^\pm_n(\bar{\bm{r}}_n)\;.
\end{equation}
From~\eqref{gamma r bm},~\eqref{rn} and~\eqref{no bm multi}, it follows that
\begin{equation}\label{gamma r}
\mathfrak{r}_n>\Gamma_n^\pm \geq R_m^\pm
\end{equation}
for $m=0,\dots,n$.
Finally, let 
\begin{equation}\label{defn rnpm}
r_n^-=r_n\;, \quad r_n^+=r_ns_n\;.
\end{equation}
By~\eqref{defn functions 0},~\eqref{defn functions n},~\eqref{no bm multi} and~\eqref{defn rnpm}, we have
\begin{equation}\label{r_n^pm le R_n^pm}
r_n^\pm\le R_n^\pm\;.
\end{equation}

\begin{prop}\label{p: xn}
For $n\in \N$, there are disjoint subsets $X_n^+,X_n^-\subset X$ and a graph structure $E_n$ on $\xn:=X_n^-\cupdot X_n^+$  such that the following properties are satisfied:
\begin{enumerate}[(i)]

\item \label{i: xn subset} $\xxn \subset \xn \subset X_{n-1}$.

\item \label{i: xn rn-1} For all $(m,x)\in \mathfrak{P}_{n-1}$, we have
\[
\hhmx \left( \xn^\pm\cap D_{-1}(p, \mathfrak{r}_m-\overline{K}_{n}) \right)
= \xn^\pm\cap D_{-1}(x,\mathfrak{r}_m-\overline{K}_{n})\;.
\]
\item \label{i: xn1} For all $x\in \xn^\pm$, we have
\[
\eta_n( |D_{n-1}(x,r_n^\pm )|)\geq (6+|D_{n-1}(x,r_n^\pm s_n)|)^2\;.
\]

\item \label{i: xn separated net} $\xn$ is $(2r_n^+ +1)$-separated and $R_n^+$-relatively dense in $(\xyn, d_{n-1})$.

\item \label{i: xn connected} $(X_{n},E_n)$ is a connected graph. 
Let $d_n$ denote the induced metric.

\item \label{i: xn dn dn-1} We have $d_{n}\leq d_{n-1}\leq l_nd_n$ and $d_n \leq d_{-1}\leq L_nd_n$.

\item \label{i: xn deltan deltan-1} We have
\[
\deg X_n\leq \, \Delta_n, \, 4(\deg X_{n-1}-1)^{2R_n^+}\;.
\]

\item \label{i: xn partial} For any $(m,x)\in \mathfrak{P}_{n-1}$, the restriction of  $\hhmx$ to  $\xn\cap D_{-1}(p, \mathfrak{r}_m-K_n)$ is  an $(s_{n+1}R_{n+1}^+ + \Gamma_n^+)$-short scale isometry with respect to $d_n$.

\end{enumerate}
\end{prop}

\begin{rem}
Note that $K_n<\ttn,\rrn$ by~\eqref{rn},~\eqref{tn} and the fact that $\bar r_n>r_n$. This and the inequality $K_n>\overline K_n$ yield $\rrm-\overline{K}_{n},\rrm-K_n>0$ in~\ref{i: xn rn-1} and~\ref{i: xn partial}. 
\end{rem}

\begin{rem}
In accordance with the discussion at the beginning of the section, to prove \Cref{t: finitary}, if $X$ does not satisfy the hypothesis of \Cref{t: finitary}~\ref{i: finitary repetitive}, items~\ref{i: xn rn-1} and~\ref{i: xn partial} must be omitted, and only the inclusion ``$\xn \subset X_{n-1}$" must be considered in~\ref{i: xn subset}.
\end{rem}

The rest of this section is devoted to prove \Cref{p: xn}.
We proceed by induction on $n$.
The following lemma follows from \Cref{p: xxn},~\eqref{constants -1} and~\eqref{defn rr-1}.
The items are irregularly numbered so that there is an obvious correspondence with those of \Cref{p: xn}.
\begin{lem}\label{l: xn induction}
The following properties hold:
\begin{enumerate}[(i')]

\item \label{i: xn subset prime} $\mathfrak{X}_0\subset X_{-1}$.

\item For all $(m,x)\in \mathfrak{P}_{-1}$, we have
 \[
  \hhmx \left(X_{-1}\cap D_{-1}(p, \mathfrak{r}_m) \right)= X_{-1}\cap D_{-1}(x,\mathfrak{r}_m)\;.
  \]
\addtocounter{enumi}{1}
\item $X_{-1}$ is $(2r_{-1}s_{-1}+1)$-separated and $R_{-1}^+$-relatively dense in $X$.

\item $(X_{-1},E_{-1})$ is a connected graph.

\item We have $d_{-1}=d=l_{-1} d_{-1}=L_{-1} d_{-1}$.

\item  We have $\deg X_{-1}=\Delta_{-1}$.

\item For any $(m,x)\in \mathfrak{P}_{-1}$, the restriction of  $\hhmx$ to $\xn\cap D_{-1}(p, \mathfrak{r}_m)$ is  an $(s_{0}R_{0}^+ + \Gamma_0^+)$-short scale isometry with respect to $d_{-1}$.
\end{enumerate}
\end{lem}
This lemma can be considered the extension to $n=-1$ of properties~\ref{i: xn subset},~\ref{i: xn rn-1} and \ref{i: xn separated net}--\ref{i: xn partial} of \Cref{p: xn}.
In this way, we include the case $n=0$ in the induction step.
Thus suppose that, given  $n\geq 0$, we have already defined $X_m$, $E_m$, $d_m$ and $r_m$ for $m<n$, satisfying all required properties.
When we invoke the induction hypothesis with some item, say  e.g.~\ref{i: xn subset}, it will refer to \Cref{l: xn induction}~\ref{i: xn subset prime} if $n=0$, and to \Cref{p: xn}~\ref{i: xn subset} if $n>0$.

By~\eqref{functions increasing}, we have $\Delta_{n-1}\leq \bm\Delta_{n-1}(\bar{\bm{r}}_{n-1})$.
From this inequality and the definitions of $\eta_n$ and $\overline{\eta}_n$ in \eqref{dfn overlineeta} and \eqref{etan defn}, we obtain, for $a\in \N$,
\begin{equation}\label{etan oletan}
\eta_n(a)\geq \overline{\eta}_n(a)\;.
\end{equation}

Let $\hat{\mathfrak{c}}_n \colon \xyn\to \{n,n+1, \dots \}$ be defined by 
\begin{equation}\label{defn hatc}
\hat{\mathfrak{c}}_n(x)=\min\{\,l \geq n \mid \exists y\in\xxl\ \text{so that}\ (l,y)\in\Ryn\ 
\text{and}\ x\in D_{-1}(y,  \rrl - K_{n-1} )\,\}\;.
\end{equation} 
This map is well-defined because $\mathfrak{r}_l\to \infty$ as $l\to \infty$ by \eqref{rn} and \eqref{tn}. 
By \Cref{p: xxn}~\ref{i: xxn subset}, for each $x\in \xyn$,  there is a unique point $\hat{\mathfrak{p}}_n(x)\in \mathfrak{X}_{\hat{\mathfrak{c}}_n(x)}$ 
such that $x\in D_{n-1}(\hat{\mathfrak{p}}_n(x),\mathfrak r_{\hat{\mathfrak c}_n(x)} - K_{n-1} )$. 
This defines a map $\hat{\mathfrak{p}}_n\colon\hat{\mathfrak{c}}_n^{-1}( \{n,n+1, \dots \}) \to \xxn$.

\begin{lem}\label{l: ymn}
For integers $m\geq n\ge0$, there are ordered sets $(\Ymn, \leq_n^m)$ so that the following properties hold:
\begin{enumerate}[(a)]

\item \label{i: ymn separated} $\Ymn$ is a maximal $2r_n$-separated subset of $(D_{-1}(p,\mathfrak{r}_m-K_{n-1})\cap X_{n-1},d_{n-1})$ containing $p$.

\item\label{i: ymn subset} If $m>n$, then $\Yymn\subset \Ymn$, and the map $(\Yymn, \leq_n^{m-1} )\hookrightarrow (\Ymn,\leq_n^m)$ is order-preserving.

\item\label{i: ymn fmlz} For any $(l,z)\in \ppmyn$, we have $\mathfrak{h}^m_{l,z}(\Yln ) = \Ymn \cap D_{-1}(z,\mathfrak{r}_l-K_{n-1} )$, 
and the map
\[
\mathfrak{h}^m_{l,z}\colon\big( \Yln, \leq^l_n\big) \to ( \Ymn \cap D_{-1}(z,\mathfrak{r}_l-K_{n-1}), \leq^m_n)
\]
is order-preserving.

\item\label{i: ymn conditions} For all $x,y\in Y^m_n$, we have $x<^m_n y$ if one  of the following conditions holds:
\begin{enumerate}[(I)]

\item\label{i: ymn cn < cn} $\hat{\mathfrak{c}}_n(x)<\hat{\mathfrak{c}}_n(y)$; 

\item\label{i: ymn pn neq pn} $\hat{\mathfrak{c}}_n(x)=\hat{\mathfrak{c}}_n(y)$ and $d_{-1}(\hat{\mathfrak{p}}_n(x),p)<d_{-1}( \hat{\mathfrak{p}}_n(y),p)$; or

\item\label{i: ymn cn = cn} $\hat{\mathfrak{c}}_n(x)= \hat{\mathfrak{c}}_n(y)$, $\hat{\mathfrak{p}}_n(x)= \hat{\mathfrak{p}}_n(y)$ and $d_{-1}(x, \hat{\mathfrak{p}}_n(x))< d_{-1}(y, \hat{\mathfrak{p}}_n(x))$.

\end{enumerate}
\end{enumerate}
\end{lem}

\begin{proof}
We proceed by induction on $m$.
Let $\Ynn$ be any maximal $2r_n$-separated subset  of the metric space  $(D_{-1}(p,\rrn - K_{n-1})\cap X_{n-1},d_{n-1})$   containing $p$. 
Let $\leq_n^n$ be any order relation on $Y_n^n$ such that, if $d_{-1}(x,p)<d_{-1}(y,p)$, then $x<^n_n y$.
Since $\hat{\mathfrak{c}}_n(x)=n$ and $\hat{\mathfrak{p}}_n(x)=p$ for all $x\in Y^n_n$, this relation satisfies the properties of the statement for $m=n$.

Suppose that we have defined $\Yln$ and $\leq^l_n$  for $ n\leq l<m$, satisfying the stated properties.
Let 
\[
\wtYmn= \bigcupdot_{(l,z)\in \olppmyn} \mathfrak{h}^m_{l,z}( \Yln )\;.
\] 
By the induction hypothesis with~\ref{i: xn partial}, for every $(l,z)\in \olppmyn$, the set $\mathfrak{h}_{l,z}( \Yln )=\mathfrak{h}^m_{l,z}( \Yln )$ is contained in $\xyn$ and is $2r_n$-separated with respect to $d_{n-1}$. 
Arguing like in the proof of \Cref{p: zmn}~\ref{i: zmn zvq}, we get that $\wtYmn$ is a maximal $2r_n$-separated subset of 
\[
\bigcupdot_{(l,z)\in \olppmyn} D_{-1}(z,\rrl-K_{n-1})\;,
\] 
with respect to $d_{n-1}$, containing $p$.
Now let $\Ymn$ be any maximal $2r_n$-separated subset of the metric space  $(D_{-1}(p,\rrn - K_{n-1})\cap X_{n-1},d_{n-1})$  containing $\wtYmn$; in particular, $\Ymn$ safisfies~\ref{i: ymn separated}.

Let $\widetilde{\leq}^m_n$ be any ordering of $\wtYmn$ satisfying the analogues of~\ref{i: ymn subset},~\ref{i: ymn cn < cn} and~\ref{i: ymn pn neq pn} with $\wtYmn$ instead of $\Ymn$. 
Then, by the induction hypothesis with~\ref{i: ymn cn = cn} and the definition of $\wtYmn$, the order $\widetilde{\leq}^m_n$ also satisfies the analogue of~\ref{i: ymn cn = cn}. 
Let $\widehat{\leq}^m_n$ be any ordering of $\widehat Y^m_n:=\Ymn\setminus \wtYmn$ satisfying the analogue of~\ref{i: ymn cn = cn} with $\widehat Y^m_n$ instead of $Y^m_n$. 
Let $\leq^m_n$ be the order relation on $Y^m_n$ defined by $\widetilde{\leq}^m_n$ and $\widehat{\leq}^m_n$ on $\wtYmn$ and $\widehat Y^m_n$, respectively, and satisfying $x\leq^m_n y$ for all $x\in \wtYmn$ and $y\in\widehat Y^m_n$.
It is easy to check that $\leq^m_n$ satisfies  the stated properties.
\end{proof}

Let $\Yn=\bigcup \nolimits_{m\geq n}Y^m_n$. 
Like in the case of the relations $\preceq^m_n$ (\Cref{s. repetitivity}), the order relations $\leq^m_n$ define an order relation $\leq_n$ on $Y_n$.

\begin{lem}\label{l: yn}
The ordered sets $(\Yn,\leq_n)$ satisfy the following properties:
\begin{enumerate}[(a)]

\item \label{i: yn maximal} $\Yn$ is a maximal $2r_n$-separated subset of $(\xyn,d_{n-1})$ containing $p$, and therefore it is $2r_n$-relatively dense in $(\xyn,d_{n-1})$.

\item \label{i: yn rn-1} For any $(l,z)\in \Ryn$, we have $\mathfrak{h}_{l,z}(\Yln ) = \Yn \cap D_{-1}(z, \mathfrak{r}_l-K_{n-1})$, and the map $$\mathfrak{h}_{l,z}\colon ( \Yln, \leq^l_n) \to ( \Yn \cap D_{-1}(z,\mathfrak{r}_l-K_{n-1}), \leq_n)$$ is order-preserving.

\item \label{i: yn conditions} For all $x,y\in Y_n$, we have $x<_n y$ if one of the following conditions holds:
\begin{enumerate}[(I)]

\item \label{i: yn cn < cn} $\hat{\mathfrak{c}}_n(x)<\hat{\mathfrak{c}}_n(y)$;

\item\label{i: yn pn neq pn} $\hat{\mathfrak{c}}_n(x)=\hat{\mathfrak{c}}_n(y)$ and $d_{-1}(\hat{\mathfrak{p}}_n(x),p)<d_{-1}( \hat{\mathfrak{p}}_n(y),p)$; or

\item \label{i: yn cn = cn} $\hat{\mathfrak{c}}_n(x)= \hat{\mathfrak{c}}_n(y)$, $\hat{\mathfrak{p}}_n(x)= \hat{\mathfrak{p}}_n(y)$ and $d_{-1}(x, \hat{\mathfrak{p}}_n(x))< d_{-1}(y, \hat{\mathfrak{p}}_n(x))$.

\end{enumerate}

\item \label{i: yn well} $(\Yn,\leq_n)$ is well-ordered.

\end{enumerate}
\end{lem}

\begin{proof}
Properties~\ref{i: yn maximal}--\ref{i: yn conditions} follow from \Cref{l: ymn}~\ref{i: ymn separated}--\ref{i: yn conditions} and the definition of $(Y_n,\leq_n)$. So let us prove~\ref{i: yn well}.
By~\ref{i: yn cn < cn}, it is enough to prove that, for each $m\geq n$, the ordered subset $(Y_n\cap \hat{\mathfrak{c}}_n^{-1}(m),\leq_n)$ is well-ordered.
By~\ref{i: yn pn neq pn}, the subsets $\{\, y\in Y_n\cap\hat{\mathfrak{c}}_n^{-1}(m)\mid d_{-1}(\hat{\mathfrak{p}}(y),p)\leq l \, \}$, with $l\in \N$, form an increasing sequence of finite initial segments of $(Y_n\cap \hat{\mathfrak{c}}_n^{-1}(m),\leq_n)$ covering $Y_n\cap \hat{\mathfrak{c}}_n^{-1}(m)$. 
Since
\[
\{\, y\in Y_n\cap\hat{\mathfrak{c}}_n^{-1}(m)\mid d_{-1}(\hat{\mathfrak{p}}(y),p)\leq l \,\}
\subset \CPen_{-1}(Y_n\cap D_{-1}(p,l), \mathfrak{r}_m-K_{n-1} )
\subset D_{-1}(p, l+\mathfrak{r}_m-K_{n-1})\;,
\]
all sets $\{\, y\in Y_n\cap \hat{\mathfrak{c}}_n^{-1}(m)\mid d_{-1}(\hat{\mathfrak{p}}(y),p)\leq l \, \}$ are finite, and therefore well-ordered with $\le_n$. 
It easily follows that $Y_n\cap \hat{\mathfrak{c}}_n^{-1}(m)$ is well-ordered, completing the proof of~\ref{i: yn well}.
\end{proof}

\begin{rem}\label{r. yn xxn}
Note that  $\{n\}\times\mathfrak{X}_n\subset \mathfrak{P}_{n-1}$ by definition. By \Cref{l: yn}~\ref{i: yn maximal},\ref{i: yn rn-1}, for any $x\in \mathfrak{X}_n$, we have $x = \mathfrak{h}_{n,x}(p)\subset \Yn$, yielding $\xxn \subset \Yn$.  
\end{rem}

\begin{rem}\label{r. yn p}
For any $x\in D_{-1}(p,\rrm - K_{n-1})$, we have $\hat{\mathfrak{c}}_n(x)=n$ and $\hat{\mathfrak{p}}_n(x)=p$ by definition. So, by~\ref{i: yn pn neq pn}, $D_{-1}(p,\rrm - K_{n-1})$ is an initial segment of $Y_n$. Therefore  $p$ is the least element of $Y_n$ by~\ref{i: yn cn = cn}. 
\end{rem}

Let now
\begin{align*}
Y^-_n &= \{\, y\in \Yn \mid \eta_n(|D_{n-1}(y,r_n^+ )|) <(6+|D_{n-1}(y,r_n^+s_n )|)^2   \, \}\;,  \\ 
Y^+_n &= \{\, y\in \Yn \mid \eta_n(|D_{n-1}(y,r_n^+ )|)\geq (6+|D_{n-1}(y, r_n^+s_n)|)^2  \, \}\;.
\end{align*}

\begin{lem}\label{l: ballnrnsn}
We have 
\[
y\in D_{-1}(p, \rrl - K_{n-1} - L_{n-1} r_ns_n^2)\Longrightarrow D_{n-1}(y, r_n^+s_n)\subset D_{-1}(p, \rrl - K_{n-1})\;.
\]
\end{lem}
\begin{proof}
By the induction hypothesis with \Cref{p: xn}~\ref{i: xn dn dn-1},  we have
\begin{align*}
d_{-1}(x,p)&\le d_{-1}(x,y)+d_{-1}(y,p)\le L_{n-1}d_{n-1}(x,y)+d_{n-1}(y,p)\\ 
&\le L_{n-1}r_ns_n^2+\rrl - K_{n-1} -L_{n-1} r_ns_n^2=\rrl - K_{n-1}\;.\qedhere
\end{align*}
\end{proof}

\begin{lem}\label{l: h preserve ynpm}
For any $(l,z)\in \mathfrak{P}_{n-1}$ and $y\in \Yn\cap D_{-1}(p, \rrl - K_{n-1} - L_{n-1} r_ns_n^2)$, we have that $y\in \Yn^{\pm}$ if and only if $\mathfrak{h}_{l,z}(y)\in \Yn^{\pm}$.
\end{lem}

\begin{proof}
By \Cref{l: ballnrnsn}, we have  $D_{n-1}(y, r_ns_n^2)\subset D_{-1}(p, \rrl - K_{n-1})\subset \dom\mathfrak{h}_{l,z}$. 
Then $|D_{n-1}(y, r_ns_n^i)|= |D_{n-1}(\mathfrak{h}_{l,z}(y), r_ns_n^i)|$ for $i=1,2$ since  $\mathfrak{h}_{l,z}$ is an  $s_nR_n^+$-short scale isometry on $(D_{-1}(p, \rrl - K_{n-1}),d_{n-1})$.
\end{proof}

Using that $(\Yn,\leq_n)$ is a well-ordered set (\Cref{l: yn}~\ref{i: yn well}), let $\xn^{+}\subset\Yn^{+}$ be  inductively defined as follows:
\begin{itemize}

\item If $y_0$ is the least element of $(Y_n^{+},\leq_n)$, then $y_0\in X_n^{+}$.

\item For  all $y \in Y_n^{+}$ such that $y>_ny_0$, we have $y\in \xn^{+}$ if and only if $d_{n-1}(y,y')>2r_ns_n$ for all $y'\in \xn^{+}$ with $y'<_ny$.

\end{itemize} 

\begin{rem}\label{r. xn+ sep} 
Observe that $\xn^{+}$ is $(2r_ns_n+1)$-separated and $2r_ns_n$-relatively dense in $(Y_n^+,d_{n-1})$.
\end{rem}

\begin{rem}\label{r. yln y cap} 
Note that \Cref{l: yn}~\ref{i: yn rn-1} yields $\Yln = \Yn \cap D_{-1}(p, \mathfrak{r}_l-K_{n-1})$ because $\mathfrak{h}_{l,p}=\id$ by \Cref{p: xxn}~\ref{i: xxn p}.
\end{rem}

\begin{lem}\label{pl. h preserves xn+}
For all $z\in \xxn$ and $y\in \Yn\cap D_{-1}(p, \mathfrak{r}_n-K_{n-1}- L_{n-1} r_ns_n^2)$, we have  $y\in \xn^{+}$ if and only if $\mathfrak{h}_{n,z}(y )\in \xn^{+}$.
\end{lem}

\begin{proof}
By \Cref{l: h preserve ynpm}, it is enough to prove the statement for points $y\in Y_n^+$.
We proceed by induction on the elements of $\Yn^+\cap D_{-1}(p, \mathfrak{r}_n-K_{n-1}- L_{n-1}r_ns_n^2)$ using $\leq_n$.
Let $y_1$ be the least element of  $\Yn^+\cap D_{-1}(p, \mathfrak{r}_n-K_{n-1}- L_{n-1} r_ns_n^2)$.
We first show that $y_1,\mathfrak{h}_{n,z}(y_1)\in \xn^{+}$, establishing the desired property for $y_1$.

By absurdity, suppose that $y_1\notin \xn^{+}$. This means that $y_1>_ny_0$ and there is some $u\in \xn^{+}$ such that $u<_n y_1$ and $d_{n-1}(y_1,u)\leq 2r_ns_n$.
Since $s_n>2$ by~\eqref{defn s0} and~\eqref{defn sn}, it follows from \Cref{l: ballnrnsn} that $u\in D_{-1}(p, \mathfrak{r}_n-K_{n-1})$.
Then  $\hat{\mathfrak{c}}_n(y_1)=\hat{\mathfrak{c}}_n(u)=n$ and $\hat{\mathfrak{p}}_n(y_1)=\hat{\mathfrak{p}}_n(u)=p$.
\Cref{l: yn}~\ref{i: yn cn = cn} and the assumption that $u<_n y_1$ yield $d_{-1}(p,u)\leq d_{-1}(p,y_1)$. 
So, in fact, $u\in D_{-1}(p, \mathfrak{r}_n-K_{n-1}- L_{n-1}r_ns_n^2)$, contradicting the hypothesis that $y_1$ is the least element of $D_{-1}(p, \mathfrak{r}_n-K_{n-1}- L_{n-1}r_ns_n^2)$.
This shows that $y_1\in X_n^+$.

By \Cref{l: yn}~\ref{i: yn rn-1} and \Cref{r. yln y cap}, the map $\mathfrak{h}_{n,z}$ preserves $\leq_n$ over $D_{-1}(p, \mathfrak{r}_n-K_{n-1})$. 
So, using the same argument, we get  $\mathfrak{h}_{n,z}(y_1)\in \xn^{+}$.

Now, given $y\in \Yn^+\cap D_{-1}(p, \mathfrak{r}_n-K_{n-1}- L_{n-1} r_ns_n^2)$ so that $y_1<_ny$, suppose that the result is true for all $y'\in \Yn^+\cap D_{-1}(p, \mathfrak{r}_n-K_{n-1}- L_{n-1} r_ns_n^2)$ with $y'<_n y$. 
By definition, we have  $y\notin \xn^+$ if and only if there is some $u\in \xn^+$ such that  $u<_n y$ and $d_{n-1}(u,p)\leq 2r_ns_n$. 
Using the same argument as before, we obtain that, necessarily, $u\in D_{-1}(p,\mathfrak{r}_n-K_{n-1}-L_{n-1} r_ns_n^2)$. 
By the induction hypothesis, we have  $\mathfrak{h}_{n,z}(u)\in \xn^+$. 
Then $y\notin \xn^+$ if and only if there is some $u\in D_{-1}(\mathfrak{r}_n-K_{n-1})$ with $\mathfrak{h}_{n,z}(u)\in \xn^+$ and $d_{n-1}(\mathfrak{h}_{n,z}(u),\mathfrak{h}_{n,z}(y))\leq 2r_ns_n$. 
But, by the induction hypothesis with \ref{i: xn partial}, we have $d_{n-1}(\mathfrak{h}_{n,z}(u),\mathfrak{h}_{n,z}(y))=d_{n-1}(u,y)\leq 2r_ns_n$.
So $y\in \xn^{+}$ if and only if $\mathfrak{h}_{n,z}(y )\in \xn^{+}$, as desired.
\end{proof}

\begin{prop}\label{p: h preserves xn+}
For all $(l,z)\in \Ryn$ and $y\in \Yn\cap D_{-1}(p, \mathfrak{r}_l-K_{n-1}- L_{n-1} r_ns_n^2)$, we have  $y\in \xn^{+}$ if and only if $\mathfrak{h}_{l,z}(y )\in \xn^{+}$.
\end{prop}

\begin{proof}
We proceed by induction on $l\geq n$.
The case $l=n$ is precisely the statement of \Cref{pl. h preserves xn+}.
Therefore take any $l>n$ and  suppose that the result is true for  $n\leq l'<l$.

By \Cref{l: h preserve ynpm}, it is enough to prove the statement for points $y\in Y_n^+$.
We proceed by induction on the elements of $\Yn^+\cap D_{-1}(p, \mathfrak{r}_l-K_{n-1}- L_{n-1}r_ns_n^2)$ using $\leq_n$.
Let $y_1$ be the least element of  $\Yn^+\cap D_{-1}(p, \mathfrak{r}_l-K_{n-1}- L_{n-1} r_ns_n^2)$.
We will prove that $y_1\notin \xn^{+}$ if and only if $ \mathfrak{h}_{l,z}(y_1)\notin \xn^{+}$, establishing the desired property for $y_1$.

The condition $y_1\notin \xn^{+}$ means that $y_1>_ny_0$ and there is some $u\in \xn^{+}$ such that $u<_n y_1$ and $d_{n-1}(y_1,u)\leq 2r_ns_n$. 
Since $s_n>2$ by~\eqref{defn s0} and~\eqref{defn sn}, it follows from \Cref{l: ballnrnsn} that $u\in D_{-1}(p, \mathfrak{r}_l-K_{n-1})$, and therefore $\hat{\mathfrak{c}}(y_1),\hat{\mathfrak{c}}(u)\leq l$.
We will consider several cases about $u$.

Suppose that  $\hat{\mathfrak{c}}_n(u)> \hat{\mathfrak{c}}_n(y_1)$. Then $y_1<_n u$ by \Cref{l: yn}~\ref{i: yn cn < cn}, contradicting the assumption that  $u<_n y_1$.

Suppose then that $\hat{\mathfrak{c}}(y_1)=\hat{\mathfrak{c}}(u)=l$.
Thus $\hat{\mathfrak{p}}(y_1)=\hat{\mathfrak{p}}(u)=p$.
\Cref{l: yn}~\ref{i: yn cn = cn} and the assumption that $u<_n y_1$ yield $d_{-1}(p,u)\leq d_{-1}(p,y_1) $.
Therefore $u\in \Yn^+\cap D_{-1}(p, \mathfrak{r}_l-K_{n-1}- L_{n-1}r_ns_n^2)$, contradicting the hypothesis that $y_1$ is the least element in $\Yn^+\cap D_{-1}(p, \mathfrak{r}_l-K_{n-1}- L_{n-1}r_ns_n^2)$.

Suppose finally that $\hat{\mathfrak{c}}(u)<l$.
Then $\mathfrak{h}_{\hat{\mathfrak{c}}(u),\hat{\mathfrak{p}}(u)}(u)\in X_n^+$ by the induction hypothesis with $l$.
But, by the induction hypothesis with \ref{i: xn partial}, we have $d_{n-1}(\mathfrak{h}_{l,z}(u),\mathfrak{h}_{l,z}(y_1))=d_{n-1}(u,y_1)\leq 2r_ns_n$. So  $\mathfrak{h}_{l,z}(y_1 )\notin \xn^{+}$.

Thus far, we have proved that $y_1\notin \xn^{+}$ implies $\mathfrak{h}_{l,z}(y_1 )\notin X_n^+$. 
The proof of the converse implication is similar

Now, given $y\in \Yn^+\cap D_{-1}(p, \mathfrak{r}_l-K_{n-1}- L_{n-1} r_ns_n^2)$ so that $y_1<_ny$, suppose that the result is true for all $y'\in \Yn^+\cap D_{-1}(p, \mathfrak{r}_l-K_{n-1}- L_{n-1} r_ns_n^2)$ with $y'<_n y$. 
By definition, $y\notin \xn^+$ if and only if there is some $u\in \xn^+$ such that  $u<_n y$ and $d_{n-1}(u,p)\leq 2r_ns_n$. 
Using the same argument as before, we obtain that, either $\hat{\mathfrak{c}}_n(u)<l$, or $u\in D_{-1}(p,\mathfrak{r}_l-K_{n-1}-L_{n-1} r_ns_n^2)$.
If  $\hat{\mathfrak{c}}_n(u)<l$, we get $\mathfrak{h}_{l,z}(y)\notin X_n^+$ arguing as before.
If $u\in D_{-1}(p,\mathfrak{r}_l-K_{n-1}-L_{n-1} r_ns_n^2)$, then $\mathfrak{h}_{l,z}(u)\in \xn^+$ by the induction hypothesis in $\Yn^+\cap D_{-1}(p, \mathfrak{r}_l-K_{n-1}- L_{n-1} r_ns_n^2)$.
Thus $y\notin \xn^+$ if and only if there is some $u\in D_{-1}(\mathfrak{r}_l-K_{n-1})$ with $\mathfrak{h}_{l,z}(u)\in \xn^+$ and $d_{n-1}(\mathfrak{h}_{l,z}(u),\mathfrak{h}_{l,z}(y))\leq 2r_ns_n$. 
But $d_{n-1}(\mathfrak{h}_{l,z}(u),\mathfrak{h}_{l,z}(y))=d_{n-1}(u,y)\leq 2r_ns_n$ by the induction hypothesis with~\ref{i: xn partial}.
So $y\in \xn^{+}$ if and only if $\mathfrak{h}_{l,z}(y )\in \xn^{+}$, as desired.
\end{proof}

Let
\begin{equation}\label{defn xn-}
\xn^{-}=\{\, y\in \Yn^{-}   \mid d_{n-1}( y, \xn^{+} )>r_n(2s_n+1) \, \}\;.
\end{equation}
Recall that $\xn=X_n^-\cupdot X_n^+$. 

\begin{lem}\label{l: p xn}
We have $p\in X_n$.
\end{lem}

\begin{proof}
Suppose first that  condition~\ref{i: rn bar} is satisfied in the definition of $r_n$, and consequently $r_n=\bar{r}_n$.
Then there is some  $x\in D_{n-1}(p,\hat{r}_n(2s_n+1))$ such that
\begin{equation}\label{inequality rn x}
(|D_{n-1}(x,\hat{r}_ns_n)|+6)^2\geq \eta_n(|D_{n-1}(x,\hat{r}_n)|)\;.
\end{equation} 
So $D_{n-1}(x,\hat{r}_ns_n)\subset D_{n-1}(p, \hat{r}_n(3s_n+1)) $, and therefore 
\begin{equation}\label{inequality rn x 2}
|D_{n-1}(p, r_n)|= |D_{n-1}(p, \hat{r}_n(3s_n+1))|\geq|D_{n-1}(x,\hat{r}_ns_n)|\;.
\end{equation}
Using~\eqref{defn barrn},~\eqref{barrn defn},~\eqref{etan oletan},~\eqref{inequality rn x} and~\eqref{inequality rn x 2}, we get 
\begin{align*}
\eta_n(|D_{n-1}(p,r_ns_n)|)
&\geq\eta_n(|D_{n-1}(x,\hat{r}_ns_n)|)
\geq\eta_n\big(\sqrt{\eta_n(|D_{n-1}(x,\hat{r}_n)|)}-6\big) \\
&\geq\overline{\eta}_n\big(\sqrt{\overline{\eta}_n(|D_{n-1}(x,\hat{r}_n)|)}-6\big) 
>\overline{\eta}_n\big(\sqrt{\overline{\eta}_n(\hat{r}_n))}-6\big) \\
&\geq \Big(4(\bm{\Delta}_{n-1}(\bar{\bm{r}}_{n-1})-1)^{\bar{r}_ns_n^2}+6\Big)^2\;.
\end{align*}
The assumption $r_n=\bar{r}_n$ implies $\bar{\bm{r}}_{n-1}=(r_0,\dots, r_{n-1})$ and $\bm{\Delta}_{n-1}(\bar{\bm{r}}_{n-1})=\Delta_{n-1}$ according to~\eqref{no bm multi}.
Hence, by \Cref{c: |D(x r)| le ...},
\begin{align*}
\eta_n(|D_{n-1}(p,r_ns_n)|)&\geq \left(4(\bm{\Delta}_{n-1}(\bar{\bm{r}}_{n-1})-1)^{\bar{r}_ns_n^2}+6\right)^2 \\
&= \left(4(\Delta_{n-1}-1)^{r_ns_n^2}+6\right)^2 \geq (|D_{n-1}(p,r_ns_n^2)|+6)^2\;,
\end{align*}
and therefore $p\in Y_n^+$.
Thus the statement follows in this case from \Cref{r. yn p} and the definition of $X_n^+$.

Suppose now that  condition~\ref{i: rn hat} holds. 
Then $p\in Y_n^-$ and $Y_n^+\cap D_{n-1}(p, r_n(2s_n+1))=\emptyset$, and the statement also  follows in this second case.
\end{proof}

By \eqref{defn bmkn},~\eqref{defn kn} and~\eqref{no bm multi}, we have
\begin{align}
\label{overlinekn actual}
\overline{K}_{n} &=K_{n-1}+ L_{n} (r_ns_n^2 + r_n(2s_n+1) )\;, \\
\label{kn actual}
K_n &= \overline{K}_{n} + L_n(s_{n+1}R_{n+1}^+ + \Gamma_n^+ +  2R_n^+)\;.
\end{align}

\begin{lem}\label{l: h preserves xn-}
For all $(l,z)\in \Ryn$ and $y\in \Yn\cap D_{-1}(p, \mathfrak{r}_l-\overline{K}_{n})$, we have  $y\in \xn^{-}$ if and only if $\mathfrak{h}_{l,z}(y )\in \xn^{-}$.
\end{lem}

\begin{proof}
Let $y\in \Yn\cap D_{-1}(p, \mathfrak{r}_l-\overline{K}_{n})$. 
Then, by \eqref{overlinekn actual},
\[
y\in \Yn\cap D_{-1}(p, \mathfrak{r}_l-K_{n-1}- L_{n-1} (r_ns_n^2 + r_n(2s_n+1) ))\;.
\]
By \Cref{l: h preserve ynpm}, we can assume $y,\mathfrak{h}_{l,z}(y )\in Y_n^-$.
Hence, by definition, $y\notin \xn^{-}$ if and only if there is some $x\in X_n^+$ with $d_{n-1}(y,x)\leq r_n(2s_n+1)$.
In this case, by the induction hypothesis with~\ref{i: xn dn dn-1}, we have $d_{-1}(x,y)\leq L_{n-1}r_n(2s_n+1)$.
Therefore, by the triangle inequality, $x\in D_{-1}(p, \mathfrak{r}_l-K_{n-1}- L_{n-1} r_ns_n^2)\subset D_{-1}(p, \mathfrak{r}_m-K_n)$.
Applying now \Cref{p: h preserves xn+}, we get $\mathfrak{h}_{l,z}(x)\in X_n^+$.
Also, by the induction hypothesis with~\ref{i: xn partial}, $\mathfrak{h}_{l,z}$ is an $s_nR_{n}^+$-short scale isometry on  $(X_{n-1}\cap D_{-1}(p, \mathfrak{r}_m-K_n),d_{n-1})$.
Therefore $\mathfrak{h}_{l,z}(x)\in X_n^+$ and $d_{n-1}(\mathfrak{h}_{l,z}(x),\mathfrak{h}_{l,z}(y))\leq r_n(2s_n+1)$, obtaining $\mathfrak{h}_{l,z}(y)\notin X_n^-$.

The proof of the converse implication is similar.
\end{proof}

After these preliminaries, let us show  that $X_n$ satisfies the statement of \Cref{p: xn}. Let us start with~\ref{i: xn subset}. 
By \Cref{l: p xn}, we have  $p\in X_n$ and $(n,x)\in \mathfrak{P}_{n-1}$ for each $x\in \mathfrak{X}_n$. 
\Cref{p: h preserves xn+} and \Cref{l: h preserves xn-} then imply   $x=\mathfrak{h}_{n,x}(p)\in X_n$ for all $x\in\mathfrak X_n$, obtaining $\mathfrak{X}_n\subset X_n$.
The inclusion $X_n\subset X_{n-1}$ follows from \Cref{l: yn}~\ref{i: yn maximal} and the fact that $X_n\subset Y_n$. This completes the proof of~\ref{i: xn subset}.

For all $(m,x)\in \mathfrak{P}_{n-1}$, the  map $\mathfrak{h}_{m,x}\colon (D_{-1}(p,\mathfrak{r}_m),p)\to (D_{-1}(x,\mathfrak{r}_m),x)$ is a pointed isometry by definition.
Therefore $\mathfrak{h}_{m,x}(D_{-1}(p,\mathfrak{r}_m-\overline{K}_n)) = D_{-1}(x,\mathfrak{r}_m-\overline{K}_n)$.
Then property~\ref{i: xn rn-1} follows from  \Cref{p: h preserves xn+} and \Cref{l: h preserves xn-}.

Let us prove~\ref{i: xn1}.
For $x\in X_n^+$, the result is an immediate consequence of the definition of $Y_n^+$ and the fact that $X_n^+  \subset Y_n^+$.
So assume $x\in X_n^-$.
By absurdity, suppose that
\[
(|D_{n-1}(x, r_ns_n)|+6)^2>\eta_n(|D_{n-1}(x,r_n)|)\;.
\]
Since $\eta_n$ is an increasing function, and using~\eqref{etan oletan},~\eqref{barrn defn},~\eqref{no bm multi} and \Cref{c: |D(x r)| le ...}, we get
\begin{align*}
\eta_n(|D_{n-1}(x, r_ns_n)|) &\geq \eta_n\big( \sqrt{\eta_n(|D_{n-1}(x,r_n )|)}-6\big) 
\geq \overline{\eta}_n\big( \sqrt{\overline{\eta}_n(|D_{n-1}(x,r_n )|)}-6\big)\\
&> \overline{\eta}_n\big( \sqrt{\overline{\eta}_n(r_n)}-6\big) 
> \Big(4(\bm{\Delta}_{n-1}(\bar{\bm{r}}_{n-1})-1)^{\bar{r}_ns_n^2}+6\Big)^2\\
&= \Big(4(\Delta_{n-1}-1)^{\bar{r}_ns_n^2}+6\Big)^2 \geq (|D_{n-1}(x,r_ns_n^2)|+6)^2\;.
\end{align*}
So  $x\notin Y_n^-$ by definition, contradicting the assumption that $x\in X_n^-$, which completes the proof of~\ref{i: xn1}.

Let us prove~\ref{i: xn separated net}. First, define
 \begin{align}
\Zmiyn &= \{\,z\in \xyn \mid  d_{n-1}(z,\xn^+)-2r_ns_n > d_{n-1}(z,\xn^-) -r_n \,\}\;, \label{def zn-}\\ 
Z^+_{n-1} &= \{\,z\in \xyn \mid  d_{n-1}(z,\xn^+)-2r_ns_n \leq d_{n-1}(z,\xn^-) -r_n \,\}\;. \label{def zn+} 
\end{align}
Thus $X_{n-1}=\Zmiyn \cupdot Z^+_{n-1}$.
On the other hand, using~\eqref{defn functions 0},~\eqref{defn functions n},~\eqref{constants -1} and~\eqref{no bm multi}, we get
\[
R_n^-=4r_n-1\;, \quad R_n^+ = r_n(2s_n+3)\;.
\]

\begin{lem}\label{l: Xn1 net in Cn1}
$\xn^+$ is $(2r_ns_n+1)$-separated and $R_n^+$-relatively dense in $(Z^+_{n-1},d_{n-1})$.
\end{lem}

\begin{proof}
By \Cref{r. xn+ sep}, we only need to show that $\xn^+$ is $R_n^+$-relatively dense in $(Z^+_{n-1},d_{n-1})$.
Take an arbitrary point $z\in Z_{n-1}^+$. 
Since $\Yn$ is $2r_n$-relatively dense in $(X_{n-1},d_{n-1})$ by \Cref{l: yn}~\ref{i: yn maximal}, there is some  $y\in Y_n$ with $d_{n-1}(x,z)\leq 2r_n$. 

If $y\in  Y_n^+$, then, by \Cref{r. xn+ sep}, there is some $x\in \xn^+$ with $d_{n-1}(y,x)\leq 2r_ns_n$. 
Using the triangle inequality, we get 
\[
d(z,x)\leq d(z,y)+d(y,x)\leq 2r_n +2r_ns_n<r_n(2s_n+3)=R_n^+\;.
\]

If $y\in \xn^-$, we have $d_{n-1}(z, \xn^-)\leq 2r_n$. 
Then \eqref{def zn+} implies $d_{n-1}(z,\xn^+)-2r_ns_n\leq r_n$, obtaining $d_{n-1}(z,\xn^+)\leq r_n(2s_n+1)<R_n^+$. 

Finally, suppose that $y\in \Yn^-\setminus \xn^-$. 
By \eqref{defn xn-}, there is some  $x\in \xn^+$ with $d_{n-1}(x,y)\leq r_n(2s_n+1)$, and the lemma follows applying the triangle inequality:
\[
d(z,x)\leq d(z,y)+d(y,x)\leq 2r_n +r_n(2s_n+1)=r_n(2s_n+3)=R_n^+\;.\qedhere
\]
\end{proof}

\begin{lem}\label{l: Xn0 net in Cn0}
$\xn^-$ is $(2r_ns_n+1)$-separated and $R_n^-$-relatively dense in $(\Zmiyn,d_{n-1})$.
\end{lem}

\begin{proof}
Let $z\in \Zmiyn$. 
Like in \Cref{l: Xn1 net in Cn1}, there is some $y\in \Yn$ with $d_{n-1}(z,y)\leq 2r_n$. 

In the case where $y\in \xn^-$, the lemma is trivial. 

If $y\in \xn^+$, then  $d_{n-1}(z,X_n^+)\leq 2r_n$, yielding  $d_{n-1}(z,\xn^+)-2r_ns_n\leq 2r_n(1-s_n)$. 
Using~\eqref{def zn-}, we get $d_{n-1}(y,\xn^-)-r_n < 2r_n(1-s_n)$, and therefore $d_{n-1}(y,\xn^-)<2r_n(2-s_n)$. 
However, by~\eqref{defn s0} and~\eqref{defn sn}, we have $s_n>2$, reaching a contradiction.
Therefore $y\notin \xn^+$.

Now suppose $y\in  \Yn^+\setminus \xn^+$. 
By \Cref{r. xn+ sep}, there is some $x\in \xn^+$ with $d_{n-1}(x,y)\leq 2r_ns_n$, and we get $d_{n-1}(z,x)\leq 2r_n(s_n+1)$ using the triangle inequality. 
Then \eqref{def zn-} yields 
\begin{align*}
d_{n-1}(z,\xn^-)&<d_{n-1}(z,\xn^+)- 2r_ns_n+r_n\leq d_{n-1}(z,x)-2r_ns_n+r_n\\
&\leq 2r_n(s_n+1)-2r_ns_n+r_n= 3r_n \leq R_n^-\;.
\end{align*}

Finally, suppose $y\in Y_n^- \setminus X_n^-$. 
By \eqref{defn xn-}, there is some  $x\in \xn^+$ with $d_{n-1}(x,y)\leq r_n(2s_n+1)$, obtaining $d_{n-1}(z,X_n^+)\leq r_n(2s_n+3)$ by the triangle inequality. 
Therefore $d_{n-1}(z,X_n^+)-2r_ns_n  \leq 3r_n$, obtaining $d_{n-1}(z,X_n^-)<4r_n$ by~\eqref{def zn-}; i.e., $d_{n-1}(z,X_n^-)\le4r_n-1=R_n^-$.
\end{proof}

To finish the proof of \Cref{p: xn}~\ref{i: xn separated net}, it only remains to show that $d_{n-1}(\xn^-,\xn^+)\geq 2r_ns_n+1$, which follows from~\eqref{defn xn-}.

To prove the next items of \Cref{p: xn}, we need some more preliminary results.

\begin{lem}\label{l: mathcalcn local}
For all $z\in \xyn$, we have  $z\in Z^+_{n-1}$ if and only if 
\begin{equation}\label{zpln xn}
d_{n-1}(z,\xn^+\cap D_{n-1}(z,R_n^+))-2r_ns_n\leq d_{n-1}(z,\xn^-\cap D_{n-1}(z,R_n^+))-r_n\;.
\end{equation}
\end{lem}

\begin{proof}
Suppose first that $z\in Z^+_{n-1}$. 
\Cref{l: Xn1 net in Cn1} implies $\xn^+\cap D_{n-1}(z,R_n^+)\neq \emptyset$, and therefore  
\[
d_{n-1}(z,\xn^+\cap D_{n-1}(z,R_n^+))=d_{n-1}(z,\xn^+)\;.
\] 
Then \eqref{def zn-} implies \eqref{zpln xn}. 

Suppose now that \eqref{zpln xn} holds for some $z\in X_{n-1}$. 
Property~\ref{i: xn separated net} implies that at least one of the inequalities $d_{n-1}(z,\xn^-)\leq R_n^+$ or $d_{n-1}(z,\xn^+)\leq R_n^+$ is satisfied. So at least the left-hand side of \eqref{zpln xn} is finite. 
Therefore~\eqref{zpln xn} yields~\eqref{def zn-}.
\end{proof}

\begin{cor}\label{c. h preserves z}
For all $u\in \xyn\cap D_{-1}(p, \rrl - \overline{K}_{n}-L_{n-1}R_n^+)$ and $(l,z)\in \mathfrak{P}_{n-1}$, we have $u\in \Zpmyn$ if and only if $\mathfrak{h}_{l,z} (u)\in \Zpmyn$.
\end{cor}

\begin{proof}
Let $u\in \xyn\cap D_{-1}(p, \rrl - \overline{K}_{n}-L_{n-1}R_n^+)$ and $(l,z)\in \mathfrak{P}_{n-1}$. 
Since $X_{n-1}= Z^-_{n-1} \cupdot Z^+_{n-1}$, it is enough to prove that $u\in Z^+_{n-1}$ if and only if $\mathfrak{h}_{l,z} (u)\in Z^+_{n-1}$.

The induction hypothesis with~\ref{i: xn dn dn-1} and the triangle inequality yield $D_{n-1}(u,R_n^+)\subset D_{-1}(p, \rrl - \overline{K}_{n})\subset \dom\mathfrak{h}_{l,z}$.
\Cref{p: h preserves xn+}, \Cref{l: h preserves xn-} and the induction hypothesis with~\ref{i: xn partial} imply that the restriction of $\mathfrak{h}_{l,z}$ to $D_{-1}(p, \rrl - \overline{K}_{n})$ preserves $X_n^\pm$ and is an $R_n^+$-short scale  isometry with respect to $d_{n-1}$. Then the result follows from \Cref{l: mathcalcn local}.
\end{proof}

\begin{rem}
Note that \eqref{overlinekn actual} yields $K_{n}\geq \overline{K}_{n}+L_{n-1}R_n^+$.
Then $\mathfrak{r}_l-\overline{K}_{n}-L_{n-1}R_n^+>0$ in \Cref{c. h preserves z} by~\eqref{rn}.
\end{rem}

Recall the definition of $r_n^\pm$ given in~\eqref{defn rnpm}.
\begin{lem}\label{l: xni ball cni}
 If $x\in \xn^\pm$, then $D_{n-1}(x,r_n^\pm)\subset Z^\pm_{n-1}$.
\end{lem}

\begin{proof}
For $x\in X_n^-$, suppose on the contrary that there is some $z\in D_{n-1}(x,r_n)$ such that 
\[
d_{n-1}(z,\xn^+)-2r_ns_n\leq d_{n-1}(z,\xn^-)-r_n\;.
\]
In particular,  $d_{n-1}(z,\xn^+)\leq 2r_ns_n$ because $d_{n-1}(z,\xn^-)\leq d_{n-1}(z,x)\leq r_n$. 
By the triangle inequality, it follows that
\[
d_{n-1}(x,\xn^+)\le d_{n-1}(x,z)+d_{n-1}(z,\xn^+)\le r_n+2r_ns_n=r_n(2s_n+1)\;,
\] 
contradicting the definition of $\xn^-$ in~\eqref{defn xn-}.

The proof when $x\in X_n^+$ is similar.
\end{proof}

For every $x\in\xn^\pm$, let
\begin{equation}\label{defn overlinecnn-1x}
\olcnyn (x)= \{\,z\in Z_{n-1}^\pm \mid d_{n-1}(z,x)= d_{n-1}(z,\xn^\pm) \,\}\;.
\end{equation}

\begin{rem}\label{r. overlinecnx partition}
Observe that the sets $\olcnyn(x)$, for $x\in \xn$, cover $\xyn$.
\end{rem}

\begin{lem}\label{l: overlinecnx rnpm}
For $x\in \xn^\pm$, we have $\olcnyn(x)\subset D_{n-1}(x, R_n^\pm) $.
\end{lem}
\begin{proof}
This is a direct consequence of \Cref{l: Xn1 net in Cn1,l: Xn0 net in Cn0}.
\end{proof}

Define a graph structure $E_n$ on $\xn$ by declaring that $x,y\in \xn$ are joined by an edge if 
\begin{equation}\label{xeny}
d_{n-1}(\olcnyn (x),\olcnyn (y) )\leq 1\;.
\end{equation}

To prove~\ref{i: xn connected}, let $x,y\in \xn$.
By the induction hypothesis with~\ref{i: xn connected}, $\xyn$ is connected, and, by construction, $X_{n}\subset \xyn$.
So there is some path in $(\xyn,E_{n-1})$ of the form $(u_0=x, u_1, \ldots, u_a= y)$. 
By \Cref{r. overlinecnx partition}, for each $i=0,\dots,a$, there is some $z_i\in \xn$ such that $u_i\in \olcnyn(z_i)$, $z_0 = x$ and $z_a=y$. 
Clearly, $d_{n-1}(\olcnyn(z_{i-1}), \olcnyn(z_i))\leq 1$ for $i=1,\dots,a$. Thus  $(z_0,\ldots, z_a)$ is a path in $\xn$ connecting $x$ to $y$. 

Let us prove~\ref{i: xn dn dn-1}. 
For $x,y\in \xn$ with $d_n(x,y)=a$, there is a finite sequence  $(x_0=x,x_1,\ldots, x_a=y)$ in $X_n$ such that $d_{n}(\olcnyn(x_{i-1}),\olcnyn(x_{i}))\leq 1$  for $i=1,\ldots, a$. 
By \Cref{l: overlinecnx rnpm},~\eqref{defn functions 0} and~\eqref{no bm multi}, we have  $d_{n-1}(x_{i-1},x_{i})\leq 2R_n^+ + 1=l_n$. 
Then~\ref{i: xn dn dn-1} follows from the triangle inequality, using~\eqref{defn functions n},~\eqref{constants -1} and~\eqref{no bm multi}.

Let us prove~\ref{i: xn deltan deltan-1}. 
For $x,y\in \xn$, if $xE_ny$, then $d_{n-1}(x,y)\leq 2R_n^+ +1$ by~\eqref{xeny} and \Cref{l: overlinecnx rnpm}.
So 
\[
|S_n(x,1)|\leq |D_{n-1}(x,2R_n^++1 )| \leq 4(\deg X_{n-1}-1)^{2R_n^+}
\]
 by \Cref{c: |D(x r)| le ...}.
Then the bound $\deg X_n\leq \Delta_n$ follows by induction with~\ref{i: xn deltan deltan-1}, using~\eqref{defn functions 0},~\eqref{defn functions n} and~\eqref{no bm multi}.

Let us prove~\ref{i: xn partial}. 
Let $(m,z)\in \mathfrak{P}_{n-1}$ and $x\in \xn\cap D_{-1}(p, \mathfrak{r}_m-\overline{K}_{n}- 2L_{n} R_n^+ )$.
Then 
\begin{equation}\label{overlinecnn-1 hmz}
\overline{C}_{n,n-1}(x)\subset D_{-1}(p, \mathfrak{r}_m-\overline{K}_{n}- L_{n} R_n^+ )\subset \dom\mathfrak{h}_{m,z}
\end{equation}
by \Cref{l: overlinecnx rnpm}, \Cref{p: xxn}~\ref{i: xxn xxm}, and the induction hypothesis with~\ref{i: xn dn dn-1} and~\ref{i: xn partial}. Recall that $\mathfrak{P}_{n-1}\subset \mathfrak{P}_{n-2}$ by~\eqref{rrn} and~\eqref{defn rr-1}.
Furthermore, from the induction hypothesis with~\ref{i: xn partial}, \Cref{p: xxn}~\ref{i: xxn xxm}, \Cref{c. h preserves z},~\eqref{defn overlinecnn-1x} and~\eqref{overlinecnn-1 hmz}, it follows that 
\begin{equation}\label{h preserves overlinec}
\mathfrak{h}_{m,z}\left(\overline{C}_{n,n-1}(x)\right)= \overline{C}_{n,n-1}\left(\mathfrak{h}_{m,z}(x)\right)\;.
\end{equation}
So, for $x,y\in \xn\cap D_{-1}(p, \mathfrak{r}_m-\overline{K}_{n}- 2L_{n} R_n^+ )$, \eqref{xeny} holds  if and only if 
\[
d_{n-1}(\overline{C}_{n,n-1}(\mathfrak{h}_{m,z}(x)),\overline{C}_{n,n-1}(\mathfrak{h}_{m,z}(y)))\leq 1\;.
\]
Therefore $xE_ny$ if and only if $\mathfrak{h}_{m,z}(x)E_n\mathfrak{h}_{m,z}(y)$. 
Then~\ref{i: xn partial} is a consequence of \Cref{c: graph partial},~\eqref{kn actual} and the induction hypothesis with~\ref{i: xn dn dn-1}.
This completes the proof of \Cref{p: xn}.

\section{Clusters}

In order to define the colorings satisfying the conditions of \Cref{t: finitary}, we will divide the sets $X_{n-1}$ into ``clusters'', denoted by $C_{n,n-1}(x)$ and indexed by $x\in X_n$.
These will be used in \Cref{s. colorings} to construct the suitable colorings ``locally" on this family of sets.

In \Cref{s. xn} we have defined well-ordered sets $(Y_n,\leq_{n})$ for $n\in \N$, whose restrictions to  the subsets $X_n$ determine a family of well-orders also denoted by $\leq_n$. 
For $n\in \N$, let $\pi_{n-1}^\pm \colon \Zpmyn \to \xn^\pm$ be defined by
\begin{equation} \label{defn pn-1}
\pi_{n-1}^\pm (u)= \inf \{\,   x\in \xn^\pm \, | \, d_{n-1}(u,x)= d_{n-1}(u,\xn^\pm)   \, \}\;, 
\end{equation}
with respect to $\leq_n$.
Denote by $\pi_{n-1}$ the union of $\pi_{n-1}^-$ and $\pi_{n-1}^+$, which is defined on $Z_{n-1}^+\cupdot Z_{n-1}^-=X_{n-1}$.
For each $n\in \N$ and $x\in \xn^\pm$, let $\cnyn(x)=(\pi_n^\pm)^{-1}(x)$.
These sets form a partition of $\xyn$ and satisfy
\begin{equation}\label{cnyn overlinecn}
\cnyn(x)=\overline{C}_{n,n-1}(x)\setminus\bigcup_{x'\in\xn^\pm,\ x'<_n x} \overline{C}_{n,n-1}(x')
\end{equation}
for $x\in\xn^\pm$ by~\eqref{defn overlinecnn-1x} and~\eqref{defn pn-1}.
For $-1\leq m< n-1 $, we continue defining sets $\cnm(x)$ and $\overline{C}_{n,m}(x)$ by reverse induction on $m$, taking 
\[ 
\cnm(x) = \bigcup_{u \in \cnsm(x)} \csmm(u)\;, \quad 
\overline{C}_{n,m}(x)=\bigcup_{u \in \overline{C}_{n,m+1}(x)} \overline{C}_{m+1,m}(u)\;.
\] 
It is straightforward to check that, for  $-1\leq l_1<l_2< l_3 \leq n$,
\begin{equation}\label{cluster intermediate}
C_{l_3,l_1}(x) = \bigcup_{u \in C_{l_3,l_2}(x)} C_{l_2,l_1}(u)\;, \quad 
\overline C_{l_3,l_1}(x) = \bigcup_{u \in \overline C_{l_3,l_2}(x)} \overline C_{l_2,l_1}(u)\;.
\end{equation}

By~\eqref{defn functions n} and~\eqref{no bm multi}, we have
\begin{equation}\label{expression gamma}
\Gamma^\pm_0=R_0^\pm \;, \quad 
\Gamma^\pm_n= R_n^\pm L_{n-1}+\Gamma_{n-1}^\pm\;.
\end{equation}

\begin{lem}\label{l: cn-1 gamma} 
$\cnone(x)\subset \overline{C}_{n,-1}(x)\subset D_{-1}(x,\Gamma_n^\pm)$.
\end{lem}

\begin{proof}
We proceed by induction on $n$. 
For $n=0$ and $x\in X_0^\pm$, we have $C_{0,-1}(x)\subset D_{-1}(x,R_0^\pm)$ by \Cref{l: overlinecnx rnpm} and~\eqref{cnyn overlinecn}.
Now take any $n>0$ and suppose that $C_{m,-1}(y)\subset \overline C_{m,-1}(y)\subset D_{-1}(y,\Gamma_m^\pm)$ for $0\leq m <n $ and $y\in X^\pm_m$.
By~\eqref{cluster intermediate}, 
\[ 
C_{n,-1}(x) = \bigcup_{u \in C_{n,n-1}(x)} C_{n-1,-1}(u)\;, \quad 
\overline C_{n,-1}(x) = \bigcup_{u \in \overline C_{n,n-1}(x)} \overline C_{n-1,-1}(u) \;.
\]
We get $d_{n-1}(x,u)\leq R_n^+$ for all $u\in \overline C_{n,n-1}(x)$ by \Cref{l: overlinecnx rnpm} and~\eqref{cnyn overlinecn}.
So $d_{-1}(x,u)\leq L_{n-1}R_n^+$ by \Cref{p: xn}~\ref{i: xn deltan deltan-1}.
Then the result follows easily from the induction hypothesis using the triangle inequality.
\end{proof}

\begin{lem}\label{l: cnn-1 contained}
We have $D_{n-1}(x,r_n^\pm)\subset C_{n,n-1}(x)$ for every $n\in \N$ and $x\in X_n^\pm$.
\end{lem}
\begin{proof}
We  have $u\in Z_{n-1}^\pm$ for $u\in D_{n-1}(x,r_n^\pm)$ by \Cref{l: xni ball cni}, and $d_{n-1}(u,X_n)\leq r_n^\pm$ by definition.
Then the result follows from~\eqref{defn pn-1} and the fact that $X_n^\pm$ is $(2r_n^++1)$-separated by \Cref{p: xn}~\ref{i: xn separated net}.
\end{proof}

The following result follows from \Cref{l: cnn-1 contained} by induction.

\begin{cor}\label{c. cn-1 contained}
We have $D_{-1}(x,\sum \nolimits_{i=0}^n r_i)\subset C_{n,n-1}(x)$ for every $n\in \N$ and $x\in X_n$.
\end{cor}

The following lemma states that every $C_{n,n-1}(x)$ is a star-shaped subset of $(X_{n-1},E_{n-1})$ with center $x$.
\begin{lem}\label{l: cnn-1 connected}
For $x\in X_n^\pm$ and $u\in C_{n,n-1}(x)$, any geodesic segment in $(X_{n-1},E_{n-1})$ of the form $\tau = (x=\tau_0,\dots,\tau_l=u)$ is a path in $C_{n,n-1}(x)$.
\end{lem}
\begin{proof}
We prove that $\tau_k\in C_{n,n-1}(x)$ by reverse induction on $k=0,\dots, l$.
We have  $\tau_l=u\in C_{n,n-1}(x)$ by hypothesis.
Now suppose that $\tau_{k+1}\in C_{n,n-1}(x)$ for some $k=0,\dots,l-1$.
Assume by absurdity that $\tau_k\notin C_{n,n-1}(x)$.
Since $\tau$ is a geodesic segment, 
\begin{align*}
d_{n-1}(\tau_k,\xn^\pm)&\le d_{n-1}(\tau_k,x)=d_{n-1}(\tau_{k+1},x)-1\\
&= d_{n-1}(\tau_{k+1},\xn^\pm)-1\le d_{n-1}(\tau_k,\xn^\pm)\;,
\end{align*}
and therefore $\tau_k\in \overline{C}_{n,n-1}(x)$. 
So, according to~\eqref{defn pn-1}, there must be some $y\in X_n^\pm$ such that $d_{n-1}(\tau_k,y)=d_{n-1}(\tau_k,x)=k$ and $y<_nx$.
But then $d_{n-1}(\tau_{k+1},y)\leq k+1 = d_{n-1}(\tau_{k+1},x)$, yielding $\tau_{k+1}\notin C_{n,n-1}(x)$ by~\eqref{defn pn-1}, a contradiction.
\end{proof}

\begin{lem}\label{l: h preserves Cn}
Let $x\in X_n\cap D_{-1}(p,\mathfrak{r}_m-K_{n-1}- 2L_{n-1}R_n^+)$ and $(m,z)\in \mathfrak{P}_{n-1}$. 
Then $C_{n,n-1}(x)\subset \dom\mathfrak{h}_{m,z}$ and $\mathfrak{h}_{m,z}(\cnyn(x) )= \cnyn(\mathfrak{h}_{m,z}(x))$.
\end{lem}
\begin{proof}
It is an immediate consequence of~\eqref{overlinecnn-1 hmz},~\eqref{h preserves overlinec},~\eqref{cnyn overlinecn} and \Cref{l: yn}~\ref{i: yn rn-1}.
\end{proof}

\section{Colorings}\label{s. colorings}

In this section we will define families of colorings ``locally"  on the clusters. For this we will need several intermediate objects, as well as suitable notions of equivalences. Again, making all these constructions compatible with the sets $\mathfrak{X}_n$ and maps $\mathfrak{h}^m_{n,x}$ makes it more convoluted, so all references to these objects can be omitted in the first reading.

\subsection{Colorings $\chi_n$}\label{s. chin}

Given $a\in \N$, let $[a]=\{0,\dots,a-1\}$. For $n\in \N$ and $x\in \xn^\pm$, let 
\begin{equation}\label{def inx}
H_{n,x} =  \left[\eta_n\left( \left|D_{n-1}\left(x,r_n^\pm \right)\right|\right)\right]\;, \quad
I_{n,x} =  \left[5+ \left|D_{n-1}\left(x,r_n^\pm s_n\right)\right|\right]\;.
\end{equation}
The standard ordering of $\N$ and the calligraphic ordering of $I_{n,x}^2$ can be used to realize $I_{n,x}^2$ as an initial segment of $\N$. 
Since $| I_{n,x}|^2\leq |H_{n,x}|$ by \Cref{p: xn}~\ref{i: xn1},  the sets $I_{n,x}$ and $I_{n,x}^2$ become initial segments of $H_{n,x}$.
For $n\in \N$, let
\begin{equation}\label{defn mathcalin}
\mathcal{H}_{n} = \bigcup_{x\in \xn} H_{n,x}\;, \quad \mathcal{I}_{n}= \bigcup_{x\in \xn} I_{n,x}\;.
\end{equation}
From now on, when referring to a coloring $\phi\colon \xn\to \mathcal{H}_{n}$ (respectively, $\phi\colon \xn\to \mathcal{I}_{n}$), we assume $\phi(x)\in H_{n,x}$ (respectively, $\phi(x)\in I_{n,x}$) for all $x\in \xn$.

\begin{prop}\label{p: chin}
For every $n\in \N$, there is a coloring $\chi_n \colon \xn \to \mathcal{I}_{n}$ satisfying the following conditions:
\begin{enumerate}[(i)]
\item \label{i: chin zero} We have $\chi_n(x)=0$ if and only if $x\in \xxn$.

\item \label{i: chin dn-1} For all $x,y\in \xn^\pm$ with $d_{n-1}(x,y)\leq r_n^\pm s_n $, we have $\chi_n(x)< \chi_n(y)$ if and only if $x<_n y$.
In particular, if $0<d_{n-1}(x,y)\leq r_n^\pm s_n $, then $\chi_n(x)\neq \chi_n(y)$.

\item \label{i: chin h} For every $(m,z)\in \mathfrak{P}_{n-1}$, the map $\mathfrak{h}_{m,z}\colon (D_n(p,\Gamma_m^+),\chi_n)\to (D_n(z,\Gamma_m^+),\chi_n)$ is color-preserving.
\end{enumerate}
\end{prop}

\begin{proof}
First, set $\chi_n(x)=0$ for all $x\in \xxn$.
Then we define $\chi_n(x)$ for $x\in X_n^\pm\setminus \xxn$ by induction using $\leq_n$.
Let $A_x^\pm=\{\ y\in X_n^\pm\  |\ y<_n x \ \}$, and let 
\begin{equation}\label{defn chin} 
\chi_n(x)=\min\big( I_{n,x} \setminus \big( \{0\} \cup\chi_n\big(A_x^\pm\cap D_{n-1}(x,r_n^\pm s_n)\big) \big) \big)\;.
\end{equation}
Note that this is well defined since 
\[
|A_x^\pm\cap D_{n-1}(x,r_n^\pm s_n)|\leq |D_{n-1}(x,r_n^\pm s_n)|-1\leq |I_{n,x}|-1\;.
\]
With this definition, it is obvious that $\chi_n$ satisfies~\ref{i: chin zero} and~\ref{i: chin dn-1}.

To prove~\ref{i: chin h}, we show by induction on $(X_n\setminus \xxn,\leq_n)$ that, if $x\subset D_n(z,\Gamma_m^+)$ for $(m,z)\in \mathfrak{P}_{n-1}$, then $\chi_n(x)= \chi_n(\mathfrak{h}_{m,z}^{-1}(x))$.
By \Cref{r. yn p},  the set $X_n\cap D_{-1}(p,\rrm - K_{n-1})$ is an initial segment of $(X_n,\leq_n)$.
For $x\in X_n\cap D_{-1}(p,\rrm - K_{n-1})$, the result is trivial since $\mathfrak h_{m,p}$ is the identity.
Suppose  $x\in X_n\cap D_{n}(z,\Gamma_m^+)$ for some $(m,z)\in \mathfrak{P}_{n-1}$ with $z\neq p$.
By~\eqref{rn} and~\eqref{no bm multi}, we have $D_{n-1}(x,r_n^\pm s_n)\subset D_{-1}(z,\rrm - K_{n-1})$.
Thus
\begin{equation}\label{defn chin preserve}
\mathfrak{h}_{m,z}\colon (D_{n-1}(\mathfrak{h}_{m,z}^{-1}(x),r_n^\pm s_n),\leq_n) \to (D_{n-1}(x,r_n^\pm s_n),\leq_n)
\end{equation}
is order-preserving and an $r_n^\pm s_n$-short scale isometry with respect to $d_{n-1}$ by \Cref{p: xn}~\ref{i: xn partial} and \Cref{l: yn}~\ref{i: yn rn-1}.
Therefore
\[
A_x^\pm \cap D_{n-1}(x,r_n^\pm s_n)=\mathfrak{h}_{m,z}(A^\pm_{\mathfrak{h}_{m,z}^{-1}(x)}\cap 
D_{n-1}(\mathfrak{h}_{m,z}^{-1}(x),r_n^\pm s_n) )\;.
\]
Then, by the induction hypothesis, we have
\[
\chi_n(A_x^\pm \cap D_{n-1}(x,r_n^\pm s_n))=\chi_n(A^\pm_{\mathfrak{h}_{m,z}^{-1}(x)}\cap 
D_{n-1}(\mathfrak{h}_{m,z}^{-1}(x),r_n^\pm s_n) ) \;.
\]
Moreover $I_{n,x}= I_{n,\mathfrak{h}_{m,y}^{-1}(x)}$ because~\eqref{defn chin preserve} is order-preserving and an $r_n^\pm s_n$-short scale isometry with respect to $d_{n-1}$.
Then the result follows from~\eqref{defn chin}.
\end{proof}

\subsection{Equivalences}

We define the notion of $n$-equivalence between points $x,y\in X_n$ by induction on $n\in \N$.
In addition, an explicit family of $n$-equivalences will be constructed, together with an induced equivalence relation.

Consider the restriction of the graph structure $E_{n-1}$ to $C_{n,n-1}(x)$ for every $n\in \N$ and $x\in X_n$.
\begin{defn}\label{d. equivalence zero}
For $x,y\in X_0$, a $0$-\emph{equivalence} from $x$ to $y$, denoted by $f\colon x \to y$, is a pointed graph isomorphism $f\colon (\overline C_{0,-1}(x),x) \to (\overline C_{0,-1}(y),y)$ such that $f( C_{0,-1}(x))= C_{0,-1}(f(x))$.
\end{defn}

Let $\sim_0^\pm$ be the equivalence relation on $X_0^\pm$  defined by declaring $x\sim_0^\pm y$ if there is some $0$-equivalence $x\to y$. 
Let $\Phi_0$ be the map defined on $X_0=X_0^+\cupdot X_0^-$ that sends every $x\in X_0^{\pm}$ to its $\sim_0^\pm$-equivalence class. 
The range of each of these maps  is obviously finite.

\begin{lem}\label{l: x0phi}
For  $n\in \N$, there are disjoint subsets $X_0^{-,\Phi}, X_0^{+,\Phi}\subset X_0$ satisfying the following properties:
\begin{enumerate}[(a)]
\item \label{i: x0phi maximal} The sets $X_0^{\pm,\Phi}$ are maximal among the subsets of $X_0^\pm$ where $\Phi_0$ is injective.
\item \label{i: x0phi dn} For $u\in X_0^{\pm,\Phi}$ and $v\in X_0^\pm$, if $\Phi_0(u)=\Phi_0(v)$, then $d_0(u,p)\leq d_0(v,p)$. 
\end{enumerate} 
\end{lem}

\begin{proof}
In every $\sim_0^\pm$-equivalence class,  take a representative that minimizes the $d_0$-distance to $p$.
\end{proof}

By \Cref{l: x0phi}, for every  $x\in X_0^{\pm}$, there is a unique element $u\in X_0^{\pm,\Phi}$ satisfying $\Phi_0(x)= \Phi_0(u)$. Let $\rep_0^\pm \colon X_0^\pm \to X_0^{\pm,\Phi}$ be the maps determined by this correspondence, and let $\rep_0\colon X_0\to X_0^\Phi:=X_0^{+,\Phi}\cupdot X_0^{-,\Phi}$ be their union.

\begin{lem}\label{l: rmn hmy zero}
For  all $(m,y)\in \mathfrak{P}_{-1}$ and $x\in X_0^\pm\cap D_{0}(p,\Gamma^+_0)$,  the following properties hold:
\begin{enumerate}[(a)]

\item \label{i: rmn hmy zero ball} $\overline C_{0,-1}(x)\subset \dom\mathfrak{h}_{m,y}$.

\item\label{i: rmn hmy zero iso} The map $\mathfrak{h}_{m,y}$ restricts to a $0$-equivalence $x\to\mathfrak{h}_{m,y}(x)$; in particular, $x\sim_0 \mathfrak{h}_{m,y}(x)$ and $p\sim_0 y$. 

\end{enumerate} 
\end{lem}

\begin{proof}
By \Cref{l: cn-1 gamma} and the triangle inequality, 
 \begin{equation}\label{r0n h0y}
 \overline C_{0,-1}(x)\subset D_{-1}(x,\Gamma^+_0)\subset D_{-1}\left(p,L_{0}\Gamma_m^+ + \Gamma^+_0 \right)\;.
 \end{equation}
By~\eqref{rn} and~\eqref{no bm multi},
\[
\mathfrak{r}_m > 4L_m\Gamma_m^+ + K_m\;.
\]
The assumption $(m,y)\in \mathfrak{P}_{-1}$ implies $m\geq 0$ according to~\eqref{defn rr-1}.
So  $L_m \geq L_0 \geq L_{-1}=1$ by~\eqref{defn functions n} and~\eqref{no bm multi}, $K_m\geq K_0>K_{-1}=0$ by~\eqref{defn bmkn},~\eqref{defn kn} and~\eqref{no bm multi}, and  $\Gamma_m^+\geq R_0^+$ by~\eqref{gamma r}.
Therefore 
\[
\mathfrak{r}_m - K_{-1}- 2L_{-1}R_0^+ > 4L_m\Gamma_m^+ + K_m - 2R_0^+ > L_{0}\Gamma_m^+ + R_0^+ \;.
\]
Then~\eqref{r0n h0y} yields
\begin{equation}\label{inclusion rm0 hmy}
\overline C_{0,-1}(x)\subset D_{-1}\left(p,\mathfrak{r}_m-K_{-1}- 2L_{-1}R_0^+\right)\;,
\end{equation}
completing the proof of~\ref{i: rmn hmy ball} because $\dom\mathfrak{h}_{m,y}=D_{-1}(p,\mathfrak{r}_m)$.

Property~\ref{i: rmn hmy iso} follows from~\eqref{h preserves overlinec} and \Cref{p: xn}~\ref{i: xn partial}. 
\end{proof}

\begin{prop}\label{p: hnx zero}
For  $x\in X_0^\pm$, there is a $0$-equivalence $h_{0,x}\colon\rep_0(x)\to x$ 
satisfying the following properties:
\begin{enumerate}[(i)]
\item \label{i: hnx zero phi} If $x\in X_0^{\pm,\Phi}$, then $h_{0,x}$ is the identity on  $\overline C_{0,-1}(x)$.

\item \label{i: hnx zero hh} For $(m,y)\in \mathfrak{P}_{-1}$ and $x\in X_0\cap D_0(y, \Gamma_0^+)$, we have $h_{0,x}=\mathfrak{h}_{m,y}h_{0, \mathfrak{h}_{m,y}^{-1}(x)}$.

\item \label{i: hnx zero xxn} If $x\in \mathfrak{X}_0$, then $h_{0,x}= \mathfrak{h}_{0,x}|_{ \overline C_{0,-1}(x)}$.
\end{enumerate}
\end{prop}

\begin{proof}
First, set $h_{0,x}=\id_{\overline C_{0,-1}(x)}$  for every $x\in X_0^{\pm,\Phi}$, so that~\ref{i: hnx zero phi} is satisfied. 

Now we are going to  give a different definition of $h_{0,x}$  for $x\in A_m\setminus A_{m-1}$, where 
\[
A_m=\bigcupdot_{y\in \mathfrak{X}_{m}} D_{0}(y, \Gamma_m^+)\cap X_0 \setminus  X_0^{\Phi}
\] 
for $m\geq 0$, and $A_{-1}=\emptyset$. Note that the sets used in this expression of $A_m$ are disjoint  by \Cref{p: xxn}~\ref{i: xxn subset}, since $\mathfrak s_m\ge\Gamma_m^+$ by~\eqref{2rl + sn < sl} and~\eqref{no bm multi}.
This will complete  the definition of $h_{0,x}$ for all $x\in X_0$ because $X_0=\bigcup_{m\ge 0}A_m$ since $p\in\mathfrak{X}_{m}$ (\Cref{p: xxn}~\ref{i: xxn subset})  and $\Gamma_m^+\uparrow\infty$.
Moreover~\ref{i: hnx zero xxn} is a direct consequence of~\ref{i: hnx zero phi} and~\ref{i: hnx zero hh}, and therefore we will  only have to check~\ref{i: hnx zero hh}.

Let $x\in A_m\setminus A_{m-1}$ for $m\ge 0$. On the one hand, if 
\[
x\in  \left(D_0(p, \Gamma^+_m) \setminus X_0^\Phi\right)\setminus A_{m-1}\;,
\]
then let $h_{0,x}$ be any $0$-equivalence $\rep_0(x)\to x$. On the other hand, if
\[
x\in \left(D_{0}(y, \Gamma^+_m) \setminus  X_0^{\Phi}\right)\setminus A_{m-1}
\]
for some $y\in \mathfrak{X}_{m}\setminus\{p\}$, then $\rep_0(x)\in D_0(p,\Gamma^+_m)$ by \Cref{l: x0phi}~\ref{i: x0phi dn}, and let $h_{0,x}=\mathfrak{h}_{m,y} h_{0, \mathfrak{h}_{m,y}^{-1}(x)}$. Note that this composite is well defined because
\[
\im h_{0, \mathfrak{h}_{m,y}^{-1}(x)} = D_{-1}(x,r_0^\pm) \subset D_{-1}(x,R_0^\pm) \subset \dom \mathfrak{h}_{m,y}
\]
by \Cref{l: rmn hmy zero}~\ref{i: rmn hmy zero ball} and~\eqref{r_n^pm le R_n^pm}.
Property~\ref{i: hnx zero hh} is obvious with this definition of $h_{0,x}$.
\end{proof}

Now, given any integer $n>0$, suppose that we have already defined the equivalence relations $\sim_m$, the sets $X_m^\Phi$, and  maps $\rep_m$ and $h_{m,x}$ for $0\leq m <n$.
Let 
\[
\mathcal C_{n,-1}(x)=  \bigcup_{v\in D_{n}(x,n)} \overline C_{n,-1}(v)\;, \quad 
\mathcal C_{n,n-1}(x)=  \bigcup_{v\in D_{n}(x,n)} \overline C_{n,n-1}(v)\;.
\]

\begin{defn}\label{d. equivalence}
For $n\in \N$ and $x,y\in X_n^\pm$, a pointed graph isomorphism $f\colon \left( \mathcal C_{n,-1}(x),x\right)\to \left(\mathcal C_{n,-1}(y),y\right)$  is called an $n$-\emph{equivalence} from $x$ to $y$, denoted by $f \colon x \to y$, if it satisfies the following properties for $0\leq m<n$ and $v\in D_{n}(x,n)$:
\begin{enumerate}[(i)]
\item \label{i: equivalence ball} We have $f(D_{n}(x,n))= D_{n}(f(x),n)$.

\item \label{i: equivalence v} We have $f(\overline C_{n,n-1}(v))= \overline C_{n,n-1}(f(v))$ and $f( C_{n,n-1}(v))=  C_{n,n-1}(f(v))$.

\item \label{i: equivalence cnm} We have 
\[
f\left (X_{n-1}^\pm\cap \mathcal C_{n,n-1}(x) \right)=X_{n-1}^\pm\cap  \mathcal C_{n,n-1}(y)\;,
\]
 and 
\[
f\colon \left(\mathcal C_{n,n-1}(x),\chi_{n-1}\right) \to \left( \mathcal C_{n,n-1}(y),\chi_{n-1}\right)
\] is a color-preserving  graph isomorphism with respect to $E_{n-1}$.

\item \label{i: equivalence xxm} We have 
\[
f\left (\xxyn\cap \mathcal C_{n,n-1}(x) \right)=\xxyn \cap  \mathcal C_{n,n-1}(y)\;.
\]

\item \label{i: equivalence cmm-1} For all $u\in  \CPen_{n-1}(C_{n,n-1}(x),1)$, the restriction $f\colon  \mathcal C_{n-1,-1}(u) \to  \mathcal C_{n-1,-1}(f(u))$ is  $h_{n-1,f(u)} h_{n-1,u}^{-1}$; in particular, it is an $(n-1)$-equivalence.

\end{enumerate}
\end{defn}

\begin{rem}
Note that $X_{n-1}^\pm\cap  \overline C_{n,n-1}(x),\overline  C_{n-1,-1}(u)\subset \overline C_{n,-1}(x)$ by~\eqref{cluster intermediate}.
\end{rem}

\begin{rem}
For $u\in  \CPen_{n-1}(C_{n,n-1}(x),1)$ and $v\in D_{n-1}(u,n-1)$, we get  $d_{n}(x,\pi_n(v))\leq n$ by \Cref{p: xn}~\ref{i: xn dn dn-1} and the definition of $E_n$.
So $\overline C_{n-1,-1}(v)\subset \dom f$ in \Cref{d. equivalence}~\ref{i: equivalence cmm-1}.
\end{rem}

The following lemma is an immediate consequence of \Cref{d. equivalence}.

\begin{lem}\label{l: equivalence composition}
For $n\in \N$, the family of $n$-equivalences  between points of $X_n^\pm$ is closed by the operations of composition and inversion of maps.
\end{lem}
 
According to \Cref{l: equivalence composition}, for $n\in \N$, an equivalence relation $\sim_n^\pm$ on $\xn^\pm$ is  defined by declaring $x\sim_n^\pm y$ if there is some $n$-equivalence $x\to y$.
Let $\Phi_n$ be the map defined on $X_n=X_n^+\cupdot X_n^-$ that sends every  $x\in X_n^{\pm}$ to its $\sim_n^\pm$-equivalence class. 
The range of each of these maps is obviously finite. 

\begin{lem}\label{l: xnphi}
For  $n\in \N$, there are disjoint subsets $\xn^{-,\Phi}, \xn^{+,\Phi}\subset X_n$ satisfying the following properties:
\begin{enumerate}[(a)]
\item \label{i: xnphi maximal} The sets $\xn^{\pm,\Phi}$ are maximal among the subsets of $\xn^\pm$ where $\Phi_n$ is injective.
\item \label{i: xnphi dn} For $u\in \xn^{\pm,\Phi}$ and $v\in \xn^\pm$, if $\Phi_n(u)=\Phi_n(v)$, then $d_n(u,p)\leq d_n(v,p)$. 
\end{enumerate} 
\end{lem}

\begin{proof}
In every $\sim_n^\pm$-equivalence class, take  a representative that minimizes the $d_n$-distance  to $p$.
\end{proof}

By \Cref{l: xnphi}, for any $x\in \xn^{\pm}$, there is a unique $u\in \xn^{\pm,\Phi}$ with $\Phi_n(x)= \Phi_n(u)$. Let $\rep_n^\pm \colon \xn^\pm \to \xn^{\pm,\Phi}$ be the maps determined by this correspondence, and let $\rep_n\colon X_n\to X_n^\Phi:=\xn^{+,\Phi}\cupdot \xn^{-,\Phi}$ be their union.

\begin{lem}\label{l: rmn hmy}
For  all $(m,y)\in \mathfrak{P}_{n-1}$ and $x\in X_n^\pm\cap D_{n}(p,\Gamma^+_m)$,  the following properties hold:
\begin{enumerate}[(a)]

\item \label{i: rmn hmy ball} $ \mathcal C_{n,-1}(v)\subset \dom\mathfrak{h}_{m,y}$.

\item\label{i: rmn hmy iso} The map $\mathfrak{h}_{m,y}$ restricts to an $n$-equivalence $x\to\mathfrak{h}_{m,y}(x)$; in particular, $x\sim_n \mathfrak{h}_{m,y}(x)$ and $p\sim_n y$. 

\end{enumerate} 
\end{lem}

\begin{proof}
 By \Cref{l: cn-1 gamma}, $\overline C_{n,-1}(v)\subset D_{-1}(v,\Gamma^+_n)$ for every $v\in D_{n}(x,n)$.
 Using the triangle inequality, we get
 \begin{equation}\label{rmn hmy}
 \mathcal C_{n,-1}(v)\subset D_{-1}(x,\Gamma^+_n+ nL_n)\subset D_{-1}\left(p,L_{n}(\Gamma_m^+ +n) + \Gamma^+_n \right)\;.
 \end{equation}
By~\eqref{rn} and~\eqref{no bm multi}, we have
\[
\mathfrak{r}_m > 4L_m(\Gamma_m^++m) + K_m\;.
\]
The assumption $(m,y)\in \mathfrak{P}_{n-1}$ implies $m\geq n$ according to~\eqref{rrn}.
So  $L_m \geq L_n > L_{n-1}$ by~\eqref{defn functions n} and~\eqref{no bm multi}, $K_m\geq K_n>K_{n-1}$ by~\eqref{defn bmkn},~\eqref{defn kn} and~\eqref{no bm multi}, and  $\Gamma_m^+\geq R_n^+$ by~\eqref{gamma r}.
Therefore 
\[
\mathfrak{r}_m - K_{n} > 4L_m(\Gamma_m^+ +m)+ K_m - K_{n}  > L_{n}(\Gamma_m^+ +n) + \Gamma^+_n\;.
\]
Then~\eqref{rmn hmy} yields
\begin{equation}\label{inclusion rmn hmy}
 \mathcal C_{n,-1}(v)\subset D_{-1}\left(p,\mathfrak{r}_m-K_{n}\right)\;,
\end{equation}
completing the proof of~\ref{i: rmn hmy ball} because $\dom\mathfrak{h}_{m,y}=D_{-1}(p,\mathfrak{r}_m)$.

Let us prove~\ref{i: rmn hmy iso}. 
We proceed by induction on $n$.
For $n=0$, the result follows from \Cref{l: rmn hmy zero}~\ref{i: rmn hmy zero iso}.
So suppose that, given some $n>0$, the result is true for $0\leq m<n$.
\Cref{d. equivalence}~\ref{i: equivalence ball} follows from \Cref{p: xn}~\ref{i: xn partial} and~\eqref{inclusion rmn hmy}.
By \Cref{l: h preserves Cn},~\eqref{h preserves overlinec} and~\eqref{inclusion rmn hmy}, we get  $\mathfrak{h}_{m,y}( \overline C_{n,n-1}(u) )= \overline C_{n,n-1}(\mathfrak{h}_{m,y}(u))$ and $\mathfrak{h}_{m,y}( C_{n,n-1}(u) )=  C_{n,n-1}(\mathfrak{h}_{m,y}(u))$ for every $v\in D_{n}(x,n)$ and $u\in\overline C_{n,l}(v)$.
Thus \Cref{d. equivalence}~\ref{i: equivalence v} is satisfied.
The map $\mathfrak{h}_{m,y}\colon \mathcal C_{n,n-1}(v)\to \mathcal C_{n,n-1}(w)$ is a graph isomorphism that preserves $\chi_{n-1}$ by Propositions~\ref{p: xn}~\ref{i: xn partial} and~\ref{p: chin}~\ref{i: chin h}.
Therefore
\[
\mathfrak h_{m,y}(X_m^\pm\cap \mathcal C_{n,n-1}(x))=X_m^\pm\cap \mathcal C_{n,n-1}(y)
\] by \Cref{p: xn}~\ref{i: xn rn-1},\ref{i: xn partial}.
Hence $\mathfrak{h}_{m,y}$ satisfies \Cref{d. equivalence}~\ref{i: equivalence cnm}.
\Cref{d. equivalence}~\ref{i: equivalence cmm-1} follows by the induction hypothesis.
By \Cref{p: xxn}~\ref{i: xxn cap bmn}, we have $\mathfrak{X}_{n-1}\cap D(y,\mathfrak{r}_l)=\mathfrak{h}_{m,y}( \mathfrak{X}^m_{n-1} )$ for each $(m,y)\in \mathfrak{P}_{n-1}$.
In particular, for $(m,y)=(m,p)$, we obtain  $ \mathfrak{X}^m_{n-1}= \mathfrak{X}_{n-1}\cap D(p,\mathfrak{r}_m)$.
So  
\[
\mathfrak{X}_{n-1}\cap  D(y,\mathfrak{r}_l)=\mathfrak{h}_{m,y}(\mathfrak{X}_{n-1}\cap D(p,\mathfrak{r}_l))\;,
\]  and \Cref{d. equivalence}~\ref{i: equivalence xxm} follows using~\eqref{rmn hmy} and~\ref{i: rmn hmy ball}, since $\mathfrak{r}_m\geq R_n^+\geq r_n^\pm$ according to~\eqref{gamma r}--\eqref{r_n^pm le R_n^pm}.
Therefore $\mathfrak{h}_{m,y}$ satisfies \Cref{d. equivalence}~\ref{i: equivalence xxm}.
This completes the proof of~\ref{i: rmn hmy iso}.
\end{proof}

\begin{prop}\label{p: hnx}
For  $n\in \N$ and $x\in \xn$, there is an $n$-equivalence $h_{n,x}\colon \rep_n(x) \to x$ satisfying the following properties:
\begin{enumerate}[(i)]
\item \label{i: hnx phi} If $x\in X_n^{\Phi}$, then $h_{n,x}$ is the identity on $\mathcal C_{n,-1}(x)$.

\item \label{i: hnx hh} For $(m,y)\in \mathfrak{P}_{n-1}$ and $x\in X_n\cap D_n(y, \Gamma_m^+)$, we have $h_{n,x}=\mathfrak{h}_{m,y} h_{n, \mathfrak{h}_{m,y}^{-1}(x)}$.

\item \label{i: hnx xxn} If $x\in \xxn$, then $h_{n,x}= \mathfrak{h}_{n,x}$ on $\mathcal C_{n,-1}(x)$.
\end{enumerate}
\end{prop}

\begin{proof}
First, define $h_{n,x}$ as the identity on $\mathcal C_{n,-1}(x)$ for every $x\in X_n^{\Phi}$, so that~\ref{i: hnx phi} is satisfied. 

Next we are going to give a different definition of  $h_{n,x}$  for $x\in A_m\setminus A_{m-1}$, where
\[
A_m=\bigcupdot_{y\in \mathfrak{X}_{m}} D_{n}(y, \Gamma_m^+)\cap X_n \setminus  X_n^{\Phi}
\] 
for $m\geq n$, and $A_{n-1}=\emptyset$. Note that the sets used in this expression of $A_m$ are disjoint  by \Cref{p: xxn}~\ref{i: xxn subset}, since $\mathfrak s_m\ge\Gamma_m^+$ by~\eqref{2rl + sn < sl} and~\eqref{no bm multi}.
This will complete  the definition of $h_{n,x}$ for all $x\in X_n$ because $X_n=\bigcup_{m\ge n}A_m$ since $p\in\mathfrak{X}_{m}$ (\Cref{p: xxn}~\ref{i: xxn subset})  and $\Gamma_m^+\uparrow\infty$.
Moreover~\ref{i: hnx xxn} is a direct consequence of~\ref{i: hnx phi} and~\ref{i: hnx hh}, and therefore we will  only have to check~\ref{i: hnx hh}.

Let $x\in A_m\setminus A_{m-1}$ for $m\ge n$. On the one hand, if 
\[
x\in  \left(D_n(p, \Gamma^+_m) \cap\xn\setminus\xn^\Phi\right)\setminus A_{m-1}\;,
\]
then let $h_{n,x}\colon \rep_n(x)\to x$ be any $n$-equivalence, whose existence is guaranteed by the definition of $\rep_n$. On the other hand, if
\[
x\in \left(D_{n}(y, \Gamma^+_m)\cap X_n \setminus  X_n^{\Phi}\right)\setminus A_{m-1}
\]
for some $y\in \mathfrak{X}_{m}\setminus\{p\}$, then $\rep_n(x)\in D_n(p,\Gamma^+_m)$ by Lemmas~\ref{l: x0phi}~\ref{i: x0phi dn} and~\ref{l: xnphi}~\ref{i: xnphi dn}. In this case, take $h_{n,x}=\mathfrak{h}_{m,y} h_{n, \mathfrak{h}_{m,y}^{-1}(x)}$, which is well defined because, for $x\in X_n^\pm$,
\[
\im h_{n, \mathfrak{h}_{m,y}^{-1}(x)} = D_{n-1}(x,r_n^\pm) \subset D_{n-1}(x,R_n^\pm) \subset \dom \mathfrak{h}_{m,y}
\]
 by \Cref{l: rmn hmy}~\ref{i: rmn hmy ball} and~\eqref{r_n^pm le R_n^pm}.
Property~\ref{i: hnx hh} is obvious with this definition of $h_{n,x}$.
\end{proof}

\begin{rem}
In accordance with the discussion at the beginning of \Cref{s. xn}, only \Cref{p: hnx}~\ref{i: hnx phi} is needed to prove \Cref{t: finitary}~\ref{i: finitary limit aperiodic}, whereas the whole \Cref{p: hnx} is needed to prove \Cref{t: finitary}~\ref{i: finitary repetitive}.
\end{rem}

\begin{rem}
Note that the definitions of $\sim_n^\pm$, $\Phi_n$ e $\rep_n^\pm$, and the properties of $\xn^{\pm,\Phi}$ already guarantee the existence of $n$-equivalences $h_{n,x}$.
Moreover there is no problem to assume~\ref{i: hnx phi} and~\ref{i: hnx xxn}.
So the really new contribution of \Cref{p: hnx} is~\ref{i: hnx hh}.
\end{rem}

\subsection{Weak equivalences}

Next we introduce another notion of equivalence very similar to that of $n$-equivalence. We need both concepts due to the way  we prove the crucial \Cref{l: h preserves}. In that result,  we will first prove that a certain map is an $n$-weak equivalence, concluding that it is in fact an $n$-equivalence over a smaller domain.

\begin{defn}\label{d. weak equivalence zero}
For $x,y\in X_0$, a $0$-\emph{weak equivalence} from $x$ to $y$, denoted by $f\colon x \to y$, is a pointed graph isomorphism $(D_{-1}(x,r_0^\pm),x) \to (D_{-1}(y,r_0^\pm),y)$.
\end{defn}

Let $\widehat \sim_0^\pm$ be the equivalence relation on $X_0^\pm$  defined by declaring $x\widehat \sim_0^\pm y$  if there is some $0$-weak equivalence $x\to y$. 
Let $\widehat\Phi_0$ be the map defined on $X_0=X_0^+\cupdot X_0^-$ that sends every  $x\in X_0^{\pm}$ to its $\widehat \sim_0^\pm$-equivalence class. 
The range of each of these maps  is obviously finite. 

The next result follows easily from \Cref{l: cnn-1 contained}.

\begin{lem}\label{l: weak}
Any $0$-equivalence $x\to y$ restricts to a $0$-weak equivalence $x\to y$; in particular, $x\sim_0 y$ implies $x \widehat{\sim}_0 y$.
\end{lem}

\begin{lem}\label{l: weak x0phi}
For  $n\in \N$, there are disjoint subsets $X_0^{-,\widehat \Phi}, X_0^{+,\widehat \Phi}\subset X_0$ satisfying the following properties:
\begin{enumerate}[(a)]

\item \label{i: weak x0phi maximal} The sets $X_0^{\pm,\widehat \Phi}$ are maximal among the subsets of $X_0^\pm$ where $\widehat \Phi_0$ is injective.

\item \label{i: weak x0phi dn} For $u\in X_0^{\pm,\widehat \Phi}$ and $v\in X_0^\pm$, if $\widehat \Phi_0(u)=\widehat \Phi_0(v)$, then $d_0(u,p)\leq d_0(v,p)$. 

\item \label{i: weak x0phi hat} We have $X_0^{\pm, \widehat \Phi}\subset X_0^{\pm, \Phi}$.
\end{enumerate} 
\end{lem}

\begin{proof}
In every $\widehat\sim_0^\pm$-equivalence class, take  a representative that minimizes the $d_0$-distance to $p$.
\end{proof}

By \Cref{l: weak x0phi}, for any $x\in X_0^{\pm}$, there is a unique $u\in X_0^{\pm,\widehat \Phi}$ satisfying $\widehat \Phi_0(x)= \widehat \Phi_0(u)$. Let $\widehat \rep_0^\pm \colon X_0^\pm \to X_0^{\pm,\widehat \Phi}$ be the maps determined by this correspondence, and let $\widehat \rep_0\colon X_0\to X_0^{\widehat \Phi}:=X_0^{+,\widehat \Phi}\cupdot X_0^{-,\widehat \Phi}$ be their union.

The following lemma follows from \Cref{l: rmn hmy zero,l: weak}.

\begin{lem}\label{l: weak rmn hmy zero}
For  all $(m,y)\in \mathfrak{P}_{-1}$ and $x\in X_0^\pm\cap D_{0}(p,\Gamma^+_0)$,  the following properties hold:
\begin{enumerate}[(a)]

\item \label{i: weak rmn hmy zero ball} $D_{-1}(x,r_0^\pm)(x)\subset \dom\mathfrak{h}_{m,y}$.

\item\label{i: weak rmn hmy zero iso} The map $\mathfrak{h}_{m,y}$ restricts to a $0$-weak equivalence $x\to\mathfrak{h}_{m,y}(x)$; in particular, $x\, \widehat \sim_0 \,\mathfrak{h}_{m,y}(x)$ and $p\,\widehat \sim_0 \,y$.

\end{enumerate} 
\end{lem}

\begin{prop}\label{p: weak hnx zero}
For any $x\in X_0^\pm$, there is a $0$-weak equivalence $\hat h_{0,x}:\widehat\rep_0(x)\to x$ satisfying the following properties:
\begin{enumerate}[(i)]
\item \label{i: weak hnx zero phi} If $x\in X_0^{\pm,\widehat \Phi}$, then $\hat h_{0,x}$ is the identity on  $D_{-1}(x,r_0^\pm)(x)$.

\item \label{i: weak hnx zero hh} For all $x\in X_0^{\pm}$, we have  $\hat h_{0,x}=  h_{0,x} \hat h_{0,\rep_0(x)}$.
\end{enumerate}
\end{prop}

\begin{proof}
First, for every $x\in X_0^{\pm, \widehat \Phi}$, let $\hat h_{0,x}$  be the identity on $D_{-1}(x,r_0^\pm)$.
Then, for points $x\in X_0^{\pm, \Phi}\setminus X_0^{\pm,\widehat \Phi}$, let $\hat h_{0,x}\colon \widehat \rep_0(x)\to x$ be any $0$-weak equivalence.
Finally, for every $x\in X_0\setminus X_0^{\pm, \Phi}$, let $\hat h_{0,x}=  h_{0,x} \hat h_{0,\rep_0(x)}$.
\end{proof}

Now, given any integer $n>0$, suppose that we have already defined the equivalence relations $\widehat\sim_m$, the sets $X_m^{\widehat\Phi}$, and the maps $\widehat\rep_m$ and $\hat h_{m,x}$ for $0\leq m <n$.
For $x\in X_n^\pm$,  let 
\[
\mathscr{C}_{n}(x)= \bigcup_{u\in D_{n-1}(x,r_n^\pm)} \overline C_{n-1,-1}(u)\;.
\]

\begin{defn}\label{d. weak equivalence}
For $n\in \N$ and $x,y\in X_n^\pm$, a pointed graph isomorphism $f\colon( \mathscr{C}_{n}(x),x)\to(\mathscr{C}_{n}(y),y)$  is called an $n$-\emph{weak equivalence} from $x$ to $y$, denoted by $f:x\to y$, if it satisfies the following properties for $0\leq m<n$ and $v\in D_{n}(x,n)$:
\begin{enumerate}[(i)]
\item \label{i: weak equivalence ball} We have $f(D_{n-1}(x,r_n^\pm))= D_{n-1}(y,r_n^\pm)$.

\item \label{i: weak equivalence cnm} We have 
\[
f\left (X_{n-1}^\pm\cap D_{n-1}(x,r_n^\pm) \right)=X_{n-1}^\pm\cap D_{n-1}(y,r_n^\pm)\;,
\]
and 
\[
f\colon \left( D_{n-1}(x,r_n^\pm),\chi_{n-1}\right) \to \left(  D_{n-1}(y,r_n^\pm),\chi_{n-1}\right)
\] 
is a color-preserving  graph isomorphism with respect to $E_{n-1}$.

\item \label{i: weak equivalence xxm} We have \[
f\left (\xxyn\cap D_{n-1}(x,r_n^\pm) \right)=\xxyn \cap  D_{n-1}(x,r_n^\pm)\;.
\]

\item \label{i: weak equivalence cmm-1} For every $u\in  D_{n-1}(x,r_n^\pm-1)$,  the restriction $f\colon  \mathcal C_{n-1}(u) \to  \mathcal C_{n-1}(f(u))$ equals $h_{n-1,f(u)} h_{n-1,u}^{-1}$; in particular, it is an $(n-1)$-equivalence.

\end{enumerate}
\end{defn}

\begin{rem}
Note that, for $n>0$, $x\in X_n$ and $u\in  D_{n-1}(x,r_n^\pm-1)$, we have $\mathcal C_{n-1}(u)\subset \mathscr C_{n}(x)$ because $D_{n-1}(u,1)\subset D_{n-1}(x,r_n^\pm)$.
\end{rem}

The following lemma is an immediate consequence of \Cref{d. equivalence,d. weak equivalence}.

\begin{lem}\label{l: weak equivalence composition}
The family of $n$-weak equivalences between points of $X_n^\pm$ is closed by the operations of composition and inversion of maps. Moreover the composition of an $n$-weak equivalence and an $n$-equivalence is an $n$-weak equivalence; in particular, every $n$-equivalence is an $n$-weak equivalence.
\end{lem}

According to \Cref{l: weak equivalence composition}, for $n\in \N$, an equivalence relation $\widehat \sim_n^\pm$ on $\xn^\pm$ is  defined by declaring $x\widehat \sim_n^\pm y$ if there is some $n$-weak equivalence $x\to y$. 
Let $\widehat \Phi_n$ be the map defined on $X_n=X_n^+\cupdot X_n^-$ that sends every  $x\in X_n^{\pm}$ to its $\widehat\sim_n^\pm$-equivalence class. 
The range of each of these maps is obviously finite. 

\begin{lem}\label{l: weak xnphi}
For  $n\in \N$, there are disjoint subsets $\xn^{-,\widehat \Phi}, \xn^{+,\widehat \Phi}\subset X_n$ satisfying the following properties:
\begin{enumerate}[(a)]
\item \label{i: weak xnphi hat} We have $X_n^{\pm, \widehat \Phi}\subset X_n^{\pm, \widehat \Phi}$.

\item \label{i: weak xnphi maximal} The sets $\xn^{\pm,\widehat \Phi}$ are maximal among the subsets of $\xn^\pm$ where $\widehat \Phi_n$ is injective.
\item \label{i: weak xnphi dn} For $u\in \xn^{\pm,\widehat \Phi}$ and $v\in \xn^\pm$, if $\widehat \Phi_n(u)=\widehat \Phi_n(v)$, then $d_n(u,p)\leq d_n(v,p)$. 
\end{enumerate} 
\end{lem}

\begin{proof}
In every $\widehat \sim_n^\pm$-equivalence class, take  a representative that minimizes the $d_n$-distance  to $p$.
\end{proof}

By \Cref{l: xnphi}, for every $x\in \xn^{\pm}$, there is a unique  $u\in \xn^{\pm,\widehat \Phi}$ with $\widehat \Phi_n(x)= \widehat \Phi_n(u)$. Let $\widehat \rep_n^\pm \colon \xn^\pm \to \xn^{\pm,\widehat \Phi}$ be the maps determined by this correspondence, and let $\widehat \rep_n\colon X_n\to X_n^{\widehat \Phi}:=\xn^{+,\widehat \Phi}\cupdot \xn^{-,\widehat\Phi}$ be their union.

The following result follows from \Cref{l: weak rmn hmy zero,l: weak equivalence composition}.

\begin{lem}\label{l: weak rmn hmy}
For  all $(m,y)\in \mathfrak{P}_{-1}$ and $x\in X_n^\pm\cap D_{n}(p,\Gamma^+_0)$,  the following properties hold.
\begin{enumerate}[(a)]
\item \label{i: weak rmn hmy ball} $\mathscr C_{n}(x)\subset \dom\mathfrak{h}_{m,y}$.
\item\label{i: weak rmn hmy iso} The map $\mathfrak{h}_{m,y}$ restricts to an $n$-weak equivalence $x\to\mathfrak{h}_{m,y}(x)$; in particular, $x\,\widehat\sim_n\,\mathfrak{h}_{m,y}(x)$ and $p\,\widehat\sim_n\,y$.
\end{enumerate} 
\end{lem}

\begin{prop}\label{p: weak hnx}
For every $x\in X_n^\pm$, there is an $n$-weak equivalence $\hat h_{n,x}\colon \widehat \rep_n(x)\to x$
satisfying the following properties:
\begin{enumerate}[(i)]

\item \label{i: weak hnx phi} If $x\in X_n^{\pm,\widehat\Phi}$, then $\hat h_{n,x}$ is the identity on  $D_{n-1}(x,r_n^\pm)$.

\item \label{i: weak hnx hh} For all $x\in X_n^{\pm}$, we have $\hat h_{n,x}=  h_{n,x} \hat h_{n,\rep_n(x)}$.

\end{enumerate}
\end{prop}

\begin{proof}
The proof is identical to that of \Cref{p: weak hnx zero}.
\end{proof}

\subsection{BFS-orderings}

We introduce a special kind of orderings on graphs that are used to produce aperiodic colorings. They are essentially a reformulation of the \emph{breadth-first search spanning trees} in \cite{CollinsTrenk2006}.

\begin{defn}\label{d. parent}
Let $(A,x)$ be a pointed connected graph with finite vertex degrees endowed with an order relation $\leq$. 
Define the \emph{parent map}, $\Pa\colon  A\setminus \{x\} \to A$, by
\begin{equation}\label{defn pa}
\Pa(u)=\min S(u,1)\;.
\end{equation}
For $v\in A$, its \emph{children set}, denoted by $\Ch(v)$,  is 
\begin{equation}\label{defn ch}
\Ch(v)=\Pa^{-1}(v)=S(v,1) \setminus\bigg(\bigcup_{w< v} S(w,1)\cup  \{x\}\bigg)\;.
\end{equation}
\end{defn}

\begin{defn}\label{d. dfs}
A \emph{BFS-ordering} on a pointed connected graph $(A,x)$ is an order $\trianglelefteq$ on $A$ satisfying the following conditions for all $u,v\in A$:
\begin{enumerate}[(i)]

\item \label{i: dfs distance} If $d(x,u)<d(x,v)$, then $u \vartriangleleft v$; in particular, $x$ is the least element of $(A,\trianglelefteq)$.

\item \label{i: dfs children} If $u,v\neq x$ and $\Pa(u)\vartriangleleft \Pa(v)$,  then $u\vartriangleleft v$.

\end{enumerate} 
\end{defn}

\begin{lem}\label{l: dfs extension}
Let $(A,x)$ be a pointed connected graph with finite vertex degrees.
Then there is a  BFS-ordering $\trianglelefteq$ on $(A,x)$.
\end{lem}

\begin{proof}
By induction on $n\in \N$,  a BFS-ordering $\trianglelefteq$ can be defined on every disk $D(x,n)$ as follows. 
It is trivially defined on $D(x,0)$. Then, assuming that $n\ge1$ and $\trianglelefteq$ is defined on $D(x,n-1)$, extend the definition of $\trianglelefteq$ to $D(x,n)$ so that $D(x,n-1)$ an initial segment, and its restriction to $S(x,n)$ is any order satisfying 
\[
\min (S(u,1)\cap D(x,n-1)) \triangleleft \min( S(v,1)\cap D(x,n-1))\Rightarrow u\triangleleft v\;.\qedhere
\]
\end{proof}

Given an isomorphism of graphs, $f\colon A\to B$, and an order relation $\leq_A$ on $A$ (${\leq_A}\subset A\times A$), the corresponding push-forward order relation on $B$ is $(f\times f)(\le_A)\subset B\times B$, which is simply denoted by $f(\leq_A)$.

Recall that $C_{n,n-1}(x)$ is a connected subgraph of $(X_{n-1},E_{n-1})$ by \Cref{l: cnn-1 connected}.
Consider the $n$-equivalences $h_{n,x}$, for $n\in \N$ and $x\in \xn$, given by \Cref{p: hnx}.

\begin{prop}\label{p: defn dfs}
For any  $n\in \N$ and $x\in \xn$, there is a BFS-ordering $\trianglelefteq_{n,x}$ on the pointed connected graph $(C_{n,n-1}(x),x)$ such that  $\trianglelefteq_{n,x}=h_{n,\rep_n(x)}(\trianglelefteq_{n,\rep_n(x)})$.
\end{prop}

\begin{proof}
Take any BFS-ordering $\trianglelefteq_{n,x}$ on $(C_{n,n-1}(x),x)$ for $x\in X_n^{\Phi}$ (\Cref{l: dfs extension}). Then define $\trianglelefteq_{n,x}=h_{n,\rep_n(x)}(\trianglelefteq_{n,\rep_n(x)})$ for $x\in X_n \setminus X_n^{\Phi}$.
\end{proof}

From now on, for every $n\in\N$ and $x\in X_n$, the notation $\Pa_{n,x}$ and $\Ch_{n,x}$ is used for the parent map and children sets on the pointed connected graph $(\cnyn(x),x)$, with the BFS-ordering $\trianglelefteq_{n,x}$ given by \Cref{p: defn dfs}.

\begin{lem}\label{l: children}
Let $n\in \N$ and $x\in \xn$. 
The following properties hold for every $u\in C_{n,n-1}(x)$: 
\begin{enumerate}[(a)]

\item \label{i: children distance} If $u\neq x$, then  $d_{n-1}(x,\Pa_{n,x}(u))=d_{n-1}(x,u)-1$.

\item \label{i: children equality} We have 
\[
\bigcupdot_{v\in \cnyn(x)} \Ch_{n,x}(v) = C_{n,n-1}(x)\setminus \{x\}\;.
\]

\item \label{i: children cardinality} If $u\neq x$, then $|\Ch_{n,x}(u)|\leq \Delta_{n-1}-1$.

\end{enumerate}
\end{lem}

\begin{proof}
Property~\ref{i: children distance} is an easy consequence of \Cref{d. parent,d. dfs}~\ref{i: dfs distance}.
Property~\ref{i: children cardinality} follows from~\ref{i: children distance} and \Cref{d. parent}, whereas~\ref{i: children equality} is obvious.
\end{proof}

\subsection{Adapted colorings for $n=0$}

In the outline of the proof of \Cref{t: finitary} given in \Cref{s. idea}, it was said that we needed to construct many colorings $\psi^{i}_{n,x}$ on the clusters $C_{n,n-1}(x)$ that break the symmetries of the cluster. These are the building blocks that will be used to  construct the colorings of the statement of \Cref{t: finitary}.

\begin{defn}\label{d. adapted 0}
For $x\in X_0$, a coloring $\psi\colon \czeroone(x)\to [\Delta]$ is said to be \emph{adapted} if it satisfies the following two conditions:
\begin{enumerate}[(i)]
\item \label{i: adapted 0 center} There is a  geodesic segment in $(X_{-1},E_{-1})$ of the form $\tau=(x=\tau_0,\dots,\tau_5)$  such that
\[
\psi^{-1}(0)\cap D_{-1}(x,7) = 
\begin{cases}
\{\tau_0,\tau_1,\tau_2,\tau_5\} & \text{if $x\in X_0^-$} \\
\{\tau_0,\tau_1,\tau_2,\tau_4,\tau_5\} &\text{if $x\in X_0^+$}\;.
\end{cases}
\]
\item \label{i: adapted 0 children} For all $u\in \czeroone(x)$, the coloring $\psi$ is injective on $\Ch_{0,x}(u)$.
\setcounter{nameOfYourChoice}{\value{enumi}}
\end{enumerate}
It is said that  $\psi$ is \emph{strongly adapted} if it is adapted and moreover the following property holds:
 \begin{enumerate}[(i)]
 \setcounter{enumi}{\value{nameOfYourChoice}}
 \item \label{i: adapted strongly} We have $\psi^{-1}(0)\setminus D_{-1}(x,7) =\emptyset$.
\end{enumerate}
\end{defn}

\begin{lem}\label{l: existence phix}
For every $x\in X_0^\pm$, there is a strongly adapted coloring $\psi_x\colon \czeroone(x)\to [\Delta] $.
\end{lem}
\begin{proof}
First, choose  a geodesic  segment in $(X_{-1},E_{-1})$ of the form $\tau=(x=\tau_0,\dots,\tau_5)$, which is contained in $C_{0,-1}(x)$ because $D_{-1}(x,r_0^\pm)\subset C_{0,-1}(x)$ (\Cref{l: cnn-1 contained}), and $r_0^\pm>2^{11}$ by~\eqref{r_0 > 2^11} and~\eqref{defn rnpm}. 

Color with the color $0$ the set $T_0^-=\{\tau_0,\tau_1,\tau_2,\tau_5\}$ if $x\in X_0^-$, or the set
$T_0^+=\{\tau_0,\tau_1,\tau_2,\tau_4,\tau_5\}$ if $x\in X_0^+$.
In both cases $\psi_x(x)=0$. 
The sets $\Ch_{0,x}(u)$, for $u\in \czeroone(x)$, form a partition of $\czeroone(x)\setminus \{x\}$ by \Cref{l: children}~\ref{i: children equality}.
Moreover $|\Ch_{0,x}(u)\setminus T_0^\pm|\leq \Delta -1$ by \Cref{l: children}~\ref{i: children cardinality}.
So, for each $u\in \czeroone(x)$, we can color the points in $\Ch_{0,x}(u)\setminus T_0^\pm$ with different colors from $\{1,\dots,\Delta-1\}$. 
This procedure defines a coloring $\psi_x\colon \czeroone(x)\to [\Delta]$ satisfying all conditions of \Cref{d. adapted 0}.
\end{proof}

For a colored graph $(X,\phi)$ and a graph isomorphism $h:X\to Y$, the notation $h(\phi)$ is used for the corresponding pushforward coloring $(h^{-1})^*\phi$  of $Y$.

\begin{prop}\label{p: phi0x0}
There is a family of strongly adapted colorings $\psi_{0,x}^0\colon \czeroone(x)\to [\Delta]$, for $x\in X_0$, satisfying $\psi_{0,x}^0=h_{0,x}(\psi_{0,\rep_0(x)}^0)$.
\end{prop}

\begin{proof}
If $x\in X_0^{\Phi}$, take any strongly adapted coloring (\Cref{l: existence phix}).
If $x\in X_0\setminus X_0^{\Phi}$, let $\psi_{0,x}^0= h_{0,x}(\psi_{0,\rep_0(x)}^0)$.
It is trivial to check that $h_{0,x}(\psi_{0,\rep_0(x)}^0)$ satisfies the properties~\ref{i: adapted 0 center} and~\ref{i: adapted strongly} of \Cref{d. adapted 0}, whereas its property~\ref{i: adapted 0 children} follows from \Cref{p: defn dfs}.
\end{proof}

\begin{rem}
	Note that the domain of $h_{0,x}$ is $\overline C_{0,-1}(x)$, so we are actually considering its restriction to $C_{0,-1}(x)$ in \Cref{p: phi0x0}. We will continue to make this assumption implicitly for the maps $h_{0,y}$.
\end{rem}

\begin{prop}\label{p: phi_{0,x}^i}
There is a family of colorings, $\psi_{0,x}^i\colon \czeroone(x)\to [\Delta]$, for $x\in X_0$ and $i\in H_{0,x}$,  satisfying the following properties: 
\begin{enumerate}[(i)]

\item \label{i: phi_0,x^0} The coloring $\psi_{0,x}^0$ is strongly adapted.

\item \label{i: phi_0 x^i h} We have $\psi_{0,x}^i=h_{0,x}(\psi_{0,\rep_0(x)}^i)$.

\item \label{i: phi_0 x^i adapted} For $i\in H_{0,x}$, the coloring $\psi_{0,x}^i$ is adapted.

\item \label{i: phi0xi i j} For $x\in X_0$ and $i,j\in H_{0,x}$,  let $A=C_{0,-1}(x)$ {\rm(}respectively, $A=D_{-1}(x,r_n^\pm)$\/{\rm)}, and let $f \colon (A,x,\psi_{0,x}^i)\to (A,x,\psi_{0,x}^j)$ be a  color-preserving  restriction of a $0$-equivalence   {\rm(}respectively, $0$-weak equivalence\/{\rm)}.  
Then $f$ is the identity map on $A$, and $i=j$.

\end{enumerate}
\end{prop}

\begin{proof}
First, for $i=0$, we take the strongly adapted colorings $\psi_{0,x}^0$ constructed in \Cref{p: phi0x0}. So~\ref{i: phi_0,x^0} is satisfied.

For every $x\in X_0^{\pm,\Phi}$, choose a maximal 3-separated subset $N_{0,x}$ of $C_{-1}(x, 10  , r_0^\pm )$, together with an enumeration of its power set, 
\[
\mathcal{P}( N_{0,x} )= \{\, \mathcal{N}_{0,x}^0 = \emptyset, \ \mathcal{N}_{0,x}^1, \ \ldots  \,\}\;.
\] 
We have $|D_{-1}(x,10)|\leq \Delta^{11}$  by \Cref{c: |D(x r)| le ...}. 
Thus $|C_{-1}(x,10,r_0^\pm)|\geq |D_{-1}(x,r_0^\pm)| - \Delta^{11}$ (recall that $r_0^\pm>2^{11}$).
By \Cref{l: maximal}, $N_{0,x}$ is $2$-relatively dense in $C_{-1}(x,10,r_0^\pm)$. 
So 
\begin{equation}\label{card n0x}
|N_{0,x}| \geq \big\lfloor\big(|D_{-1}(x,r_n^\pm)|-\Delta^{11}\big)/\Delta^3\big\rfloor \ge \big\lfloor\big(|D_{-1}(x,r_n^\pm)|-\Delta^{11}-1\big)/\Delta^3\big\rfloor
\end{equation}
by \Cref{l: cardinality separated}.
Therefore 
\[
|\mathcal{P}( N_{0,x} )| \ge \exp_2 \big(\big\lfloor\big(|D_{-1}(x,r_n^\pm)|-\Delta^{11}-1\big)/\Delta^3\big\rfloor\big)=\eta_0(|D_{-1}(x,r_n^\pm)|)\;. 
\]
Thus an injective map $H_{0,x}\to\mathcal P(N_{0,x})$  is well defined by $i\mapsto\mathcal N_{0,x}^i$.

If $x\notin X_0^\Phi$, let $N_{0,x}=h_{0,x}(N_{0,\rep_0(x)})$ and  $\mathcal{N}^i_{0,x}=h_{0,x}(\mathcal{N}^i_{0,\rep_0(x)})$, so that $N_{0,x}$ satisfies~\eqref{card n0x}.
Then define 
\[
\psi^i_{0,x}(u)= 
\begin{cases}
\psi^0_{0,x}(u) &\text{if $u\notin \mathcal{N}_{0,x}^i$}\\ 
0  &\text{if $u\in \mathcal{N}_{0,x}^i$}\;.
\end{cases}
\]
Note that this definition agrees with the previous one in the case $i=0$. 
Property~\ref{i: phi_0 x^i h} follows immediately from \Cref{p: phi0x0} and the fact that  $\mathcal{N}^i_{0,x}=h_{0,x}(\mathcal{N}^i_{0,\rep_0(x)})$. 

To prove~\ref{i: phi_0 x^i adapted}, note that $\psi_{0,x}^i=\psi_{0,x}^0$ on $D_{-1}(x,10)$ by construction.
 So \Cref{d. adapted 0}~\ref{i: adapted 0 center} is trivially satisfied by $\psi_{0,x}^i$.
 For every $u\in \czeroone(x)$, we have $\Ch_{0,x}(u)\subset D_{-1}(u,1)$, which yields $d(v,w)\leq 2$ for all $v,w\in \Ch_{0,x}(u)$.
Hence $|N_{0,x}\cap\Ch_{0,x}(u)|\le1$ because $N_{0,x}$ is $3$-separated, and therefore $|\mathcal{N}_{0,x}^i\cap \Ch_{0,x}(u)|\le1$.
 The coloring $\phi_{0,x}^0$ assigns different colors to all  points in $\Ch_{0,x}(u)$ (\Cref{d. adapted 0}~\ref{i: adapted 0 children}.
If $u\in D_{-1}(x,9)$, then $\Ch_{0,x}(u)\subset D_{-1}(x,10)$, and therefore $\psi_{0,x}^i$ also assigns different colors to all points in $\Ch_{0,x}(u)$ since $\psi_{0,x}^i=\psi_{0,x}^0$ on $D_{-1}(x,10)$.
If $u\in C_{0,-1}(x)\setminus D_{-1}(x,9)$, then $\psi_{0,x}^0$ assigns different colors to all  points in $\Ch_{0,x}(u)$, all of them different from $0$,  and it follows from \Cref{d. adapted 0} and \Cref{p: phi_{0,x}^i}~\ref{i: phi_0,x^0} that $\psi_{0,x}^i$  assigns different colors to those points too. 
Thus \Cref{d. adapted 0}~\ref{i: adapted 0 children} is satisfied by $\psi_{0,x}^i$, and the coloring $\psi_{0,x}^i$ is adapted.

To prove~\ref{i: phi0xi i j}, suppose first that $A=C_{0,-1}(x)$ and $f$ is a $0$-equivalence. For all $u\in C_{0,-1}(x)$, we are going to  show that $f$ is the identity map on $\Ch_{0,x}(u)$, and that $\mathcal{N}^i_{0,x}\cap \Ch_{0,x}(u) = \mathcal{N}^j_{0,x}\cap\Ch_{0,x}(u)$, using induction on $u$ with $\trianglelefteq_{0,x}$. 
This will complete the proof because it follows that $f$ is the identity map and $\mathcal{N}^i_{0,x} = \mathcal{N}^j_{0,x}$, yielding $i=j$.

First, we have $f(x)=x$ by \Cref{d. adapted 0}~\ref{i: adapted 0 center}, since $x$ is the unique point having the correct coloring pattern  on some geodesic segment of the form $\tau=(x=\tau_0,\dots, \tau_5)$. 
Also, we have
\[
\mathcal{N}^i_{0,x}\cap \Ch_{0,x}(x) = \mathcal{N}^j_{0,x}\cap\Ch_{0,x}(x)=\emptyset
\]
since $N_{0,x}\cap D(x, 10)=\emptyset$.

Suppose now that, for some $u\in C_{0,-1}(x)$ with $d_{-1}(u,x)>0$, $f$ is the identity map on $\Ch_{0,x}(v)$ and $\mathcal{N}^i_{0,x}\cap \Ch_{0,x}(v) = \mathcal{N}^j_{0,x}\cap\Ch_{0,x}(v)$ for all $v\vartriangleleft_{0,x}u$.
In particular,  $f$ is the identity map on $\Ch_{0,x}(\Pa_{0,x}(u))$, and therefore $f(u)=u$.
Furthermore this implies $f(\Ch_{0,x}(u))=\Ch_{0,x}(u)$ by~\eqref{defn ch}.
By definition, for $l=i,j$, we have $\psi^l_{0,x}=\psi^0_{0,x}$ on $\Ch_{0,x}(u) \setminus N_{0,x}$, and $\psi^l_{0,x}(u)=0$ if $u\in \mathcal{N}^l_{0,x}$.
Recall that  $N_{0,x}\cap \Ch_{0,x}(u)$ has at most one point, which is denoted by $w$ if it exists.
In this case,  by~\ref{i: phi_0 x^i adapted} and \Cref{d. adapted 0}~\ref{i: adapted 0 children}, $\psi_{0,x}^0$ is injective on $\Ch_{0,x}(u)\setminus \{w\}$.
Thus $\psi_{0,x}^i$ and $\psi_{0,x}^j$ agree and are injective on $\Ch_{0,x}(u)\setminus \{w\}$, and therefore $f$ is the identity on $\Ch_{0,x}(u)\setminus \{w\}$.
But this yields $f(w)=w$. Thus, in any case, $f$ is the identity map on $\Ch_{0,x}(u)$ and  $\Ch_{0,x}(u) \cap \mathcal{N}^i_{0,x}=\Ch_{0,x}(u) \cap \mathcal{N}^j_{0,x}$.

The proof of~\ref{i: phi0xi i j} when $A=D_{-1}(x,r_0^\pm)$ and $f$ is a $0$-weak equivalence is similar. 
\end{proof}

\begin{cor}\label{c. id zero}
Let $x,y\in X_0$, $i\in H_{0,x}$ and $j\in H_{0,y}$,  let $A=C_{0,-1}(x)$ {\rm(}respectively, $A=D_{-1}(x,r_0^\pm)${\rm)}, and let $f \colon (A,x,\psi_{0,x}^i)\to (f(A),y,\psi_{0,y}^j)$  be a color-preserving restriction of a $0$-equivalence {\rm(}respectively, $0$-weak equivalence\/{\rm)} $x\to y$.
Then  $i=j$ and $f=h_{0,y} h_{0,x}^{-1}$ on $A$.
\end{cor}

\begin{proof}
Suppose that $A=C_{0,-1}(x)$.
Since there is a $0$-equivalence $x\to y$, we have  $\Phi_0(x)=\Phi_0(y)$ and $\rep_0(x)=\rep_0(y)=:z$.
So $h_{0,x}^*\psi_{0,x}^l= \psi_{0,z}^l$ for $l=i,j$ by \Cref{p: phi_{0,x}^i}~\ref{i: phi_0 x^i h}. Thus 
\[
h_{0,y}^{-1} f h_{0,x}\colon (C_{0,-1}(z),z,\psi_{0,z}^i) \to (C_{0,-1}(z),z,\psi_{0,z}^j)
\]
is a color-preserving $0$-equivalence.
Then the result follows from \Cref{p: phi_{0,x}^i}~\ref{i: phi0xi i j}.

The case where $A=D_{-1}(x,r_0^\pm)$ is similar.
\end{proof}

\subsection{Adapted colorings for $n>0$}

\begin{defn}\label{d. phi_{n,x}^0 i}
Let $x\in \xn$. A coloring $\psi\colon \cnyn(x)\to \mathcal{I}_{n-1}$  is said to be \emph{adapted} if the following conditions are satisfied:
\begin{enumerate}[(i)]

\item \label{i: phi_{n,x}^0 i zero} We have $\psi^{-1}(0)= \xxyn\cap \cnyn(x)$.

\item \label{i: phi_{n,x}^0 i one} We have
\[
\psi^{-1}(1)=
\begin{cases}
\{x\} & \text{if $x\in \xn^- \setminus \xxyn$} \\ 
\emptyset & \text{otherwise}\;.
\end{cases}
\]

\item \label{i: phi_{n,x}^0 i two} We have
\[
\psi^{-1}(2)=
\begin{cases}
\{x\}  & \text{if $x\in \xn^+ \setminus \xxyn$} \\ 
\emptyset & \text{otherwise}\;.
\end{cases}
\]

\item\label{i: phi_{n,x}^0 i three xxn} If $x\in \xxyn\cap X_{n}^+$, then $\psi^{-1}(3)=\{y\}$ for some $y\in S_{n-1}(x,1)$, otherwise $\psi^{-1}(3)=\emptyset$.

\item\label{i: phi_{n,x}^0 i four xxn} If $x\in \xxyn\cap X_{n}^-$, then $\psi^{-1}(4)=\{y\}$ for some $y\in S_{n-1}(x,1)$, otherwise $\psi^{-1}(4)=\emptyset$.\setcounter{nameOfYourChoice2}{\value{enumi}}

\setcounter{nameOfYourChoice2}{\value{enumi}}
\end{enumerate}
The coloring $\psi$ is \emph{strongly adapted} if it is adapted and, additionally, it satisfies the following condition:
\begin{enumerate}[(i)]
\setcounter{enumi}{\value{nameOfYourChoice2}}
\item\label{i: phi_{n,x}^0 i four} $\psi^{-1}(5)=\emptyset$.
\end{enumerate}
\end{defn}

Recall that the sets $C_{n,n-1}(x)$, for $x\in X_n$, form a partition of $X_{n-1}$ by definition.

\begin{lem}\label{l: condition center n}
Consider a family of adapted colorings, $\psi_x\colon C_{n,n-1}(x)\to \mathcal{I}_{n-1}$, for $ x\in X_n$, whose combination is denoted by $\psi$.
For every $u\in X_{n-1}$, we have $u\in X_n$ if and only if, either $\psi(u)\in\{1,2\}$, or $\psi(u)=0$ and there is some $v\in S_{n-1}(u,1)$ such that $\psi(v)\in\{3,4\}$.
\end{lem}

By \Cref{p: xn}~\ref{i: xn dn dn-1}, and \Cref{l: overlinecnx rnpm,l: cn-1 gamma}, we have $d_{-1}(u,v)\leq 2L_{n-1}R_n^+$ for any $u,v\in \cnyn(x)$.
On the other hand, if $u,v\in \xxyn$, then $d_{-1}(u,v)\geq \mathfrak{s}_{n-1}$ by \Cref{p: xxn}~\ref{i: xxn subset}.
Since $\mathfrak{s}_{n-1}>3L_{n-1}\Gamma_n^+\ge3L_{n-1}R_n^+$ by~\eqref{2rl + sn < sl},~\eqref{no bm multi} and~\eqref{gamma r}, it follows that
\begin{equation}\label{cnyn xxyn}
|\cnyn(x)\cap \xxyn|\leq 1\;.
\end{equation} 

\begin{lem}\label{l: exists adapted n}
For every $x\in \xn$, there is a strongly adapted coloring $\psi_x\colon C_{n,n-1}(x)\to \mathcal{I}_{n-1}$.
\end{lem}

\begin{proof}
First, note that $[7]\subset I_{n-1,u}$ for all $u\in C_{n,n-1}(x) $ by~\eqref{def inx}.
Define $\psi_x(u)=0$ for every $u\in \cnyn(x)\cap \xxyn$.
In the case where $x\in  \xxyn$, choose some $y\in S_{n-1}(x,1)$ and define 
\[
\psi_x(y)=
\begin{cases}
3 &\text{if $x\in \xxn^-$} \\ 
4 &\text{if $x\notin \xxn^+$}\;.
\end{cases}
\] 
If $x\notin \xxyn$,  set 
\[
\psi_x(x)=
\begin{cases}
1 \quad &\text{if $x\in \xxn^-$} \\ 
2 &\text{if $x\notin \xxn^+$}\;.
\end{cases}
\]

Let $A$ be the set of points in $\cnyn(x)$ that have been already colored.
For $u\in C_{n,n-1}(x)\setminus A$, let $\phi_x(u)$ be any color in $I_{n-1,u}\setminus [6]$.
\end{proof}

\begin{prop}\label{p: phinx0}
There is a family of strongly adapted colorings, $\psi_{n,x}^0\colon \cnyn(x)\to \mathcal{I}_{n-1}$, for $x\in \xn$, satisfying $\psi_{n,x}^0=h_{n,x}(\psi_{n,\rep_n(x)}^0)$.
\end{prop}

\begin{proof}
This follows from \Cref{l: exists adapted n} like \Cref{p: phi0x0}.
\end{proof}

\begin{prop}\label{p: phi_n x^i}
There is  a family of colorings, $\psi_{n,x}^i\colon \cnyn(x)\to \mathcal{I}_{n-1}$, for $x\in \xn$ and $i\in H_{n,x}$, satisfying the following properties:
\begin{enumerate}[(i)]

\item \label{i: phi_n x^0} The coloring $\psi_{n,x}^0$ is strongly adapted.

\item\label{i: phinx0 h} We have $\psi_{n,x}^i=h_{n,x}(\psi_{n,\rep_n(x)}^i)$. 

\item \label{i: phi_{n,x}^i adapted} Each coloring $\psi_{n,x}^i$ is adapted.

\item \label{i: phi_n x^i different} There are sets $\mathcal{N}_{n,x}^i\subset C_{n-1}(x,10,r_n^\pm-1)$, for $x\in X_n$ and $i\in H_{n,x}$,
satisfying:
\begin{enumerate}[(a)]
\item \label{i: phi_n x^i hat h} $\mathcal{N}_{n,x}^i=\hat h_{n,x}(\mathcal{N}_{n,\widehat \rep_n (x)}^i)$;
\item $(\psi_{n-1,x}^i)^{-1}(4)=\mathcal{N}_{n,x}^i$; and
\item $\mathcal{N}_{n,x}^i \neq \mathcal{N}_{n,x}^j$ if $i\neq j$.
\end{enumerate}
\end{enumerate}
\end{prop}

\begin{proof}
First, for $i=0$, we take the strongly adapted colorings $\phi_{0,x}^0$ constructed in \Cref{p: phi0x0}, so that~\ref{i: phi_n x^0} is satisfied.

For every $x\in \xn^{\pm,\widehat \Phi}$, let $N_{n,x}$ be a maximal subset of $C_{n-1}(x,10,r_n^\pm)\setminus \mathfrak{X}_{n-1}$ that is $r_{n-1}^2s_{n-1}$-separated with respect to $d_{n-2}$. Choose an enumeration of the  power set $\mathcal{P}(N_{n,x})$,
\[
\mathcal{P}( N_{n,x})= \{\, \mathcal{N}_{n,x}^0:=\emptyset ,\, \mathcal{N}_{n,x}^1, \, \ldots  \,\}\;.
\] 
We have $|D_{n-1}(x,10)|\leq (\deg X_{n-1})^{11}$  and $|C_{n,n-1}(x)\cap \xxyn|\leq 1$ by \Cref{c: |D(x r)| le ...} and~\eqref{cnyn xxyn}.
Therefore 
\[
|C_{n-1}(x,10,r_n^\pm)\setminus \mathfrak{X}_{n-1}| \geq  |D_{n-1}(x,r_n^\pm)|-(\deg X_{n-1})^{11}-1\;.
\]
By \Cref{l: maximal}, $N_{n,x}$ is $(r_{n-1}^2s_{n-1}-1)$-relatively dense in $|C_{n-1}(x,10,r_n^\pm)|$ with respect to $d_{n-2}$, so 
\begin{equation}\label{cardinality nnx}
|N_{n,x}| \geq \Big\lfloor\big(|D_{n-1}(x,r_n^\pm)|-(\deg X_{n-1})^{11} -1\big)/(\deg X_{n-2})^{r_{n-1}^2s_{n-1}}\Big\rfloor
\end{equation}
by \Cref{l: cardinality separated}. Therefore, by~\eqref{etan defn},
 \[
 |\mathcal P(N_{n,x})| 
\geq \exp_2 \Big(\Big\lfloor\big(|D_{n-1}(x,r_n^\pm)|-(\deg X_{n-1})^{11} -1\big)/(\deg X_{n-2})^{r_{n-1}^2s_{n-1}}\Big\rfloor\Big)
=\eta_n(|D_{n-1}(x,r_n^\pm)|)\;.
 \]
Thus an injective map $H_{n,x}\to\mathcal P(N_{n,x})$  is well defined by $i\mapsto\mathcal N_{n,x}^i$.
  
If $x\notin X_0^{\widehat \Phi}$, let $N_{n,x}=\hat h_{n,x}(N_{n,\rep_n(x)})$ and $\mathcal{N}^i_{n,x}=\hat h_{n,x}(\mathcal{N}^i_{n,\rep_n(x)})$, so that $N_{n,x}$ satisfies~\eqref{cardinality nnx}.
Then define 
\[
\psi^i_{n,x}(u)= 
\begin{cases}
\psi^0_{n,x}(u) &\text{if $u\notin \mathcal{N}_{n,x}^i$} \\ 
4 &\text{if $u\in \mathcal{N}_{n,x}^i$}\;.
\end{cases}
\]
With this definition,~\ref{i: phi_n x^0} is obvious because $\mathcal{N}_{n,x}^0=\emptyset$.
Property~\ref{i: phinx0 h} follows immediately from \Cref{p: phinx0} and the fact that $\mathcal{N}^i_{0,x}=h_{0,x}(\mathcal{N}^i_{0,\rep_0(x)})$ if $x\notin X_0^\Phi$. 
Finally,~\ref{i: phi_n x^i different} follows since $\mathcal{N}^i_{n,x}\neq \mathcal{N}^j_{n,x}$ for $i\neq j$.
\end{proof}

\begin{rem}\label{r. notation}
In \Cref{s. chin}, it was said that  $I_{n,x}^2$ is considered as an initial segment of $H_{n,x}$ for every $x\in X_n$.
Let $\iota_{n,x}$ denote the inclusion $I_{n,x}^2\hookrightarrow H_{n,x}$.
From now on, the notation $\psi_{n,x}^{i,j}$ will refer to the coloring $\psi_{n,x}^{\iota_{n,x}(i,j)}$.
\end{rem}

\subsection{Colorings $\phi_n^N$}

In this subsection we  define the colorings $\phi_n^N$, which will induce the colorings $\phi^N$ in the statement of \Cref{t: finitary}.
First we define the notion of a rigid coloring, which is 
 obtained by combining different colorings $\psi_{0,x}^i$ over clusters $C_{0,-1}(x)$.

\begin{defn}
Let $n \in \N$ and $x\in X_n$.
A coloring $\phi\colon C_{n,-1}(x)\to [\Delta]$ is called \emph{rigid} if, for all $u\in C_{n,0}(x)$, there is some $i\in H_{n,x}$ such that the restriction of $\phi $ to $C_{0,-1}(u)$ equals $\psi_{0,x}^i$.
\end{defn}

\begin{lem}\label{l: distances +}
For all $x_1,x_2\in X_n^+$, if $d_n(x_1,x_2)\leq 2$, then $d_{n-1}(x_1,x_2)< r_n^+s_n$.
\end{lem}

\begin{proof}
By the definition of $E_n$, there are some  $x_3\in X_n$,  $u_1\in \overline C_{n,n-1}(x_1)$, $u_2\in \overline C_{n,n-1}(x_2)$ and $u_3,u'_3\in \overline C_{n,n-1}(x_3)$ such that $u_1 E_{n-1} u_3$ and $u'_3 E_{n-1} u_2$.
By \Cref{l: overlinecnx rnpm}, the triangle inequality,~\eqref{defn functions n} and~\eqref{no bm multi}, we get 
\[
d_{n-1}(x_1,x_2)\leq 4R_n^+ + 2=4(r_n(2s_n+3))+2\leq 20r_ns_n < r_ns_n^2\;,
\]
since $s_n>20$ by~\eqref{defn s0} and~\eqref{defn sn}.
\end{proof}

\begin{lem}\label{l: distances -}
For all $x_1,x_2,x_3\in X_n^-$, if $x_1E_nx_2E_nx_3$, then $d_{n-1}(x_1,x_3)< r_n^- s_n$.
\end{lem}
\begin{proof}
By the definition of $E_n$, there are some  $u_1\in \overline C_{n,n-1}(x_1)$, $u_2,u'_2\in \overline C_{n,n-1}(x_2)$ and $u_3\in \overline C_{n,n-1}(x_3)$ such that $u_1 E_{n-1} u_2$ and $u_2' E_{n-1} u_3$.
By \Cref{l: overlinecnx rnpm}, the triangle inequality,~\eqref{defn functions n}, and~\eqref{no bm multi}, we get 
\[
d_{n-1}(x_1,x_2)\leq 4R_n^- + 2=4(4r_n+2)+2\leq 26r_n<r_ns_n\;,
\]
since $s_n>26$ by~\eqref{defn s0} and~\eqref{defn sn}.
\end{proof}

\begin{prop}\label{p: equivalence identity}
For $n\in \N$ and $x\in X_n^\pm$, let
\[
A=C_{n,-1}(x)\quad \Big(\text{respectively,}\ A=\bigcup_{a\in D_{n-1}(x,r_n^\pm-1)} C_{n-1,-1}(a)\Big)\;, 
\]
let $\zeta:A\to [\Delta]$ be a rigid coloring {\rm(}respectively, the restriction of a rigid coloring{\rm)}, and let $f\colon x\to x$ be an $n$-equivalence {\rm(}respectively, an $n$-weak equivalence\/{\rm)} preserving $\zeta$. Then $f$ is the identity map on $A$.
\end{prop}

\begin{proof}
We proceed by induction on $n\in \N$.
If $n=0$, then the result follows from \Cref{p: phi_{0,x}^i}~\ref{i: phi0xi i j}.
Therefore suppose that $n>0$ and the result is true for $0\leq m <n$.
By hypothesis, $f$ is an $n$-(weak) equivalence and $f(x)=x$.
Thus $f(C_{n-1,n-2}(x))=C_{n-1,n-2}(x)$ and $f:x\to x$ is an $(n-1)$-equivalence by Definitions~\ref{d. equivalence}~\ref{i: equivalence cmm-1} and~\ref{d. weak equivalence}~\ref{i: weak equivalence cmm-1}. 
Hence $f$ is the identity on $C_{n-1,n-2}(x)$ by the induction hypothesis.

Let us prove that $f$ is the identity on $C_{n-1,n-2}(u)$ by induction on $u\in A$, using $\trianglelefteq_{n,x}$.
The case $u=x$ was proved in the previous paragraph. Thus let $u\neq x$ and suppose that the result has been proved for $v\triangleleft_{n,x}u$.
By the induction hypothesis and \Cref{d. parent}, we have $f(\Pa_{n,x}(u))=\Pa_{n,x}(u)$ and $u\,E_{n-1}\,\Pa_{n,x}(u)$. 
Therefore $f(u)\,E_{n-1}\,f(\Pa_{n,x}(u))$  by Definitions~\ref{d. equivalence}~\ref{i: equivalence cnm} and~\ref{d. weak equivalence}~\ref{i: weak equivalence cnm}, and we get $f(u)\,E_{n-1}\,\Pa_{n,x}(u)$.
We consider the following cases.

If $u,f(u)\in X_{n-1}^+$, then $d_{n-2}(u,f(u))<r_{n-1}^+s_{n-1}$ by \Cref{l: distances +}. 

If $u,\Pa_{n,x}(u)\in X_{n-1}^-$, then $f(u)\in X_{n-1}^-$ by Definitions~\ref{d. equivalence}~\ref{i: equivalence cnm} and~\ref{d. weak equivalence}~\ref{i: weak equivalence cnm}, and we obtain $d_{n-2}(u,f(u))<r_{n-1}^-s_{n-1}$  by \Cref{l: distances -}.
By Definitions~\ref{d. equivalence}~\ref{i: equivalence cnm} and~\ref{d. weak equivalence}~\ref{i: weak equivalence cnm}, we have $\chi_{n-1}(u)=\chi_{n-1}(f(u))$.
Thus \Cref{p: chin}~\ref{i: chin dn-1} yields $f(u)=u$ in these two cases.

Finally, suppose that $u,f(u)\in X_{n-1}^-$ and $\Pa_{n,x}(u)\in X_{n-1}^+$.
By the definition of $E_{n-1}$, there is some $u'\in X_{n-1}^+\cap D_{n-1}(\Pa_{n,x}(u),1)$ such that there are $v\in \overline C_{n-1,n-2}(u)$ and $v'\in C_{n-1,n-2}(u')$ with $vE_{n-2}v'$.
Note that this implies $d_{n-1}(x,u')\leq d_{n-1}(x,u)$. If $f$ is an $n$-equivalence, then this implies $u'\in \mathcal{C}_{n,n-1}(x)$, whereas if $f$ is an $n$-weak equivalence, we obtain $u'\in D_{n-1}(x,r_n^\pm-1)$.
In any case, using \Cref{d. equivalence,d. weak equivalence} we get that $f$ restricts to an $(n-1)$-equivalence $u'\to f(u')$.
Since $u'\,E_{n-1}\,\Pa_{n,x}(u)$  and $f(\Pa_{n,x}(u))=\Pa_{n,x}(u)$, we obtain $d_{n-2}(u,f(u))<r_{n-1}^+s_{n-1}$, and the same argument of the previous paragraph gives us $f(u')=u'$. 
Then the induction hypothesis (on $n$) yields $f(v')=v'$.
Therefore $d_{n-2}(v,f(v))\leq 2$, and we obtain $d_{n-2}(u,f(u))\leq 2R_n^-+2$.
Then $f(u)=u$ as before, and we get that $f$ is the identity on $C_{n-1,-1}(u)$ by the induction hypothesis.
\end{proof}

\begin{cor}\label{c. equivalence}
For $n\in \N$ and $x,y\in X_n^\pm$, let
	\[
	A=C_{n,-1}(x)\quad \Big(\text{respectively,}\ A=\bigcup_{a\in D_{n-1}(x,r_n^\pm-1)} C_{n-1,-1}(a)\Big)\;, 
	\]
	let $f\colon x\to x$ be an $n$-equivalence {\rm(}respectively, an $n$-weak equivalence\/{\rm)}, and let $\zeta:A\to [\Delta]$ and $\hat \zeta\colon f(A)\to [\Delta]$  be rigid colorings {\rm(}respectively, restrictions of  rigid colorings{\rm)}. If  $f^*\hat\zeta=\zeta$, then $f=h_{n,y} h_{n,x}^{-1}$  {\rm(}respectively, $f=\hat h_{n,y} \hat h_{n,x}^{-1}$\/{\rm)} on $A$.
\end{cor}

\begin{defn}\label{d. psin}
For $N\in \N$, let $\phi^N_n\colon X_n\to \mathcal{I}_n^2$ be defined by reverse induction on $n=-1,\ldots, N$ as follows: 
\begin{itemize}

\item For $n=N$, let $\phi_{N}^N=(\chi_N,0)$.

\item For $0\leq n<N$, define $\phi_n^N$ so that, for every $x\in X_{n+1}$, 
\begin{equation}\label{psin} 
\phi^N_{n}|_{C_{n+1,n}(x)}=\Big(\psi_{n+1,x}^{\phi^N_{n+1}(x)}, \chi_{n}(x)\Big)\;.
\end{equation}

\item Finally, define $\phi_{-1}^N$ so that, for every $x\in X_0$,
\begin{equation}\label{psicero} 
\phi^N_{-1}|_{C_{0,-1}(x)}=\psi_{0,x}^{\phi^N_{0}(x)}\;.
\end{equation}
\end{itemize}
\end{defn}

\begin{rem}\label{r. psi adapted}
It follows from \Cref{p: chin}~\ref{i: chin dn-1}  that $\phi_n^N(x)\neq \phi_n^N(y)$  for $x,y\in X_n^\pm$ if $0<d_{n-1}(x,y)<r_n^\pm s_n$.
\end{rem}

\begin{rem}\label{r. psi xxn}
 By Definitions~\ref{p: chin}~\ref{i: chin zero} and~\ref{d. phi_{n,x}^0 i}~\ref{i: phi_{n,x}^0 i zero}, for all $0\leq m \leq N$ and $x\in X_m$, the value $\phi_m^N(x)$ determines whether $x\in\xxm$.
\end{rem}

We now prove the crucial lemma from which we will derive \Cref{t: finitary}. In order to do this, we will show  that a pointed, colored, ``local"  graph isomorphism gets more rigid on smaller domains,  meaning that it preserves more of the structure that we have defined in the course of this section.

Let $W_0=10$ and $W_i=2$ for $i>0$, and let $\Upsilon_n$ be recursively defined by
\begin{equation}\label{defn upsilon}
\Upsilon_{-1}=0 \;, \quad \Upsilon_n=\Upsilon_{n-1} +L_{n-1}(W_n+3R_n^+ + 1)+ 2\Gamma_n^+ + nL_n\;.
\end{equation}

\begin{lem}\label{l: h preserves}
Fix $0\leq n \leq N$ and $R>\Upsilon_n$. 
Let $A\subset X$ and $x\in \aset$ be such that $D_{-1}(x,R)\subset \aset$, and let $f\colon (\aset,x,\phi_{-1}^N)\to (f(\aset),f(x),\phi_{-1}^N)$ be a pointed colored graph isomorphism with respect to the restriction of $E_{-1}$. 
Then the following properties hold for  $0\leq m\leq n$ and $0\leq l \leq n+1$:
\begin{enumerate}[(a)]

\item \label{i: h preserves induction} The restriction 
\[
f\colon \left(X_{l-1}\cap D_{-1}(x,R-\Upsilon_{l-1}),x,\phi_{l-1}^N\right)\to \left(X_{l-1}\cap D_{-1}(f(x),R-\Upsilon_{l-1}),f(x), \phi_{l-1}^N\right)
\]
is a  pointed colored graph isomorphism with respect to $E_{l-1}$.

\item \label{i: h preserves xn} For any $z\in X_{m-1}\cap D_{-1}(x,R-\Upsilon_{m-1}-L_{m-1}W_m)$, we have $z\in \xm^\pm$ if and only if $f(z)\in \xm^\pm$.

\item \label{i: h preserves weak equivalence}  For all $z\in \xm\cap D_{-1}(x,R-\Upsilon_{m-1} -L_{m-1}(W_m+r_m^+))$,  the restriction of $f$ is an $m$-weak equivalence.

\item \label{i: h preserves psi} For any $z\in \xm\cap D_{-1}(x,R-\Upsilon_{m-1} -L_{m-1}(W_m+r_m^+))$,  we have $\phi_{m}^N(z)=\phi_{m}^N(f(z))$.

\item \label{i: h preserves xxm} For any $z\in \xm\cap D_{-1}(x,R-\Upsilon_{m-1} -L_{m-1}(W_m+r_m^+ + 1))$,  we have $z\in \xxm$ if and only if $f(z)\in \xxm$.

\item \label{i: h preserves cm<} For all $ z\in X_{m-1}\cap D_{-1}(x,R-\Upsilon_{m-1} - L_{m-1}(W_m+2R_m^+))$, we have  $z\in Z_{m-1}^\pm$ if and only if $f(z)\in Z_{m-1}^\pm$.

\item \label{i: h preserves overline} For any $z\in \xm\cap  D_{-1}(x,R-\Upsilon_{m-1} - L_{m-1}(W_m+3R_m^+))$,  we have  $f(\olcmym(z))=\overline{C}_{m,m-1}(f(z))$.

\item \label{i: h preserves cluster}  For any $z\in \xm\cap  D_{-1}(x,R-\Upsilon_{m-1} - L_{m-1}(W_m+3R_m^+)-L_m)$,  we have  $f(C_{m,m-1}(z))=C_{m,m-1}(f(z))$.

\item \label{i: h preserves em}  For all $z,z'\in \xm\cap D_{-1}(x,R-\Upsilon_{m-1} -L_{m-1}(W_m+3R_m^+ +1)-\Gamma_m^+)$,  we have $zE_{m}z'$ if and only if $f(z)E_{m}f(z')$.

\item \label{i: h preserves equivalence}  For all $z\in \xm\cap D_{-1}(x,R-\Upsilon_{m-1} -L_{m-1}(W_m+3R_m^+ +1)-2\Gamma_m^+- mL_m)$,  the restriction of $f$ to $\mathcal{C}_{m,-1}$ is an $m$-equivalence $z\to f(z)$.

\end{enumerate}
\end{lem}

\begin{proof}
We proceed by induction on $m$ and $l$. 
For $l=0$,~\ref{i: h preserves induction} is true by hypothesis. 
When $l>0$,~\ref{i: h preserves induction} follows from~\eqref{defn upsilon} and the induction hypothesis  for $m=l-1$ with~\ref{i: h preserves psi} and~\ref{i: h preserves em}.
For $m=0,\ldots,n$, we are going to derive~\ref{i: h preserves xn}--\ref{i: h preserves equivalence} from~\ref{i: h preserves induction}, completing the proof of the lemma.

Let us prove~\ref{i: h preserves xn}. 
The coloring $\phi_{m-1}^N$ is adapted by \Cref{r. psi adapted}.
For every $z\in X_{m-1}$, we have $z\in \xm^\pm$ if and only if the colored set $(D_{m-1}(z,W_m/2),\phi_{m-1}^N)$ has one of the patterns described in \Cref{d. adapted 0}~\ref{i: adapted 0 center} and \Cref{l: condition center n}. 
By \Cref{p: xn}~\ref{i: xn dn dn-1} and the triangle inequality, we get 
\[
D_{m-1}(z, W_m)\subset D_{-1}(z,L_{m-1}W_m )\subset D_{-1}(x, R-\Upsilon_m).
\] 
Therefore the restriction  $f\colon D_{m-1}(z, W_m/2) \to D_{m-1}(f(z), W_m/2)$ is an isometry by \Cref{c: preserves metric}. 
The induction hypothesis with~\ref{i: h preserves induction} implies that  the set $D_{m-1}(z,W_m/2)$ has one of the patterns of \Cref{d. adapted 0}~\ref{i: adapted 0 center} and \Cref{l: condition center n} if and only if $D_{m-1}(f(z),W_m/2)$ does.
Then~\ref{i: h preserves xn} follows from~\ref{i: h preserves induction}.

To prove~\ref{i: h preserves weak equivalence}, let $z\in X_m^\pm$.
If $m=0$,~\ref{i: h preserves weak equivalence} is obvious. Thus suppose $m>0$.
We have $f(z)\in X_m^\pm$ by~\ref{i: h preserves xn}.
By \Cref{p: xn}~\ref{i: xn dn dn-1}, 
\[
D_{m-1}(z,r_m^\pm)\subset D_{-1}(z, L_{m-1}r_m^+ )   \subset D_{-1}(x,R-\Upsilon_{m-1}-L_{m-1}W_m)\;.
\]
Now, in \Cref{d. weak equivalence}, properties~\ref{i: weak equivalence ball} and~\ref{i: weak equivalence cnm} follow from~\ref{i: h preserves induction},
the property~\ref{i: weak equivalence xxm} holds by the induction hypothesis with~\ref{i: h preserves xxm}, and the property~\ref{i: weak equivalence cmm-1} follows from \Cref{c. equivalence} and the induction hypothesis with~\ref{i: h preserves equivalence}.

Let us prove~\ref{i: h preserves psi}.
By \Cref{d. psin}, the restriction of $\phi_{m-1}^N$ to $C_{m,m-1}(z)$ equals $(\psi_{m-1,z}^i,\chi_{m-1}(z))$ for some $i\in H_{m,z}$.
Then $\phi_m^N(z)=\phi_m^N(f(z))$ if and only if the restrictions of $\phi_{m-1}^N$ to $C_{m,m-1}(z)$ and $C_{m,m-1}(f(z))$ are equal to $(\psi_{m-1,z}^i,\chi_{m-1}(z))$ and $(\psi_{m-1,f(z)}^i,\chi_{m-1}(f(z)))$, respectively.
Furthermore $i$ is determined by $(\phi_{m-1}^N)^{-1}(4)\cap D_{m-1}(z,r_m^\pm -1)=\mathcal{N}_{m,x}^i$ if $m>0$, or by $(\phi_{-1}^N)^{-1}(0)\cap C_{-1}(z,10,r_0^\pm -1)=\mathcal{N}_{0,x}^i$ if $m=0$.
By~\ref{i: h preserves induction}, 
\[
f((\phi_{m-1}^N)^{-1}(4)\cap D_{m-1}(z,r_m^\pm -1))= (\phi_{m-1}^N)^{-1}(4)\cap D_{m-1}(f(z),r_m^\pm -1)
\]
if $m>0$, and
\[
f((\phi_{-1}^N)^{-1}(0)\cap C_{-1}(z,10,r_0^\pm -1)) = (\phi_{-1}^N)^{-1}(0)\cap C_{-1}(f(z),10, r_0^\pm -1)
\]
if $m=0$. Since $\chi_{m-1}(z)=\chi_{m-1}(f(z))$ by~\ref{i: h preserves weak equivalence} and \Cref{d. weak equivalence}~\ref{i: weak equivalence cnm}, property~\ref{i: h preserves psi} follows from \Cref{p: phi_n x^i}~\ref{i: phi_n x^i hat h}.

Property~\ref{i: h preserves xxm} follows from~\ref{i: h preserves psi} and \Cref{r. psi xxn}.

Let us prove~\ref{i: h preserves cm<}.
Let $z\in D_{m-1}(x,R-\Upsilon_{m-1} - L_{m-1}(W_m+2R_m^+))$.
By~\ref{i: h preserves induction}, \Cref{p: xn}~\ref{i: xn dn dn-1} and \Cref{c: graph partial}, the restriction of $f$ to $D_{m-1}(x,R-\Upsilon_{m-1}-L_{m-1}(W_m+R_m^+))$ preserves $X_n^\pm$ and is an $R_m^+$-short scale isometry with respect to $E_{m-1}$.
Then  $z$ satisfies~\eqref{zpln xn} if and only if $f(z)$ does, and~\ref{i: h preserves cm<} follows.

To prove~\ref{i: h preserves overline}, let $z\in \xm\cap  D_{m-1}(x,R-\Upsilon_{m-1} - L_{m-1}(W_m+3R_m^+))$.
By \Cref{l: overlinecnx rnpm}, we have $\overline{C}_{m,m-1}(z)\subset D_{m-1}(z,R_n^+)$.
Using \Cref{p: xn}~\ref{i: xn dn dn-1} and the triangle inequality, we get 
\[
\overline{C}_{m,m-1}(z)\subset D_{m-1}(x,R-\Upsilon_{m-1} - L_{m-1}(W_m+2R_m^+))\;.
\]
Therefore, for all $u\in \overline{C}_{m,m-1}(z)$, we have $u\in Z_{m-1}^\pm$ if and only if $f(u)\in Z_{m-1}^\pm$ by~\ref{i: h preserves cm<}.
Let $y\in X_m$ such that $d_{m-1}(u,X_m)= d_{m-1}(u,y)$.
Thus $d_{m-1}(u,y)\leq R_m^+$ by \Cref{p: xn}~\ref{i: xn separated net}, yielding $d_{-1}(u,y)\leq L_{m-1}R_m^+$ by \Cref{p: xn}~\ref{i: xn dn dn-1}.
By~\ref{i: h preserves induction},~\ref{i: h preserves xn} and \Cref{c: graph partial}, we get $f(y)\in X_m^\pm$ if and only if $y\in X_m^\pm$ and $d_{m-1}(u,y)= d_{m-1}(f(u),f(y))$.
Then~\ref{i: h preserves overline} follows by~\eqref{defn overlinecnn-1x}.

Let us prove~\ref{i: h preserves cluster}.
By \Cref{p: xn}~\ref{i: xn dn dn-1} and the triangle inequality, we get
\[
D_{m}(z,1)\subset D_{-1}(z,L_m)\subset D_{-1}(x,R-\Upsilon_{m-1} - L_{m-1}(W_m+3R_m^+)))\;.
\]
Therefore $f(\overline C_{n,n-1}(u))= \overline C_{n,n-1}(f(u))$ for all $u\in D_{m}(z,1)$ by~\ref{i: h preserves overline}.
Moreover $\phi^N_m(u)= \phi^N_m(f(u))$ for all $u\in D_{m}(z,1)$ by~\ref{i: h preserves psi}.
In particular, this yields $\chi_m(u)= \chi_m(f(u))$.
Then the result follows from \Cref{p: chin}~\ref{i: chin dn-1} and~\eqref{cnyn overlinecn}.

Property~\ref{i: h preserves em} follows easily from~\ref{i: h preserves overline}, \Cref{c: graph partial} and the definition of $E_m$.

Let us prove~\ref{i: h preserves equivalence}. By~\eqref{rmn hmy} and the triangle inequatity we have 
\[\mathcal C_{n,-1}(v)\subset D_{-1}(x,\Gamma^+_n+ nL_n)\subset D_{-1}(x,R-\Upsilon_{m-1} -L_{m-1}(W_m+3R_m^+ +1)-\Gamma_m^+)\;.\]
\Cref{d. equivalence}~\ref{i: equivalence ball} and~\ref{i: equivalence v} follow from~\ref{i: h preserves induction},~\ref{i: equivalence cnm} from~\ref{i: h preserves induction} and~\ref{i: h preserves xn},~\ref{i: equivalence xxm} from~\ref{i: h preserves xxm}, and~\ref{i: equivalence cmm-1} from \Cref{c. equivalence} and the induction hypothesis.
\end{proof}

We are now in position to complete the proof of \Cref{t: finitary}. Consider the increasing sequence of positive integers $\epsilon_n$ of the statement of \Cref{t: finitary}, used in \Cref{s. constants}.
Let $\delta_n=4\Gamma_n^+ + \Upsilon_n+ 2L_n $.

\begin{prop}\label{p: definitive}
For $0\leq n \leq N$ and $u\in X$, and using restrictions of $E_{-1}$, let 
\[
f\colon \left(D_{-1}\left(u,\delta_n \right),u,\phi_{-1}^N\right) \to \left(D_{-1}\left(f(u),\delta_n\right), f(u),\phi_{-1}^N\right)
\]
be a color-preserving pointed graph isomorphism. Then, either $f(u)=u$, or $d_{-1}(u,f(u))>\epsilon_n$.
\end{prop}

\begin{proof}
Let $x\in X_n^\pm$ such that $u\in C_{n,-1}(x)$.
We have $d_{-1}(u,x)\leq \Gamma_n^+$ by \Cref{l: cn-1 gamma}, and $D_{-1}(x,3\Gamma_n^+ + \Upsilon_n+ 2L_n) \subset \dom f$ by the triangle inequality.
By \Cref{l: h preserves}~\ref{i: h preserves xn},\ref{i: h preserves psi}, we obtain  $f(x)\in X_n^\pm$ and $\phi_n^N(x)=\phi_n^N(f(x))$.
In particular, $\chi_n(x)=\chi_n(f(x))$.
Therefore, either $f(x)=x$, or $d_{n-1}(x,f(x))\geq r_n^\pm s_n$ by \Cref{p: chin}~\ref{i: chin dn-1}.

If $f(x)=x$, then $f(u)=u$ by \Cref{p: equivalence identity} and the result follows.
So suppose $d_{n-1}(x,f(x))\geq 2r_n^\pm s_n$.
By \Cref{l: cn-1 gamma}, $d_{-1}(u,x)=d_{-1}(f(u),f(x))\leq \Gamma_n^{\pm}$.
Then, by the triangle inequality, $d(u,f(u))\geq r_n^{\pm}s_n - 2 \Gamma_n^{\pm}$. Applying now \Cref{l: rnpm gamma}, we get $d(u,f(u))\geq \epsilon_n$.
\end{proof}

This completes the proof of \Cref{t: finitary}~\ref{i: finitary limit aperiodic} by  taking $\phi^N=\phi_{-1}^N$.

\begin{prop}\label{p: definitive rep}
For $-1\leq m < n\le N$, $x\in \xxn$, and $u\in C_{n,m}(p)$, we have $\phi_{m}^N(u)=\phi_{m}^N(\mathfrak{h}_{n,x}(u))$.
\end{prop}

\begin{proof}
We proceed by inverse induction on $m$.
For $m=N$, we have $\phi_{N}^N=(\chi_{N},0)$.
So $\phi_{N}^N(u)=\phi_{N}^N(\mathfrak{h}_{n,x}(u))$ by \Cref{p: chin}~\ref{i: chin h}.

Suppose that, for $0\leq m <N-1$, the result is true for $m+1$.
Let  $u\in C_{n,m}(p)$ and $z\in C_{n,m+1}(p)$ such that $u\in C_{m+1,m}(z)$.
By the induction hypothesis, $\phi_{m+1}^N(z)= \phi_{m+1}^N(\mathfrak{h}_{n,x}(z))$.
By the definition of $\phi_{m+1}^N$, \Cref{l: rmn hmy zero,l: rmn hmy}, and \Cref{c. equivalence}, this means that the restrictions of $\phi_{m+1}^N$ to $C_{m+1,m}(z)$ and $C_{m+1,m}(\mathfrak{h}_{n,x}(z))$ equal $\psi_{m,x}^{i,j}$ and $\psi_{m,\mathfrak{h}_{n,x}(z)}^{i,j}$ for some $(i,j)\in I_{m,x}^2\subset H_{m,x}$ (see \Cref{r. notation}).
But  $\psi_{m,\mathfrak{h}_{n,x}(z)}^{i,j}=\mathfrak{h}_{n,x}(\psi_{m,x}^{i,j})$ by \Cref{p: phi_n x^i}~\ref{i: phinx0 h}.
\end{proof}

\Cref{p: definitive rep,p: xn}~\ref{i: xn subset}, together with \Cref{c. cn-1 contained}, yield   $\mathfrak{X}_n\subset \widehat \Omega_n$ for $n\leq N$ by taking $\phi^N=\phi_{-1}^N$, with the set $\widehat \Omega_n$ defined in \Cref{t: finitary}~\ref{i: finitary repetitive}. Then \Cref{t: finitary}~\ref{i: finitary repetitive} follows from \Cref{p: mathfrak X_n is a net,p: xn}~\ref{i: xn subset} taking $\alpha_n=2\ssn +\ttn + 3\omega_n$.

\bibliographystyle{amsplain}


\begin{thebibliography}{10}

\bibitem{AlbertsonCollins1996}
M.O. Albertson and K.L. Collins, \emph{Symmetry breaking in graphs}, Electron.
  J. Combin. \textbf{3} (1996), no.~1, 17 pp. \MR{1394549}

\bibitem{AlvarezBarral-realization}
J.A. \'Alvarez~L\'opez and R.~Barral~Lij\'o, \emph{Realization of manifolds as
  leaves using graph colorings}, arXiv:2002.08662, 2020.

\bibitem{AlvarezCandel2011}
J.A. \'Alvarez~L\'opez and A.~Candel, \emph{Algebraic characterization of
  quasi-isometric spaces via the {H}igson compactification}, Topology Appl.
  \textbf{158} (2011), no.~13, 1679--1694. \MR{2812477}

\bibitem{AlvarezCandel2018}
\bysame, \emph{Generic coarse geometry of leaves}, Lecture Notes in
  Mathematics, vol. 2223, Springer, Heidelberg-New York, 2018. \MR{3822768}

\bibitem{AubrunBarbieriThomasse2019}
N.~Aubrun, S.~Barbieri, and S.~Thomass\'e, \emph{Realization of aperiodic
  subshifts and uniform densities in groups}, Group. Geom. Dyn. \textbf{13}
  (2019), no.~1, 107--129. \MR{3900766}

\bibitem{BellissardBenedettiGambaudo2006}
J.~Bellissard, R.~Benedetti, and J.-M. Gambaudo, \emph{Spaces of tilings,
  finite telescopic approximations and gap-labeling}, Comm. Math. Phys.
  \textbf{261} (2006), no.~1, 1--41. \MR{2193205}

\bibitem{BlockWeinberger1992}
J.~Block and S.~Weinberger, \emph{Aperiodic tilings, positive scalar curvature
  and amenability of spaces}, J. Amer. Math. Soc. \textbf{5} (1992), no.~4,
  907--918. \MR{1145337}

\bibitem{BroerePilsniak2015}
I.~Broere and M.~Pilsniak, \emph{The distinguishing index of infinite graphs},
  Electron. J. Combin. \textbf{22} (2015), no.~1, 10~pp. \MR{3336592}

\bibitem{CollinsTrenk2006}
K.L. Collins and A.N. Trenk, \emph{The distinguishing chromatic number},
  Electron. J. Combin. \textbf{13} (2006), no.~1, Research Paper~16, 19.
  \MR{2200544}

\bibitem{DranishnikovSchroeder2007}
A.~Dranishnikov and V.~Schroeder, \emph{Aperiodic colorings and tilings of
  {C}oxeter groups}, Groups Geom. Dyn. \textbf{1} (2007), no.~3, 311--328.
  \MR{2314048}

\bibitem{GaoJacksonSeward2009}
S.~Gao, S.~Jackson, and B.~Seward, \emph{A coloring property for countable
  groups}, Math. Proc. Cambridge Philos. Soc. \textbf{147} (2009), no.~3,
  579--592. \MR{2557144}

\bibitem{Gromov1981}
M.~Gromov, \emph{Groups of polynomial growth and expanding maps. {Appendix by
  Jacques Tits}}, Inst. Hautes \'Etudes Sci. Publ. Math. \textbf{53} (1981),
  53--73. \MR{623534}

\bibitem{Gromov1999}
\bysame, \emph{Metric structures for {R}iemannian and non-{R}iemannian spaces},
  Progress in Mathematics, vol. 152, Birkh{\"a}user Boston Inc., Boston, MA,
  1999, Based on the 1981 French original [MR0682063], With appendices by
  M.~Katz, P.~Pansu and S.~Semmes, Translated from the French by Sean Michael
  Bates. \MR{1699320}

\bibitem{Huningetal}
S.~H\"{u}ning, W.~Imrich, J.~Kloas, H.~Schreiber, and T.~Tucker,
  \emph{Distinguishing graphs of maximum valence 3}, arXiv:1709.05797, 2017.

\bibitem{ImmelWenger2017}
P.~Immel and P.S. Wenger, \emph{The list distinguishing number equals the
  distinguishing number for interval graphs}, Discuss. Math. Graph Theory
  \textbf{37} (2017), no.~1, 165--174. \MR{3601040}

\bibitem{ImrichJerebicKlavzar2008}
W.~Imrich, J.~Jerebic, and S.~Klav\v{z}ar, \emph{The distinguishing number of
  {C}artesian products of complete graphs}, European J. Combin. \textbf{29}
  (2008), no.~4, 922--929. \MR{2408368}

\bibitem{KlavzarWongZhu2006}
S.~Klav\v{z}ar, T.~Wong, and X.~Zhu, \emph{Distinguishing labelings of group
  action on vector spaces and graphs}, J. Algebra \textbf{303} (2006), no.~2,
  626--641. \MR{2255126}

\bibitem{LehnerPilsniakStawiski}
F.~Lehner, M.~Pil\'{s}niak, and M.~Stawiski, \emph{Distinguishing infinite
  graphs with bounded degrees}, arXiv:1810.03932, 2018.

\bibitem{NesetrilOssona2008}
J.~Ne\v{s}et\v{r}il and P.~Ossona~de Mendez, \emph{From sparse graphs to
  nowhere dense structures: decompositions, independence, dualities and
  limits}, European {C}ongress of {M}athematics. {Proceedings of the 5th ECM
  congress, Amsterdam, Netherlands, July 14--18, 2008} (Z\"urich), Eur. Math.
  Soc., 2008, pp.~135--165. \MR{2648324}

\bibitem{Sadun2003}
L.~Sadun, \emph{Tiling spaces are inverse limits}, J. Math. Phys. \textbf{44}
  (2003), no.~11, 5410--5414. \MR{2014868}

\end{thebibliography}


\providecommand{\bysame}{\leavevmode\hbox to3em{\hrulefill}\thinspace}
\providecommand{\MR}{\relax\ifhmode\unskip\space\fi MR }
\providecommand{\MRhref}[2]{%
  \href{http://www.ams.org/mathscinet-getitem?mr=#1}{#2}
}
\providecommand{\href}[2]{#2}

\end{document}